%% file: curvesv3.tex
\documentclass[11pt]{amsart}
\usepackage[margin=2.5cm]{geometry}               
\usepackage[usenames,dvipsnames,svgnames,table]{xcolor}    
\usepackage[all]{xy}
\usepackage{graphicx}
\usepackage{tikz-cd}
\usetikzlibrary{decorations.pathreplacing}
\usepackage{lineno,hyperref}
\usepackage{amssymb,amsfonts, amsmath, amsthm}
\usepackage{epstopdf,enumerate}
\usepackage{ulem}
\usepackage{cancel}
\usepackage{thmtools}
\usepackage{mathrsfs}
\input{commands.tex}

\title{Definable sets of Berkovich curves}

\author{Pablo Cubides Kovacsics }

\author{J\'er\^ome Poineau}

\address{Pablo Cubides Kovacsics, TU Dresden, Fachrichtung Mathematik, Institut f\"ur Algebra, 01062 Dresden, Zellescher Weg 12--14, 
Willersbau Zi. C 114 }
\email{pablo.cubides\_\,kovacsics@tu-dresden.de}

\address{J\'er\^ome Poineau, Universit\'e de Caen, Laboratoire de math\'ematiques Nicolas Oresme, CNRS UMR 6139,
14032 Caen cedex, France }
\email{jerome.poineau@unicaen.fr}

\thanks{The authors were supported by the ERC project TOSSIBERG (grant agreement 637027).}

\begin{document}

\begin{abstract}
In this article, we functorially associate definable sets to $k$-analytic curves, and definable maps to analytic morphisms between them, for a large class of $k$-analytic curves. Given a $k$-analytic curve $X$, our association allows us to have definable versions of several usual notions of Berkovich analytic geometry such as the branch emanating from a point and the residue curve at a point of type 2. We also characterize the definable subsets of the definable counterpart of $X$ and show that they satisfy a bijective relation with the radial subsets of $X$. As an application, we recover (and slightly extend) results of Temkin concerning the radiality of the set of points with a given prescribed multiplicity with respect to a morphism of $k$-analytic curves.  

In the case of the analytification of an algebraic curve, our construction can also be seen as an explicit version of Hrushovski and Loeser's theorem on iso-definability of curves. However, our approach can also be applied to strictly $k$-affinoid curves and arbitrary morphisms between them, which are currently not in the scope of their setting.   
\end{abstract}

\maketitle
\setcounter{tocdepth}{1}
\tableofcontents

\newcounter{eqn}
\normalem


\section{Introduction}

Let $(k,\va)$ be a complete rank 1 non-archimedean algebraically closed non-trivially valued field. In this article we further develop the interplay between Berkovich spaces and the model theory of algebraically closed valued fields by functorially associating a definable set to a $k$-analytic curve in a large class, including analytifications of algebraic curves as well as strictly $k$-affinoid curves. Our approach is direct and geometric: it is based on the local structure of Berkovich curves (or equivalently the semistable reduction theorem for curves), which enables us, in particular, to find definable counterparts of several usual notions in Berkovich analytic geometry such as the branch emanating from a point, the residue curve at a point of type~2, etc. In addition, the concrete nature of our construction also allows us to provide an explicit description of the definable subsets of our model-theoretic version of Berkovich curves.

\medbreak

Our results are deeply inspired by the foundational work of Ehud Hrushovski and Fran\c{c}ois Loeser~\cite{HL}. Let us recall that, given an algebraic variety~$X$, they introduced a model-theoretic avatar of its Berkovich analytication~$X^\an$, denoted by $\widehat{X}$ and called the \emph{stable completion of $X$}. This model-theoretic setting allows them to deduce, among others, striking results about the homotopy type of~$X^\an$, under quasi-projectivity assumptions but removing assumptions from~\cite{Berkovich_contractible} such as smoothness, compactness and the existence of a polystable model. 

A crucial property of the space $\widehat{X}$ is that it carries a strict pro-definable structure, that is $\widehat{X}$ is a projective limit of definable sets with definable surjective transition maps. When $X$ is a curve, Hrushovski and Loeser prove that~$\widehat{X}$ is in fact \emph{iso-definable} (\cite[Theorem 7.1.1]{HL}), namely in pro-definable bijection with a definable set. However, their proof is not constructive and one cannot explicitly extract from their arguments a particular definable set to identify~$\widehat{X}$ with. As a result, in the case of the analytification of an algebraic curve, our construction can also be seen as an explicit version of Hrushovski and Loeser's result on iso-definability of curves. Note that we show (Theorem \ref{thm:prodef_facade}) that the definable set we associate to $X^\an$ is in pro-definable bijection with $\widehat{X}$. Although restricted to one-dimensional spaces, an interesting advantage of our approach is that it can handle some non-algebraic curves (and non-algebraic morphisms between them), which currently lie beyond the scope of~\cite{HL}. In addition, our methods make no use of elimination of imaginaries in algebraically closed valued fields, which plays an important role in \cite{HL}. 

In the setting of Berkovich curves, Michael Temkin introduced a notion of ``radial set'' in his work about ramification of finite morphisms of $k$-analytic curves~\cite{temkin_2017_metric}. The explicit nature of our construction allows us to characterize the definable subsets of the definable set associated to a $k$-analytic curve $X$ (Theorem \ref{thm:definablecurves}) and to deduce that they are in canonical bijection with radial subsets of $X$ (Theorem \ref{thm:radialdef}). As an application, we are able to recover (and slightly extend), \textit{via} model-theoretic methods, one of the main results of~\cite{temkin_2017_metric}: given a flat morphism $h\colon X\to Y$ of strictly $k$-analytic curves of relative dimension~0, the set $N_{h,d}:=\{x\in X :  \deg_x(h)=d\}$ is radial (Theorem~\ref{thm:radiality_curves}).

It is worthy to note that the relation between definability and radiality is not new. Using results from \cite{HL}, John Welliaveetil has recently studied it in \cite{W}. His results are somehow complementary to ours. On the one hand, he only works with definable sets which are definably path-connected, a restriction which is not present in our approach. On the other hand, some of his results hold in families, a step which, even if it might work in our setting, has not been developed in this article. 

\medbreak

The article is laid out as follows. In Section \ref{sec:prelim} we fix the notation and provide the needed background both on Berkovich spaces and on the model theory of algebraically closed valued fields. Section \ref{sec:bricks} provides a local-global analysis of $k$-analytic curves which will constitute the core of our construction. The definable set associated to a Berkovich curve is introduced in Section \ref{sec:functorial}, in which we also prove the functoriality of our construction. Definable subsets of curves are described in Sections \ref{sec:definable} and \ref{sec:def_of_curves}. Radiality and definability are discussed in Section~\ref{sec:radiality}. Finally, the comparison with Hrushovski-Loeser spaces is presented in Section \ref{sec:HL}.

\subsection*{Further directions}

There are at least two interesting topics for further research in the direction of this article. The first one is to extend the construction to fields of higher rank in the spirit of~\cite{FR}. The second is to show that the construction can be made uniform in families.

\subsection*{Acknowledgements}

We are grateful to the referee for a careful thorough reading of our paper and for many useful suggestions. In particular, we would like to thank him/her 
for suggesting us the content of Remark \ref{rem:deformation}.

\section{Preliminaries}\label{sec:prelim}

Throughout this article, $(k,\va)$ will be a complete rank 1 non-archimedean algebraically closed non-trivially valued field. We denote by~$k^\circ$ its valuation ring, by~$k^{\circ\circ}$ its maximal ideal and by~$\tilde{k}$ its residue field. We denote by $|k|$ and $|k^\times|$ the images of~$k$ and~$k^\times$ respectively by the absolute value~$\va$.

We set $\overline{\R}_{\ge 0} := \R_{\ge0} \cup\{+\infty\}$ and $\overline\R_{> 0} := \R_{>0} \cup\{+\infty\}$. For~$a\in k$ and~$r\in \R_{\ge 0}$, we let 
\[
D_{k}(a,r):=\{x\in k :   |x-a|\leqslant r\} 
\] 
denote the closed disc centered at~$a$ of radius~$r$, and for~$r\in \overline\R_{>0}$, we let 
\[
D_{k}^-(a,r):=\{x\in k :   |x-a|< r\} 
\] 
denote the open disc centered at~$a$ of radius~$r$. We will often remove~$k$ from the notation when no confusion arises. Note that points of~$k$ are closed discs of radius~$0$ and that $k$ is an open disc of radius~$+\infty$. 

A \emph{Swiss cheese} is a set of the form~$B\setminus (\bigcup_{i=1}^m B_i)$ where~$B$ is either~$k$ or a disc (open or closed), each~$B_i$ is a disc (open or closed) properly contained in~$B$ such that~$B_i\cap B_j=\emptyset$ for~$i\neq j$ and all discs have radius in~$|k|$.

\subsection{Berkovich spaces}

Recall that the Berkovich affine line~$\AA_k^{1,\an}$ is the space of multiplicative seminorms on the polynomial ring in one variable~$k[T]$ which extend the norm on $k$. Given~$x\in \AA_k^{1,\an}$, $\sH(x)$ denotes the completion of the fraction field of~$k[T]/\mathrm{ker}(x)$, and is called the \emph{completed residue field at $x$}. The valued field extension~$\sH(x)|k$ determines the four possible types of points of the line (see~\cite[Section~1.4.4]{rouge} for more details). Note that this classification can also be applied to any given analytic curve. In the case of the Berkovich affine line, we will use the~$\eta$ notation to denote points of type 1, 2 and 3 which we briefly recall. For~$a\in k$ and~$r\in \RR_{\geqslant 0}$, we let~$\eta_{a,r}$ denote the seminorm defined by 
\[
\sum_{i=1}^n a_i(T-a)^i \mapsto \max\{|a_i|r^i :   1\leqslant i\leqslant n\}. 
\]
The point~$\eta_{a,r}$ is of type 1 if~$r=0$, of type 2 if~$r\in |k^\times|$ and of type 3 if~$r\in \RR_{>0}\setminus |k^\times|$. Since $k$~is algebraically closed, all points of type 1,2 or 3 are of the form~$\eta_{a,r}$ for some~$a\in k$ and~$r\in \RR_{\geqslant 0}$. Points of type 4 are all remaining points. Note that if~$k$ is maximally complete, no point of type 4 exists.  

We will fix from now on a point at infinity~$\infty$ in~$\PP^1(k)$ and a coordinate~$T$ on $\AA^1_{k} = \PP^1_{k} \setminus \{\infty\}$.

\medbreak

For $a\in k$ and $r \in \R_{\ge 0}$, we denote the closed Berkovich disc with center~$a$ and radius~$r$ by 
\[\DD_{k}(a,r) := \{x\in \AA^{1,\an}_{k}  :  |(T-a)(x)|\le r\}\] 
and, for $a\in k$ and $r \in \overline\R_{> 0}$, we denote the open Berkovich disc with center~$a$ and radius~$r$ by 
\[\DD_{k}^-(a,r) := \{x\in \AA^{1,\an}_{k}  :  |(T-a)(x)|< r\}.\]
In particular, the affine line is an open disc of infinite radius. We set $\DD_{k} := \DD_{k}(0,1)$ and  $\DD^-_{k} := \DD_{k}^-(0,1)$. We will often remove~$k$ from the notation when no confusion arises. 

\medbreak

\begin{definition}\label{def:nice} 
A \emph{$k$-analytic curve} is a purely one-dimensional separated reduced $k$-analytic space. A \emph{nice curve} is a curve that is isomorphic to the complement of finitely many $k$-rational points in a compact strictly $k$-analytic curve.
\end{definition}

Let~$X$ be a $k$-analytic curve. For~$i\in\{1,2,3,4\}$, we let~$X^{(i)}$ denote the set of points of type~$i$ in~$X$. For~$i,j\in\{1,2,3,4\}$, we let~$X^{(i,j)}$ denote~$X^{(i)}\cup X^{(j)}$, etc. For example, we have
\[\DD_{k}^\Def = \{\eta_{a,r} \in \AA_k^{1,\an}  :  a\in k^\circ, r \in |k^\times| \cap [0,1]\}\]
and
\[(\DD_{k}^-)^\Def = \{\eta_{a,r} \in \AA_k^{1,\an}  :  a\in k^{\circ\circ}, r \in |k^\times| \cap [0,1)\}.\]
To simplify notation, we will write $\AA_k^{1,(i)}$ instead of $\AA_k^{1,\an,(i)}$, etc. 

In what follows, we will often identify the set~$X^{(1)}$ of points of type~1 with the set~$X(k)$ of $k$-rational points.

\subsubsection{Residue curves and branches}\label{sec:residuecurve}

Let~$X$ be a $k$-analytic curve. Let $x\in X$ be a point of type~2. The residue field~$\widetilde{\sH(x)}$ of~$\sH(x)$ has transcendence degree~1 over~$\tilde k$, hence it is the function field of a well-defined connected smooth projective curve over~$\tilde k$. We call the latter the \emph{residue curve at~$x$} and denote it by~$\cC_{x}$. 

The set of branches emanating from~$x$ is defined as 
\[\cB_{x} := \varprojlim_{U} \pi_{0}(U\setminus\{x\}),\] 
where~$U$ runs through the neighborhoods of~$x$ in~$X$. If~$x$ belongs to the interior of~$X$, there is a natural bijection between the set of branches and the set of closed points of the residue curve (see~\cite[4.2.11.1]{Duc-book}). We denote it by $\beta_{x} : \cB_{x} \simto \cC_{x}(\tilde k)$.

Let $x\in X^{(2)}$. A connected affinoid domain~$V$ of~$X$ is said to be a \emph{tube centered at~$x$} if it contains~$x$ and if each connected component of~$V\setminus \{x\}$ is an open disc. Let~$V$ be a tube centered at~$x$. Then, to each connected component~$E$ of~$V\setminus \{x\}$, one may associate a unique branch~$\beta_{E}$ emanating from~$x$. We denote by
\[\rho_{V}\colon V \to \cC_{x}\]
the map that sends~$x$ to the generic point of~$\cC_{x}$ and whose restriction to~$E$ is the constant map with value~$\beta_{x}(b_{E})$, for each connected component~$E$ of~$V\setminus\{x\}$. 

We denote by~$\cU_{V}$ the image of~$\rho_{V}$. It is a Zariski-open subset of~$\cC_{x}$. Note that~$V$ cannot contain all the branches emanating from~$x$ (otherwise it would be boundary-free), hence~$\cU_{V}$ is a proper subset of~$\cC_{x}$. In particular, it is an affine curve over~$\tilde k$.

Let us make this construction more explicit in the case where $X=\PP^{1,\an}_{k}$ and $x=\eta_{0,1}$. Then the closed unit disc $V = \DD_{k}$ is a tube centered at~$\eta_{0,1}$. Denoting by~$t$ the image of~$T$ in $\widetilde{\sH(\eta_{0,1})}$, we have $\widetilde{\sH(\eta_{0,1})} = \tilde k(t)$, hence $\cC_{\eta_{0,1}} = \PP^1_{\tilde k}$, and the natural map $\cB_{\eta_{0,1}} \to \pi_{0}(\PP^{1,\an}_{k}\setminus\{\eta_{0,1}\})$ is a bijection. It is not difficult to check that, in this case, the map~$\rho_{V}$ coincides with the reduction map $\red \colon \DD_{k} \to \AA^1_{\tilde{k}}$ from~\cite[\S~2.4]{rouge}. In particular, we have $\cU_{V} = \AA^1_{\tilde{k}}$.

Let $f \colon X \to Y$ be a morphism from~$X$ to a $k$-analytic curve~$Y$ that is finite at~$x$. Set $y:=f(x)$. We have an induced finite morphism $\tilde f_{x} \colon \cC_{x} \to \cC_{y}$. Moreover, by \cite[Th\'eor\`eme~4.3.13]{Duc-book} for each $a \in \cC_{x}(\tilde k)$, the ramification index~$e_{a}$ of~$\tilde f_{x}$ at~$a$ coincides with the degree of~$f$ on the corresponding branch (\emph{i.e.} the degree of the restriction of~$f$ to the connected component of $U\setminus\{x\}$ for every small enough neighborhood~$U$ of~$x$ in~$X$).

\subsubsection{Triangulations}

In this subsection we recall different concepts from~\cite{Duc-book} and~\cite{temkinETAL_2016_Morphisms} 
that will be needed in the following sections. For convenience, we slightly change some of the definitions and notation used in those references. Nonetheless, it will be very easy to switch between our setting and theirs.

From now on, by an interval $I$, we mean a topological space homeomorphic to a non-empty interval of~$\R$ (closed, open, or semi-open). Graphs will be denoted by the letter $\Gamma$. The set of vertices of a graph~$\Gamma$ will be denoted by $S$ and its set of edges by $E$, each edge being isomorphic to an open interval. We do not require edges to have endpoints. In particular, we allow the graph with one edge and no vertices. The arity of a vertex $x\in S$ is the number of edges to which $x$ is attached. 

\begin{definition}[Triangulation]
Let~$X$ be a strictly $k$-analytic curve. A \emph{triangulation} of~$X$ is a locally finite subset~$S$ of~$X^{(1,2)}$ such that
\begin{enumerate}[$i)$]
\item $S$ meets every connected component of~$X$;
\item $X \setminus S$ is a disjoint union of open discs and open annuli (possibly punctured open discs).
\end{enumerate}

We consider the affine line (resp. the punctured affine line) to be an open disc (resp. an open annulus).
\end{definition}

Remark that any triangulation necessarily contains the singular points of~$X$, its boundary points and its points  of type~2 where the residue curve has positive genus.

Let~$X$ be a strictly $k$-analytic curve endowed with a triangulation~$S$. The set~$\cA$ of connected components of $X\setminus S$ that are annuli is locally finite. Recall that the skeleton~$\Gamma_{A}$ of an open annulus~$A$ is defined as the set of points with no neighborhoods isomorphic to an open disc. It is homeomorphic to an open interval. We define the \emph{skeleton} of the triangulation~$S$ to be
\[\Gamma_{S} := S \cup \bigcup_{A\in \cA} \Gamma_{A}.\]
It is a locally finite graph. We define its set of vertices as~$S$ and its set of edges as $E_{S} := \{\Gamma_{A}  :  A\in \cA\}$ (in a loose sense, since some edges may have only one or no endpoints). Remark that $\Gamma_{S}$ contains no points of type~4 and that $X \setminus \Gamma_{S}$ is a disjoint union of open discs.

There exists a deformation retraction from the curve~$X$ to its skeleton~$\Gamma_{S}$. 

\begin{lemma}\label{lem:taudisc}
The map
\[\begin{array}{cccccl}
\tau_{\DD}^{(1,2,3)} \colon & [0,1] &\times& \DD_{k}^{(1,2,3)} & \to &  \DD_{k}^{(1,2,3)}\\
& (t&,&\eta_{a,r}) & \mapsto & \eta_{a,\max(r,t)} 
\end{array}\]
is well-defined and continuous. It extends by continuity to a map $\tau_{\DD} \colon [0,1] \times \DD_{k}  \to   \DD_{k}$ that is a deformation retraction of~$\DD_{k}$ onto~$\{\eta_{0,1}\}$.
\qed
\end{lemma}

Let $x\in S^{(2)}$. Denote by~$W_{S,x}$ the union of the connected components of $X\setminus\{x\}$ that are discs with boundary point~$x$. Assume that~$W_{S,x}$ misses at least one branch emanating from~$x$. Then, there exists a continuous injective map $f \colon W_{S,x} \to \DD_{k}$ sending~$x$ to~$\eta_{0,1}$. Since~$f$ is a homeomorphism onto its image, the map~$\tau_{\DD}$ induces a deformation retraction $\tau_{S,x}\colon [0,1] \times W_{S,x} \to W_{S,x}$ of~$W_{S,x}$ onto~$\{x\}$. Moreover, $\tau_{\DD}$ does not depend on~$f$. 

If~$W_{S,x}$ contains every branch emanating from~$x$, we can still prove the existence of the deformation retraction~$\tau_{S,x}$ by covering~$W_{S,x}$ by two smaller subsets of the kind we had before and applying the argument to them.

\begin{lemma}\label{lem:tauannulus}
Let $u<v \in \overline\R_{\ge 0}$ and consider the annulus $A := \DD^-_{k}(0,v) \setminus \DD_{k}(0,u)$. The map
\[\begin{array}{cccccl}
\tau_{A}^{(1,2,3)} \colon & [0,1] &\times& A^{(1,2,3)} & \to &  A^{(1,2,3)}\\
& (t &,&\eta_{a,r} \text{ with } r\le |a|) & \mapsto & \eta_{a,\max(r,t|a|)} 
\end{array}\]
is well-defined and continuous. It extends by continuity to a map $\tau_{A} \colon [0,1] \times A  \to  A$ that is a deformation retraction of~$A$ onto $\Gamma(A) = \{\eta_{0,r} \,:\, r\in (u,v)\}$.
\qed
\end{lemma}

Denote by~$\cA_{S}$ the set of connected components of~$X\setminus S$ that are annuli. Let us now consider the map 
\[\begin{array}{cccl}
\tau_{S} \colon & [0,1] \times X &\to& X\\
& (t,y) & \mapsto & 
\begin{cases} 
\tau_{S,x}(t,y) & \text{if } y\in W_{x} \text{ for some } x\in S^{(2)};\\
\tau_{A}(t,y) & \text{if } y\in A \text{ for some } A \in \cA_{S};\\
y & \text{if } y\in S^{(1)}.
\end{cases}
 \end{array}\]
 It is a deformation retraction of~$X$ onto~$\Gamma_{S}$.
 
\begin{convention}
In this article, all triangulations will be assumed to be finite except in Section~\ref{sec:radiality}. In particular, all skeleta will be finite graphs.
\end{convention}

\bigbreak

Let~$X$ be a nice curve. By definition, there exists a compact strictly $k$-analytic curve~$\ov X$ containing~$X$ such that $\ov X \setminus X$ is a finite set of $k$-rational points. Such a curve will be called a \emph{compactification} of~$X$. Let us fix such a compactification~$\ov X$. Note that~$X$ is dense in~$\overline X$.

We will denote by~$\widetilde X$ the normalisation of~$\overline X$ and by $n_{X} \colon \widetilde X \to \overline X$ the corresponding morphism (see \cite[Section~5.2]{Ducros_Excellence} for details). The curve~$\widetilde X$ is compact and quasi-smooth (which is the same as rig-smooth or geometrically regular, see  \cite[Chapter~5]{Ducros_Families} for a complete reference). We will denote by~$s(\ov X)$ the singular locus of~$\ov X$. Since~$\ov X$ is reduced and~$k$ is algebraically closed, $\ov X$ is generically smooth, hence $s(\ov X)$ is a finite set of $k$-rational points. Moreover, the morphism $\wt X \setminus n_{X}^{-1}(s(\ov X)) \to \ov X \setminus s(\ov X)$ induced by~$n_{X}$ is an isomorphism.

\begin{lemma}\label{lem:Swtov}
Let~$S$ be a triangulation of~$X$. Then $\overline S := S \cup (\overline X \setminus X)$ is a triangulation of~$\overline X$ with skeleton $\Gamma_{\overline S} = \Gamma_{S} \cup (\overline X \setminus X)$. 

Let~$S'$ be a triangulation of~$\overline X$ containing $\overline X \setminus X$. Then $S' \cap X$ is a triangulation of~$X$ with skeleton $\Gamma_{S' \cap X} = \Gamma_{S'}\cap X$.

Let~$S'$ be a triangulation of~$\ov X$. Then $n_{X}^{-1}(S')$ is a triangulation of~$\wt X$ and~$n_{X}$ induces an isomorphism $\Gamma_{n_{X}^{-1}(S')} \setminus n_{X}^{-1}(s(\ov X)) \simto \Gamma_{S'} \setminus s(\ov X)$.

Let $S''$ be a triangulation of~$\wt X$ containing $n_{X}^{-1}(s(\ov X))$. Then $n_{X}(S'')$ is a triangulation of~$\ov X$ and~$n_{X}$ induces an isomorphism $\Gamma_{S''} \setminus n_{X}^{-1}(s(\ov X))) \simto \Gamma_{n_{X}(S'')} \setminus s(\ov X)$.
\qed
\end{lemma}

\begin{theorem}\label{thm:semistable}
Let~$X$ be a nice curve. Let~$S_{0}$ be a finite subset of~$X^{(1,2)}$. Then, there exists a finite triangulation~$S$ of~$X$ containing~$S_{0}$.
\end{theorem}
\begin{proof}
Lemma~\ref{lem:Swtov} shows that we may replace~$X$ by~$\wt X$, hence assume that~$X$ is compact and quasi-smooth. In this case, the result follows from the semistable reduction theorem (see \cite[Theorem~4.3.1]{rouge} or \cite[Th\'eor\`eme~5.1.14]{Duc-book}). 
\end{proof}

\begin{definition}\label{def:triang-refinement} Let $X$ be a strictly $k$-analytic curve and let $S_1,S_2$ be two triangulations of~$X$. We say that \emph{$S_2$ refines $S_1$} if $S_1\subseteq S_2$.
\end{definition}

Note that if $S_2$ refines $S_1$, the associated skeleta $\Gamma_{S_{1}}$ and $\Gamma_{S_{2}}$ have the same homotopy type. 

\begin{lemma}\label{lem:ref-arity} Let $X$ be strictly $k$-analytic curve and $S_1,S_2 \subseteq X$ be two triangulations of $X$ such that $S_2$ refines $S_1$. If $S_1\neq S_2$, then there is $x\in S_2\setminus S_1$ of arity $\leqslant 2$. Moreover, for every such $x$, $S_2\setminus \{x\}$ is a triangulation that refines $S_1$. 
\end{lemma}

\begin{proof} We split in cases.

\textbf{Case 1:} Suppose there is $y\in S_2\cap (X\setminus \Gamma_{S_{1}})$. Then, the connected component $E$ of~$y$ in $X\setminus \Gamma_{S_{1}}$ is isomorphic to an open disc. Since $\Gamma_{S_{1}}$ and $\Gamma_{S_{2}}$ have the same homotopy type, there cannot be a loop in $\Gamma_{S_{2}}$ containing $y$. Therefore, $E\cap S_2$ must contain a point $x$ of arity $1$ in $\Gamma_{S_{2}}$.  

\textbf{Case 2:} Suppose $S_2\subseteq \Gamma_{S_{1}}$. For $x\in S_2\setminus S_1$, since $\Gamma_{S_{1}}\setminus S_1$ is a disjoint union of open intervals, $x$ is contained in only one of such intervals. Therefore its arity is 2.   

We leave to the reader the proof of the final assertion of the lemma. 
\end{proof}

\begin{definition}[Compatibility]
Let $f\colon X_1\to X_2$ be a morphism of strictly $k$-analytic curves. A pair of triangulations $(S_1,S_2)$, where $S_i$ is a triangulation of $X_i$, is said to be \emph{$f$-compatible} if we have $f^{-1}(S_{2}) = S_{1}$ and $f^{-1}(E_{S_{2}}) = E_{S_{1}}$.
\end{definition}

In the compact case, compatible triangulations always exist, as a consequence of Theorem~\ref{thm:semistable} (see the proof of the simultaneous semistable reduction theorem in~\cite[Section~3.5.11]{temkinETAL_2016_Morphisms} for more details). 

\begin{theorem}
Let $f\colon X_1\to X_2$ be a morphism of compact strictly $k$-analytic curves of relative dimension~0. Then, there exists an $f$-compatible pair of triangulations $(S_{1},S_{2})$. 

Moreover, we may assume that~$S_{1}$ (resp. $S_{2}$) contains any given finite subset of~$X_{1}^{(1,2)}$ (resp. $X_{2}^{(1,2)}$).
\qed
\end{theorem} 

We now prove a similar result under slightly less restrictive hypotheses.

\begin{definition}
A morphism $f\colon X_1\to X_2$ of nice curves is said to be \emph{compactifiable} if there exist compactifications~$\ov X_{1}$ and~$\ov X_{2}$ of~$X_{1}$ and~$X_{2}$ and a morphism $\ov f\colon  \ov X_1\to \ov X_2$ such that the diagram
\[\begin{tikzcd}
X_{1} \arrow[hook]{r}\arrow{d}{f}& \ov X_{1}\arrow{d}{\ov f}\\
X_{2} \arrow[hook]{r}& \ov X_{2}
\end{tikzcd}\]
commutes. 
\end{definition}

\begin{remark}\label{rem:compactifiable}
Let~$C_{1}$ be a reduced algebraic curve over~$k$. Then, it admits a (unique) compactification~$\ov C_{1}$ such that $\ov C_{1} \setminus C_{1}$ is a finite set of smooth $k$-rational points. In particular, the local ring at each point of $\ov C_{1} \setminus C_{1}$ is a discrete valuation ring. Let~$C_{2}$ be an algebraic curve over~$k$ and let~$\ov C_{2}$ be any compactification of it. Then the valuative criterion of properness ensures that any morphism $f \colon C_{1} \to C_{2}$ extends to a morphism $\ov f \colon \ov C_{1} \to \ov C_{2}$. In particular, the morphism $f^\an \colon C_{1}^\an \to C_{2}^\an$ is compactificable (with compactification $\ov{f}^\an$).
\end{remark}

\begin{remark}\label{rem:flat}
Let $f\colon X_1\to X_2$ be a compactifiable morphism of nice curves of relative dimension~0 and let $\ov f\colon  \ov X_1\to \ov X_2$ be a compactification of it. Then $\bar f$~is still of relative dimension~0, since curves and their compactifications only differ by finitely many $k$-rational points. 

By normalizing around the points of $\ov X_{2} \setminus X_{2}$, we get another compactification $\ov X_{2}'$ with the property that $\ov X_{2}' \setminus X_{2}$ contains only smooth $k$-points. Similarly, we can define another compactification $\ov X_{1}'$ of~$X_{1}$. The universal property of normalization ensures that~$\ov f$ extends to a morphism $\ov f'\colon  \ov X_1'\to \ov X_2'$. Since the points in $\ov X_{2}' \setminus X_{2}$ are smooth, the morphism~$\ov f'$ is flat above those points. In particular, if~$f$ is flat, then~$\ov f'$ is flat. 
\end{remark}

\begin{corollary}\label{cor:compact_compatible}

Let $f\colon X_1\to X_2$ be a compactifiable morphism of nice curves of relative dimension~0. Then, there exists an $f$-compatible pair of triangulations $(S_{1},S_{2})$. 

Moreover, we may assume that~$S_{1}$ (resp. $S_{2}$) contains any given finite subset of~$X_{1}^{(1,2)}$ (resp. $X_{2}^{(1,2)}$).
\qed
\end{corollary}

We end this section with two lemmas that will be useful later.

\begin{lemma}\label{lem:kpoints}
Let $a_{0},b_{0} \in k$ and $r_{0},s_{0} \in \overline\R_{>0}$. Let $h\colon \DD^-(a_{0},r_{0})\to \DD^-(b_{0},s_{0})$ be a non-constant morphism. Let $r\in (0,r_{0})$, $s\in (0,s_{0})$ and $a,b\in k$ with $|a-a_{0}|<r_{0}$ and $|b-b_{0}|<s_{0}$.

Then, the following conditions are equivalent:
\begin{enumerate}[$i)$]
\item $h(\eta_{a,r}) = \eta_{b,s}$;
\item $h(\DD(a,r)) = \DD(b,s)$;
\item $h(D(a,r)) = D(b,s)$.
\end{enumerate}
In particular, the image of~$\DD(a,r)$ is a closed disc.
\end{lemma}
\begin{proof}
$i) \implies ii)$ Recall that we can define a partial order on the points of a Berkovich disc~$\DD^-(\alpha,\rho)$ by setting $x\le y$ when we have $|f(x)|\le |f(y)|$ for each $f\in \cO(\DD^-(\alpha,\rho))$. Remark that the order is preserved by morphisms of discs, since functions pull back. In what follows, we will consider this order on the discs $\DD^-(a_{0},r_{0})$ and $\DD^-(b_{0},s_{0})$.

Let $x\in \DD(a,r)$. We then have $x\le \eta_{a,r}$, hence $h(x) \le h(\eta_{a,r}) = \eta_{b,s}$, and it follows that $h(x) \in \DD(b,s)$. We have proven that $h(\DD(a,r)) \subseteq \DD(b,s)$.

Let us now prove the converse inclusion. Up to changing coordinates, we may assume that $a_{0}=a=0$ and $b_{0}=b=0$. The morphism~$h$ is then given by a power series $H \in k[\![ T]\!]$ of radius of convergence at least~$r_{0}$. Set $c = H(0)$. We have $|c|<s_{0}$.

Recall that the Newton polygon of a power series $F = \sum_{n\ge 0} f_{n}\, T^n \in k[\![ T]\!]$ of radius of convergence at least~$R$ is defined as the lower convex hull of the set of points $\{(n, -\log(|f_{n}|))  :  n\ge 0\}$ and that, for any $\rho \in (0,R)$, $F$ has a root of absolute value~$\rho$ if, and only if, the Newton polygon has an edge of slope~$\log(\rho)$. 

Set $K := \sH(\eta_{0,s})$. Let $z \in \DD(0,s)(K)$ whose image in~$\DD(0,s)$ is~$\eta_{0,s}$. By assumption, the power series $H - z \in  K[\![ T]\!]$ has a zero of absolute value~$r$ (over~$\eta_{0,r}$) in some extension of~$K$, hence the corresponding Newton polygon has an edge of slope~$\log(r)$.

Let $y\in \DD(0,s)$. Let~$L$ be a complete valued extension of~$K$ such that $y\in \DD(0,s)(L)$. We want to prove that $y \in h(\DD(0,r))$. If $H(0) =y$, then we are done, so we assume that this is not the case. It is enough to prove that the power series $H - y \in L[\![ T]\!]$ has a zero (in some extension of~$L$) with absolute value less than or equal to~$r$, or equivalently that the Newton polygon of~$H-y$ has an edge with slope less than or equal to~$\log(r)$. Since the Newton polygons of~$H-y$ and~$H-z$ have the same endpoint, it is enough to prove that the former lies above the latter.

The coefficients of the power series~$H-y$ and~$H-z$ are the same except for the constant ones, which are~$c-y$ and~$c-z$ respectively. To prove the result, it is enough to show that $|c-y| \le |c-z|$. If $|c|>s$, then we have $|c-y|=|c-z|=|c|$, and we are done. If $|c|\le s$, then we have $|c-y| \le s$ and $|c-z| = s$, because~$z$ lies over~$\eta_{0,s}$, so we are done too.

\smallbreak

Note that the final part of the statement follows from this implication. Indeed, since~$h$ is quasi-finite, the image of the point~$\eta_{a,r}$ is of type~2 or~3, hence of the form~$\eta_{b',s'}$, so we have $h(\DD(a,r)) = \DD(b',s')$.

\smallbreak

$ii) \implies iii)$ This is obvious.

\smallbreak

$iii) \implies i)$ There exist $s'\in (0,s_{0})$ and $b'\in k$ with $|b'-b_{0}|<s_{0}$ such that $h(\eta_{a,r}) = \eta_{b',s'}$. It follows from the former implications that we have $D(b',s') = h(D(a,r)) = D(b,s)$. Over a field whose valued group is dense in~$\R_{>0}$, 
a Berkovich disc is characterized by its rational points. We deduce that $\DD(b',s') = \DD(b,s)$ and $\eta_{b',s'} = \eta_{b,s}$.
\end{proof}

\begin{lemma}\label{lem:kpointsannuli}
Let $a_{1},a_{2} \in k$ and $r_{1},r_{2},r'_{1},r'_{2} \in \overline\R_{\ge 0}$ with $r_{1}<r'_{1}$ and $r_{2}<r'_{2}$. For $i=1,2$, consider the annulus $A_{i} := \{x\in  \AA_k^{1,\an}  :  r_{i} < |(T-a_{i})(x)| < r'_{i}\}$ with skeleton $\Gamma_{A_{i}} = \{\eta_{a_{i},s}  :  r_{i}< s<r'_{i}\}$. Let $h \colon A_{1} \to A_{2}$ be a surjective morphism such that $h^{-1}(\Gamma_{A_{2}}) = \Gamma_{A_{1}}$. 

Then, for $s_{1} \in (r_{1},r'_{1}) \cap |k^\times|$ and $s_{2} \in (r_{2},r'_{2})\cap |k^\times|$, the following conditions are equivalent:
\begin{enumerate}[$i)$]
\item $h(\eta_{a_{1},s_{1}}) = \eta_{a_{2},s_{2}}$;
\item $h(\DD(a_{1},s_{1}) \setminus \DD^-(a_{1},s_{1})) = \DD(a_{2},s_{2}) \setminus \DD^-(a_{2},s_{2})$;
\item $h(D(a_{1},s_{1}) \setminus D^-(a_{1},s_{1})) = D(a_{2},s_{2}) \setminus D^-(a_{2},s_{2})$.
\end{enumerate}
and, for $b_{1},b_{2} \in k$ with $r_{1}<|b_{1}-a_{1}|<r'_{1}$ and $r_{2}<|b_{2}-a_{2}|<r'_{2}$, $s_{1} \in [0,|b_{1}-a_{1}|)$ and $s_{2} \in [0,|b_{2}-a_{2}|)$, the following conditions are equivalent:
\begin{enumerate}[$i')$]
\item $h(\eta_{b_{1},s_{1}}) = \eta_{b_{2},s_{2}}$;
\item $h(\DD(b_{1},s_{1})) = \DD(b_{2},s_{2})$;
\item $h(D(b_{1},s_{1})) = D(b_{2},s_{2})$.
\end{enumerate}
\end{lemma}
\begin{proof}
By assumption, we have $h^{-1}(\Gamma_{A_{2}}) = \Gamma_{A_{1}}$. It follows that the connected components of $A_{1}\setminus \Gamma_{A_{1}}$ (which are discs) are sent to connected components of $A_{2}\setminus \Gamma_{A_{2}}$ (which are discs). The equivalence between $i')$, $ii')$ and~$iii')$ now follows from Lemma~\ref{lem:kpoints}.

$i) \implies ii)$ For $i=1,2$, the union of the connected components of $A_{i}\setminus \Gamma_{A_{i}}$ whose closure contains~$\eta_{a_{i},s_{i}}$ is equal to $\DD(a_{i},s_{i})\setminus \DD^-(a_{i},s_{i})$. We deduce that $h(\DD(a_{1},s_{1}) \setminus \DD^-(a_{1},s_{1})) \subseteq \DD(a_{2},s_{2}) \setminus \DD^-(a_{2},s_{2})$. By \cite[3.6.24]{Duc-book}, the restriction of~$h$ to~$\Gamma_{A_{1}}$ is injective, and the other inclusion follows.

\smallbreak

$ii) \implies iii)$ This is obvious.

\smallbreak

$iii) \implies i)$ The union of the connected components of $A_{1}\setminus \Gamma_{A_{1}}$ containing $D(a_{1},s_{1}) \setminus D^-(a_{1},s_{1})$, which is equal to $\DD(a_{1},s_{1}) \setminus \DD^-(a_{1},s_{1})$ is sent by~$h$ to the union of the connected components of $A_{2}\setminus \Gamma_{A_{2}}$ containing $D(a_{2},s_{2}) \setminus D^-(a_{2},s_{2})$, which is equal to $\DD(a_{2},s_{2}) \setminus \DD^-(a_{2},s_{2})$. By continuity, the boundary of the former set, which is $\{\eta_{a_{1},s_{1}}\}$ is sent to the boundary of the latter, which is $\{\eta_{a_{2},s_{2}}\}$.
\end{proof}

\subsection{Model theory of algebraically closed valued fields}\label{subsec:mod}

Let $\cL_{0}$ be a first order one-sorted language. We write $\cL_{0}$-structures in bold letters like $\mathbf{M}$ or $\mathbf{k}$ and we let $M$ and $k$ be their underlying universes. We also use the notation $(\mathbf{k}, \cL_{0})$ to indicate that $\mathbf{k}$ is an $\cL_{0}$-structure. All multi-sorted languages $\cL$ that will be considered in this article will be reducts of $\cL_{0}^{\mathrm{eq}}$ for some one-sorted language~$\cL_{0}$, and we suppose that $\cL$ contains the home sort from $\cL_{0}$. As such, we will also write $\cL$-structures with bold letters like $\mathbf{M}$ or $\mathbf{k}$ where $M$ and $k$ denote respectively the universe of the home sort $\cL_{0}$. We write $a\in \mathbf{k}$ and $C\subseteq \mathbf{k}$ to say that $a$ is an element of \emph{some} sort in $\cL$ and that $C$ is a subset of the (disjoint) union of all sorts. 

Let $\mathbf{k}$ be a (possibly multi-sorted) $\cL$-structure and let $\theta(x)$ be a formula where $x$ is a tuple of variables ranging over possibly different sorts. We let $\theta(\mathbf{k})$ denote the set of tuples in $\mathbf{k}$ that satisfy $\theta$. For sorts $S_1,\ldots, S_n$ in $\cL$, a subset $X\subseteq S_1\times\cdots \times S_n$ is $\cL$-definable if $X=\theta(\mathbf{k})$ for some $\cL$-formula with parameters in $\mathbf{k}$. In particular when $X$ is $\cL$-definable (over $\mathbf{k}$) and $\mathbf{k}'$ is an elementary extension of $\mathbf{k}$, we use $X(\mathbf{k}')$ to denote the set of tuples in $\mathbf{k}'$ that satisfy the formula that defines $X$. If needed, we sometimes write $X(\mathbf{k})$ redundantly to express that we work over the points of $X$ in $\mathbf{k}$. Note that we use $X(\mathbf{k})$ and not $X(k)$, as the latter expression will usually have a different meaning (see Convention \ref{convention.alg}).

\subsubsection{Algebraic structure} 

We will study the valued field $(k,|\cdot|)$ as a first order structure using different languages. The first one is the three-sorted language~$\cL_{3}$ defined by:
\begin{itemize}
\item a sort $\mathbf{VF}$ for the valued field $k$ (the home sort) in the language of rings $\cLr:=\{+,-,\cdot,0,1\}$, 
\item a sort $\mathbf{VG}$ for $|k|$ (that is, the value group $|k^\times|$ extended with an infinitely small element denoted by 0) in the language $\{\leqslant, \cdot, 1, 0\}$, 
\item a sort $\mathbf{RF}$ for the residue field $\tilde{k}$ in the language $\cLr$,
\item the valuation $|\cdot|\colon k\to |k|$ map and the residue map $\res(x,y)\colon K^2\to \tilde{k}$, sending $(x,y)$ to the residue of $xy^{-1}$ if $|x|\leqslant |y|\neq 0$, and to 0 otherwise. 
\end{itemize}

We will also consider a language $\cL_\mathbb{B}$ extending~$\cL_3$ in which we add a sort for the interpretable set of closed discs with radius in $|k|$. Formally, consider the $\cL_3$-definable equivalence relation $\sim$ on $\mathbf{VF}\times \mathbf{VG}$ given by 
\[
(a,r)\sim (b,s) \Leftrightarrow r=s \wedge |a-b|\leqslant r\Leftrightarrow D(a,r)=D(b,s).
\]
We let the set $\mathbb{B}$ denote the quotient~$(\mathbf{VF}\times \mathbf{VG})/\sim$. The language~$\cL_\mathbb{B}$ corresponds to~$\cL_3$ extended by a new sort for~$\mathbb{B}$ and a symbol $\mathbf{b}$ for the quotient map $\mathbf{b}\colon (\mathbf{VF}\times \mathbf{VG}) \to \mathbb{B}$. 

\medbreak

Let $k$ be a valued field and $\mathbf{k}$ be its associated $\cL_\BB$-structure. The set $\mathbb{B}(\mathbf{k})$ is in bijection with $\AA_k^{1,\Def}$ via the map $(a,r)/{\sim} \mapsto \eta_{a,r}$. Using this identification, in what follows we will treat $\AA_k^{1,\Def}$ as the $\cL_\BB$-definable set $\BB(\mathbf{k})$ and we will use the notation $\eta_{a,r}$ to denote its elements. The reader should think of $\AA_k^{1,\Def}$ as the $\cL_\BB$-definable avatar of the Berkovich affine line. This identification will be further strengthened in Section \ref{sec:HL}. As usual, when $\mathbf{k}$ is clear form the context, we will write $\BB$ instead of $\BB(\mathbf{k})$.  

The canonical injection from $k$ to $\BB$ sending $a$ to $\eta_{a,0}$ and the map $\eta_{a,r}\mapsto r$ are clearly $\cL_\BB$-definable. Similarly, the map $\red\colon(\mathbb{D}_k^{\D}\setminus\{\eta_{0,1}\}) \to \tilde{k}$ given by $\eta_{a,r}\mapsto \res(a,1)$ is $\cL_\BB$-definable. Note that the function $\red$ agrees with the usual reduction map in the sense of~\cite[Section~2.4]{rouge}. 

\begin{remark}\label{rem:auto} For every automorphism $\sigma\colon \AA_k^1\to\AA_k^1$, the restriction of $\sigma^\an$ to $\AA_k^{1,\Def}$ is $\cL_\BB$-definable. Indeed, $\sigma^\an$ can be defined by 
\[
\sigma^\an(\eta_{a,r}) = \eta_{b,s} \Leftrightarrow \sigma(D(a,r))=D(b,s). 
\] 
\end{remark}
The theory of algebraically closed valued fields does not eliminate imaginaries in the languge $\cL_\BB$. By a result of Haskell, Hrushovski and Macpherson \cite{HHM}, it does in the so-called \emph{geometric language} which we now recall. The geometric language, denoted $\cL_\mathcal{G}$, corresponds to the extension of $\cL_3$ by adding the following sorts:
\begin{itemize} 
\item for each $n>0$, a sort $S_n$ of $k^\circ$-lattices in $k^n$ (free $k^\circ$-submodules of rank $n$), which corresponds to the quotient $\mathrm{GL}_n(k)/\mathrm{GL}_n(k^\circ)$, 
\item for each $n>0$, a sort $T_n$ for the union of all quotients $s/k^{\circ\circ} s$, where $s$ is an $k^\circ$-lattice in $k^n$ (they are $\tilde{k}$-vector spaces of dimension $n$). 
\end{itemize}

We finish this subsection with the following convention concerning algebraic varieties. 

\begin{convention} \label{convention.alg}
By elimination of imaginaries in algebraically closed fields, all varieties (algebraic and projective) and all finite algebraic morphisms over $k$ (resp. over $\tilde{k}$) may be identified with $\cLr$-definable sets defined over $k$ (resp. over $\tilde{k}$). If $X$ is a variety, we let $X(k)$ be some $\cLr$-definable set which is in bijection with the set of $k$-points of $X$. The association $X \mapsto X(k)$ is functorial up to definable bijections. We refer the reader to \cite[Remark 3.10]{pillay_ACF} and \cite[Chapter II, Propositions 2.6 and 4.10]{hartshorne}  for the necessary background on this identification. 
\end{convention}

\subsubsection{Analytic structure}

We will also consider expansions of $(k,\cL_3)$ and $(k,\cL_\mathbb{B})$ by adding analytic structure to the valued field sort $\mathbf{VF}$ as defined in \cite{CL} by Cluckers and Lipshitz. Following their terminology, an analytic structure is given by a \emph{separated Weierstrass system}, that is, a family $\bigcup_{m,n} A_{m,n}$ of rings $A_{m,n}\subseteq k^\circ[\![X,Y]\!]$ with $X=(X_1,\ldots,X_m)$, $Y=(Y_1,\ldots,Y_n)$, satisfying further analytic properties such as Weierstrass preparation and division theorems, among others (see \cite[Section 4.1]{CL}). Here we will work over the analytic structure defined in \cite[Example (3), Section 4.4]{CL} which goes back to Lipshitz \cite{L}. The key property for our purposes is that each $A_{m,n}$ contains $k^\circ\{X_1,\ldots,X_m\}$, the ring of power series with coefficients in $k^\circ$ which converge to 0. The language $\cL^\an$ corresponds to the language of rings $\cLr$ together with a function symbol for each $f\in \bigcup_{m,n} A_{m,n}$ and a new unary function symbol $\cdot^{-1}$. We interpret the language $\cL^\an$ on $k$ as follows:

\begin{itemize}
\item the symbol $\cdot^{-1}$ is interpreted as the function sending $x\in k^\times$ to its multiplicative inverse $x^{-1}$ and $0^{-1}:=0$ by convention; 
\item to each $f\in A_{m,n}$ the corresponding function symbol $f^{\cL^{\an}}$ is interpreted by 
\[
f^{\cL^{\an}}(x,y)=
\begin{cases}
f(x,y) & \text{ if } (x,y)\in (k^\circ)^m\times (k^{\circ\circ})^n\\
0 & \text{ otherwise. }
\end{cases}
\]
\end{itemize}
The languages $\cL_3^\an$ and $\cL_\mathbb{B}^\an$ correspond respectively to the extensions of $\cL_3$ and $\cL_\mathbb{B}$ in which $\cLr$ is replaced by $\cL^\an$ in the valued field sort $\mathbf{VF}$. They are interpreted in $k$ in the obvious way. 

Let $\cA$ be a strictly $k$-affinoid algebra. By definition, it admits a presentation of the form
\[\cA = k\{T_{1},\dotsc,T_{n}\}/(f_{1},\dotsc,f_{m}),\]
with $f_{1},\dotsc,f_{m} \in k\{T_{1},\dotsc,T_{n}\}$. We let $f$ denote the tuple $(f_1,\ldots,f_m)$ and we treat $f$ as a given presentation of $\cA$.  

\begin{definition}
Let $\cA = k\{T_{1},\dotsc,T_{n}\}/(f_{1},\dotsc,f_{m})$ be a strictly $k$-affinoid algebra. 

\begin{enumerate}[$i)$]
\item The given presentation is said to be \emph{Lipshitz} if $f_{1},\dotsc,f_{m} \in k^\circ\{T_{1},\dotsc,T_{n}\}$.
\item The given presentation is said to be \emph{distinguished} if the spectral norm on~$\cA$ coincides with the residue norm induced by the supremum norm on the Tate algebra $k\{T_{1},\dotsc,T_{n}\}$.
\end{enumerate}
\end{definition}

\begin{lemma}\phantomsection\label{lem:presentation}
\begin{enumerate}[$i)$]
\item Any strictly $k$-affinoid algebra admits a Lipshitz presentation.
\item Any reduced strictly $k$-affinoid algebra admits a distinguished Lipshitz presentation.
\end{enumerate}
\end{lemma}
\begin{proof}
Let $\cA = k\{T_{1},\dotsc,T_{n}\}/(f_{1},\dotsc,f_{m})$ be a strictly $k$-affinoid algebra. Multiplying the elements $f_{1},\dotsc,f_{m}$ by a constant does not change the algebra~$\cA$, which proves point~i).

Assume that~$\cA$ is reduced. By \cite[Theorem 6.4.3/1 and Proposition 6.2.1/4]{BGR}, it admits a distinguished presentation, which we can turn into a distinguished Lipshitz presentation as before.
\end{proof}

Suppose from now on that $\cA = k\{T_{1},\dotsc,T_{n}\}/(f_{1},\dotsc,f_{m})$ is a strictly $k$-affinoid algebra with a Lipshitz presentation. Set $X := \cM(\cA)$. We associate to this data the $\cL^\an$-definable subset 
\[
\{(x_{1},\dotsc,x_{n}) \in k^n  :  |x_{i}| \le 1, f_{j}(x_{1},\dotsc,x_{n})=0, 1\le i\le n, 1\le j\le m \}.
\]
The given presentation induces a bijection between this set and the set~$X(k)$ of $k$-rational points of~$X$. We will therefore abuse notation and denote the former set by~$X_f(k)$ to keep in mind that it depends on the choice of a presentation of~$\cA$.  

Let $\cB = k\{U_{1},\dotsc,U_{p}\}/(g_{1},\dotsc,g_{q})$ be a strictly $k$-affinoid algebra with a Lipshitz presentation $g=(g_{1},\dotsc,g_{q})$. Set $Y := \cM(\cB)$ and define $Y_g(k)$ as before.

A bounded morphism $\varphi^* \colon \cB \to \cA$ is determined by the data of $h_{1} = \varphi(U_{1}),\dotsc,h_{p} = \varphi(U_{p})$ in~$\cA$. Note that we have 
\begin{equation}\label{eq:hi}
\begin{cases}
\forall i\in \{1,\dotsc,p\},\ \|h_{i}\|_{\sp} \le 1;\\
\forall j \in \{1,\dotsc,q\},\ g_{j}(h_{1},\dotsc,h_{p}) = 0,
\end{cases}
\end{equation}
where~$\nm_{\sp}$ denotes the spectral seminorm on~$\cA$. 

Conversely, given~$h_{1},\dotsc,h_{p} \in \cA$ satisfying~\eqref{eq:hi}, there exists a bounded morphism~$\varphi^* \colon \cB \to \cA$ such that, for each~$i\in \{1,\dotsc,p\}$, $\varphi(U_{i}) = h_{i}$. 

Let~$\varphi\colon X\to Y$ denote the morphism induced by~$\varphi^*$ and~$\varphi_{fg}\colon X_f(k)\to Y_g(k)$ be the map induced by~$\varphi$. 

\begin{lemma}\label{lem:mapdistinguished}
Assume that the presentation of~$\cA$ is distinguished. Then, the map $\varphi_{fg} \colon X_f(k) \to Y_g(k)$ is $\cL^\an$-definable. 
\end{lemma}
\begin{proof}
Let $i\in \{1,\dotsc,p\}$. Since the presentation of~$\cA$ is distinguished, we may find a representative $H_{i} \in k\{T_{1},\dotsc,T_{n}\}$ of $h_{i}\in \cA$ with $\|H_{i}\| = \|h_{i}\|_{\sp}$. Since $\|h_{i}\|_{\sp} \le 1$, we have $H_{i} \in k^\circ\{T_{1},\dotsc,T_{n}\}$ which are elements of $A_{m,0}$.

The map $\varphi_{fg} \colon X_f(k) \to Y_g(k)$ now coincides with the map
\[(x_{1},\dotsc,x_{n}) \in X_f(k) \mapsto (H_{1}(x_{1},\dotsc,x_{n}),\dotsc,H_{p}(x_{1},\dotsc,x_{n})) \in Y_g(k),\]
which is $\cL^\an$-definable.
\end{proof}

\begin{corollary} The map $\varphi_{fg} \colon X_f(k) \to Y_g(k)$ induced by~$\varphi$ is $\cL^\an$-definable.
\end{corollary}
\begin{proof}
Let~$\cA_\red$ be the reduction of~$\cA$. By Lemma~\ref{lem:presentation}, $\cA_\red$ admits a distinguished Lipshitz presentation $h$ that we use from now on. Set $X' := \cM(\cA_\red)$ and let $\psi\colon X'\to X$ be the natural morphism. 

By Lemma~\ref{lem:mapdistinguished}, the map $\psi_{hf} \colon X_h'(k) \to X_f(k)$ is  $\cL^\an$-definable. Note that it is also bijective and that its inverse is necessarily $\cL^\an$-definable too. 

By Lemma~\ref{lem:mapdistinguished} again, the map $\varphi_{fg} \circ \psi_{hf} \colon X_h'(k) \to Y_g(k)$ is  $\cL^\an$-definable. The result now follows by writing $\varphi_{fg} = (\varphi_{fg} \circ \psi_{hf}) \circ \psi_{hf}^{-1}$.
\end{proof}

\begin{convention}\label{convention.an} 
Thanks to those results, in the rest of the text, we will not choose presentations for the affinoid algebras we consider. We will implicitly assume that they are endowed with Lipshitz presentations and, in this case, the associated definable sets will only depend on those presentations up to definable bijections. In particular if $\cA$ is a strictly $k$-affinoid algebra and $X:=\cM(\cA)$ we will write $X(k)$ to denote the $\cL^\an$-definable set $X_f(k)$ for some Lipshitz presentation $f$ of $\cA$.   
\end{convention}

\subsection{Common results:}

We refer the reader to \cite[Theorem 7.1]{HHM} and \cite[Theorem 4.5.15]{CL} for the proof of the following results. 

\begin{theorem}\label{thm:QE} The theories of $(k,\cL_3)$ and $(k,\cL_3^\an)$ have quantifier elimination. \qed
\end{theorem}

\begin{theorem}[{\cite[Theorem 1.6]{LR}}]\label{thm:cmin} The structure $(k,\cL_3^\an)$ is \emph{$C$-minimal}, that is, every $\cL_3^\an$-definable subset of $X\subseteq k$ is a finite disjoint union of Swiss cheeses (and this condition also holds in every elementarily equivalent structure). \qed
\end{theorem}

The following are some consequences of quantifier elimination (which also follow from $C$-minimality, see~\cite{HM} or~\cite{C}).  

\begin{lemma}\label{lem:min} Consider a structure $(k,\cL_3^\an)$. Then 
\begin{enumerate}[$i)$]
\item the value group is o-minimal, that is, every $\cL_3^\an$-definable set $Y\subseteq |k|$ is a finite union of intervals and points (intervals with endpoints in $|k|\cup \{+\infty\}$). 
\item The residue field is strongly minimal, that is, every $\cL_3^\an$-definable set $Y\subseteq \tilde{k}$ (respectively $Y\subseteq C(\tilde{k})$ for $C$ an algebraic curve over $\tilde{k}$) is either finite or cofinite.  
\end{enumerate}
\qed
\end{lemma} 
In fact, one can show both $|k|$ and $\tilde{k}$ are \emph{stably embedded}: every $\cL_3^\an$-definable subset of $|k|^n$ is already definable in the language $\{\leqslant,\cdot,0,1\}$, and analogously, every $\cL_3^\an$-definable subset of $(\tilde{k})^n$ is already $\cLr$-definable.

\section{Local and global structure of curves}\label{sec:bricks}

In this section, we will state structure results for $k$-analytic curves. Those results are mostly well-known and derive from the semistable reduction theorem (see for instance \cite[Theorem~4.3.1]{rouge}). In~\cite{Duc-book}, Antoine Ducros provided a thorough reference on those questions, including full purely analytic proofs. We heavily borrow from his presentation and our proofs are often inspired by his. The main novelty of our presentation lies in the formalism we introduce, that fits our needs, and in making more precise the categories in which we can find the various isomorphisms coming into play, which will be crucial for us in the following.

\subsection{Algebraic curves}

\begin{convention}
In what follows, an algebraic curve will be a separated reduced purely one-dimensional scheme of finite type over~$k$.
\end{convention}

Note that the analytification of an algebraic curve is nice in the sense of Definition~\ref{def:nice}.

\subsubsection{Local structure}\label{sec:localalgebraic}

Let~$C$ be an algebraic curve over~$k$. We will prove local structure results around smooth points of~$C^\an$. Note that, since~$C^\an$ is reduced and~$k$ is algebraically closed, each point of type~2, 3 or~4 is automatically smooth.

We will denote by~$\sM$ the sheaf of meromorphic functions on~$C$. Note that, in our case, a meromorphic function is locally a quotient of regular functions.

\begin{proposition}\label{prop:type1}
Let $x\in C^\an$ be a smooth point of type~1. Let~$U$ be a neighborhood of~$x$ in~$C^\an$. There exist a Zariski-open subset~$O$ of~$C$, a morphism $f \colon O \to \AA^1_{k}$ and an open neighborhood~$V$ of~$x$ in~$U$ such that
\begin{enumerate}[$i)$]
\item $V \subseteq O^\an$;
\item $f^\an(V)$ is an open disc with radius in~$|k^\times|$ in~$\AA^{1,\an}_{k}$;
\item $f^\an_{|V}$ induces an isomorphism onto its image.
\end{enumerate}
\end{proposition}
\begin{proof}
Since~$C$ is smooth at~$x$, there exists a Zariski-open subset~$O$ of~$C$ containing~$x$ and an \'etale morphism $f \colon O \to \AA^1_{k}$. By the inverse function theorem, $f^\an$ induces an isomorphism around~$x$ and the result follows.
\end{proof}

\begin{lemma}\label{lem:densityHx}
Let $x \in C^\an$ and $\eps>0$. For each $\alpha_{0} \in \sH(x)$, there exists $\alpha\in \sM(C)$ such that $\alpha$ is regular at~$x$, non-constant around~$x$ and satisfies $|\alpha(x) - \alpha_{0}| < \eps$.
\end{lemma}
\begin{proof}
We may reduce to the case where~$C$ is affine, say $C = \textrm{Spec}(A)$. In this case, the result follows from the density of the total ring of fractions of~$A$ in~$\sH(x)$.
\end{proof}

\begin{proposition}\label{prop:type4}
Let $x\in C^\an$ be a smooth point of type~4. Let~$U$ be a neighborhood of~$x$ in~$C^\an$. There exist a Zariski-open subset~$O$ of~$C$, a morphism $f \colon O \to \AA^1_{k}$ and an open neighborhood~$V$ of~$x$ in~$U$ such that
\begin{enumerate}[$i)$]
\item $V \subseteq O^\an$;
\item $f^\an(V)$ is an open disc with radius in~$|k^\times|$ in~$\AA^{1,\an}_{k}$;
\item $f^\an_{|V}$ induces an isomorphism onto its image.
\end{enumerate}
\end{proposition}
\begin{proof}
By \cite[Th\'eor\`eme~4.5.4]{Duc-book}, there exists a isomorphism $f_{0} \colon U' \to D$ from a neighborhood~$U'$ of~$x$ in~$U$ to a disc~$D$. In particular, there exists $\alpha_{0} \in \sH(x)$ such that $k(\alpha_{0})$ is dense in~$\sH(x)$. By~\cite[Lemma~3.1.6]{temkin_transcendence} 
and Lemma~\ref{lem:densityHx}, there exists $\alpha \in \sM(C)$ such that $\alpha$ is non-constant, regular at~$x$ and $k(\alpha(x))$ is still dense in~$\sH(x)$.

The element~$\alpha$ induces a non-constant morphism $f \colon O \to \PP^1_{k}$, where~$O$ is a Zariski-open subset of~$C$ containing~$x$. Note that it is \'etale at~$x$. Up to composing by an automorphism of~$\PP^1_{k}$, we may assume that $f^\an(x) \ne \infty$. Up to replacing~$O$ by $O\setminus f^{-1}(\infty)$, we may assume that we have a morphism $f \colon O \to \AA^1_{k}$. By construction, $f$~induces an isomorphism $\sH(f^\an(x)) \simto \sH(x)$. 
By \cite[Theorem~3.4.1]{bleu}, it is an isomorphism around~$x$ and the result follows.
\end{proof}

\begin{proposition}\label{prop:type3}
Let $x\in C^\an$ be a point of type~3. Let~$U$ be a neighborhood of~$x$ in~$C^\an$. There exists a Zariski-open subset~$O$ of~$C$, a morphism $f \colon O \to \AA^1_{k}$ and an open neighborhood~$V$ of~$x$ in~$U$ such that
\begin{enumerate}[$i)$]
\item $V \subseteq O^\an$;
\item $f^\an(V)$ is an open annulus with radii in~$|k^\times|$ in~$\AA^{1,\an}_{k}$;
\item $f^\an_{|V}$ induces an isomorphism onto its image.
\end{enumerate}
\end{proposition}
\begin{proof}
By \cite[Th\'eor\`eme~4.5.4]{Duc-book}, there exists an isomorphism $f_{0} \colon U' \to A$ from a neighborhood~$U'$ of~$x$ in~$U$ to an annulus~$A$. We can then argue as in the preceding proof.
\end{proof}

In order to handle points of type~2, we will need the following result.

\begin{lemma}\label{lem:liftresiduemorphism}
Let $x\in C^\an$ be a point of type~2. Let $\tilde f \colon \cC_{x} \to \PP^1_{\tilde k}$ be a finite morphism. There exist a Zariski-open subset~$O$ of~$C$ and a morphism $f \colon O \to \AA^1_{k}$ such that
\begin{enumerate}[$i)$]
\item $x\in O^\an$, $f^\an(x) = \eta_{0,1}$ and $f^\an$ is finite at~$x$;
\item $\tilde f_{x} = \tilde f$.
\end{enumerate}
\end{lemma}
\begin{proof}
Denote by~$t$ the image of~$T$ in~$\widetilde{\sH(\eta_{0,1})}$. Remark that $\tilde k (t) \simeq\widetilde{\sH(\eta_{0,1})}$.

Set $\tilde \alpha := \tilde f^*(t)$. Choose $\alpha \in \sM(C)$ such that $|\alpha(x)| =1$ and $\widetilde{\alpha(x)} = \tilde\alpha$. There exists a Zariski-open subset~$O$ of~$C$ and a morphism $f \colon O \to \PP^1_{k}$ such that $f^\ast(T) = \alpha$. As in the proof of Proposition~\ref{prop:type4}, up to composing by an automorphism of~$\PP^1_{k}$ and shrinking~$O$, we may assume that we have a morphism $f \colon O \to \AA^1_{k}$. The required properties are now satisfied by construction.
\end{proof}

The following result is essentially a reformulation of \cite[Th\'eor\`eme~4.4.15]{Duc-book} in a more precise form. We include the proof for the reader's convenience.

\begin{proposition}\label{prop:type2generic}
Let $x\in C^\an$ be a point of type~2. Let~$U$ be a neighborhood of~$x$ in~$C^\an$. There exist a Zariski-open subset~$O$ of~$C$, a morphism $f \colon O \to \AA^1_{k}$ and an affinoid domain~$V$ of~$U$ with Shilov boundary~$\{x\}$ such that
\begin{enumerate}[$i)$]
\item $V \subseteq O^\an$;
\item $f^\an(x) = \eta_{0,1}$ and $f^\an(V)$ is equal to~$\DD_{k}$ deprived of finitely many open unit discs;
\item for each connected component~$E$ of~$V\setminus\{x\}$, $f^\an(E)$ is a connected component of $f^\an(V)\setminus\{\eta_{0,1}\}$ and the morphism~$f^\an$ induces an isomorphism between~$E$ and~$f^\an(E)$.
\end{enumerate}
In particular, every connected component of~$V\setminus\{x\}$ is an open disc.
\end{proposition}
\begin{proof}
Let $\tilde f \colon \cC_{x} \to \PP^1_{\tilde k}$ be a finite generically \'etale morphism and lift it to a morphism $f \colon O \to \AA^1_{k}$ as in Lemma~\ref{lem:liftresiduemorphism}. 

Let~$U'$ be a connected open neighborhood of~$x$ in~$O^\an$. Up to restricting~$U'$, we may assume that~$f^\an$ induces a finite morphism from~$U'$ to an open subset of $\AA^{1,\an}_{k}$ and that each connected component of $U'\setminus\{x\}$ contains only one branch emanating from~$x$ (see \cite[Theorem~4.5.4]{Duc-book}). In this case, the degree of the restriction of~$f^\an$ to such a connected component is equal to the ramification index of~$\tilde f$ on the corresponding rational point.

Denote by~$\cE$ the set of connected components of $\DD_{k} \setminus \{\eta_{0,1}\}$. For each finite subset~$\cF$ of~$\cE$, the set $V_{\cF} := \{\eta_{0,1}\} \cup \bigcup_{E \in \cE\setminus \cF} E$ is an affinoid domain of~$\DD_{k}$. More precisely, it is equal to~$\DD_{k}$ deprived of finitely many open unit discs. Moreover, each neighborhood of~$\eta_{0,1}$ contains such a set $V_{\cF}$. In particular, by choosing~$\cF$ big enough, we may assume that  the connected component~$V$ of $(f^\an)^{-1}(V_{\cF})$ containing~$x$ is contained in~$U'$ and that the morphism~$\tilde f$ is \'etale, hence unramified, over $\cU_{x} := \tilde{f}^{-1}(\cU_{\cF})$, where~$\cU_{\cF}$ denotes the complement of the image of~$\cF$ in $\cC_{\eta_{0,1}} = \PP^1_{\tilde k}$. The result follows.
\end{proof}

We give names to the notions we have just introduced.

\begin{definition}
An \emph{algebraic open disc} is the data of a Zariski-open subset~$O$ of~$C$, an open subset~$V$ of~$C^\an$ and a morphism $f \colon O \to \AA^1_{k}$ such that
\begin{enumerate}[$i)$]
\item $V \subseteq O^\an$;
\item $f^\an(V)$ is an open disc with radius in~$|k^\times|$ in~$\AA^{1,\an}_{k}$;
\item $f^\an_{|V}$ induces an isomorphism onto its image.
\end{enumerate}

An \emph{algebraic open annulus} is the data of a Zariski-open subset~$O$ of~$C$, an open subset~$V$ of~$C^\an$ and a morphism $f \colon O \to \AA^1_{k}$ such that
\begin{enumerate}[$i)$]
\item $V \subseteq O^\an$;
\item $f^\an(V)$ is an open annulus with radii in~$|k|$ in~$\AA^{1,\an}_{k}$;
\item $f^\an_{|V}$ induces an isomorphism onto its image.
\end{enumerate}

An \emph{algebraic tube} centered at a point~$x\in C^\an$ of type~2 is the data of a Zariski-open subset~$O$ of~$C$, an affinoid domain~$V$ of~$C^\an$ with Shilov boundary~$\{x\}$ and a morphism $f \colon O \to \AA^1_{k}$ such that
\begin{enumerate}[$i)$]
\item $V \subseteq O^\an$;
\item $f^\an(x) = \eta_{0,1}$ and $f^\an(V)$ is equal to~$\DD_{k}$ possibly deprived of finitely many open unit discs;
\item for each connected component~$E$ of~$V\setminus\{x\}$, $f^\an(E)$ is a connected component of $f^\an(V)\setminus\{\eta_{0,1}\}$ and the morphism~$f^\an$ induces an isomorphism between~$E$ and~$f^\an(E)$.
\end{enumerate}

We call \emph{algebraic brick} any triple $(O,V,f)$ of one of the three preceding sorts.

\medbreak

We will sometimes abusively say that a subset~$V$ of~$C^\an$ is an algebraic open disc (resp. algebraic open annulus, etc.) if there exist a Zariski-open subset~$O$ of~$C$ and a morphism $f \colon O \to \AA^1_{k}$ such that $(O,V,f)$ is an algebraic open disc (resp. algebraic open annulus, etc.).
\end{definition}

\begin{remark}\label{rem:restriction}
Let $(O,V,f)$ be an algebraic open disc or annulus. Let~$V'$ be a subset of~$V$ isomorphic to an open disc with radius in~$|k^\times|$ (resp. an open annulus in~$V$ with radii in~$|k|$). Then $(O,V',f)$ is an algebraic open disc (resp. algebraic open annulus). Similar results holds for tubes. For instance, if~$V'$ is a closed disc with radius in~$|k^\times|$ or a closed annulus with equal inner and outer radii belonging to~$|k^\times|$, then there exists an automorphism~$\alpha$ of~$\AA^1_{k}$ such that $(O,V',\alpha\circ f)$ is an algebraic tube. Note that the automorphism~$\alpha$ is needed to ensure that the image of~$V'$ lies in the closed unit disc centered at~0.

Let $(O,W,f)$ be an algebraic tube centered at~$x$. Let~$E$ be a connected component of $W\setminus\{x\}$. Then $(O,E,f)$ is an algebraic open disc and $(O,W\setminus E,f)$ is still an algebraic tube centered at~$x$. 

We will use freely the results of this remark in the rest of the text.
\end{remark}

In the setting of Proposition~\ref{prop:type2generic}, $V$ is a tube in the sense of Section~\ref{sec:residuecurve} and we have an associated reduction map $\rho_{V}\colon V \to \cC_{x}$ whose image~$\cU_{V}$ is an affine curve over~$\tilde k$. We fix a closed embedding $\iota_{V} \colon \cU_{V} \hookrightarrow \AA^r_{\tilde k}$.

\begin{lemma}\label{lem:residue_morph}
There exist a Zariski-open subset~$O$ of~$C$ and morphisms $f_{1},\dotsc,f_{r} \colon O \to \AA^1_{k}$ such that 
\begin{enumerate}[$i)$]
\item $x\in O^\an$;
\item for each $i\in \{1,\dotsc,r\}$, $f_{i}^\an(x) = \eta_{0,1}$ and $f_{i}^\an(O^\an \cap V) \subseteq \DD_{k}$;
\item the map 
\[z \in O^\an \cap V\setminus\{x\} \mapsto (\red(f_{1}^\an(z)),\dotsc,\red(f_{r}^\an(z))) \in \tilde{k}^r\] 
coincides with the restriction of $\iota_{V}\circ\rho_{V}$ to $O^\an \cap V\setminus\{x\}$.
\end{enumerate}
\end{lemma}
\begin{proof}
Let $i\in \{1,\dotsc,r\}$. Denote by $\tilde f_{i} \colon \cU_{V} \to \AA^1_{\tilde k}$ the composition of~$\iota_{V}$ with the $i^\text{th}$ projection map. It extends uniquely to a finite morphism $\tilde f_{i}' \colon \cC_{x} \to \PP^1_{\tilde k}$. Let us lift the latter to a morphism $f_{i} \colon O_{i} \to \AA^1_{k}$ as in Lemma~\ref{lem:liftresiduemorphism}. Up to shrinking the~$O_{i}$'s, we may assume that they are all equal to some common~$O$. We have $f_{i}^\an(x)=\eta_{0,1}$ as required. Moreover, since, by construction, each connected component of~$V\setminus\{x\}$ is sent into~$\cU_{V}$, we have
\[f_{i}^\an(O^\an \cap V\setminus\{x\}) \subseteq \beta_{\eta_{0,1}}^{-1}(\tilde f_{i}'(\cU_{V})) \subseteq \beta_{\eta_{0,1}}^{-1}(\AA^1_{\tilde k}) = \DD_{k},\]
where we have used the identification $\cB_{\eta_{0,1}} = \pi_{0}(\PP^{1,\an}_{k}\setminus\{\eta_{0,1}\})$. This proves property~i).

Property~ii) follows directly from the choice of the morphisms~$\tilde f_{i}$ and the explicit description of the map~$\beta_{\eta_{0,1}}$ in terms of~$\red$.
\end{proof}

\begin{proposition}\label{prop:type2branch}
Let $x\in C^\an$ be a point of type~2. Let~$U$ be a neighborhood of~$x$ in~$C^\an$ and let~$B$ be a connected component of~$U\setminus\{x\}$. There exist a Zariski-open subset~$O$ of~$C$, a morphism $f \colon O \to \AA^1_{k}$ and an open subset~$V$ of~$B$ whose closure contains~$x$ such that
\begin{enumerate}[$i)$]
\item $V \subseteq O^\an$;
\item $f^\an(V)$ is an open annulus with radii in~$|k^\times|$ in~$\AA^{1,\an}_{k}$;
\item $f^\an_{|V}$ induces an isomorphism onto its image.
\end{enumerate}
Moreover, given a finite set~$\cB'$ of connected components of~$U\setminus\{x\}$ not containing~$B$, we may ensure that $(f^\an)^{-1}(f^\an(V))$ does not meet any element of~$\cB'$. 
\end{proposition}
\begin{proof}
Let~$b$ be a branch emanating from~$x$ whose projection on $\pi_{0}(U\setminus\{x\})$ is~$B$. Let~$a$ be the $\tilde k$-rational point of~$\cC_{x}$ corresponding to~$b$. Pick a uniformizer of the local ring at~$a$. Seeing it as a meromorphic function on~$\cC_{x}$, it gives rise to a finite morphism $\tilde f \colon \cC_{x} \to \PP^1_{k}$ with a simple zero at~$a$. Let us lift it to a morphism $f \colon O \to \AA^1_{k}$ as in Lemma~\ref{lem:liftresiduemorphism}.

Since~$\tilde f$ is unramified at~$a$, the morphism~$f$ induces an isomorphism between a section of~$b$ and its image, \textit{i.e.} an open subset~$U'$ of~$B$ whose closure contains~$x$ and its image. The image~$f(U')$ is an open subset of~$\AA^{1,\an}_{k}$ whose closure contains $f^\an(x) = \eta_{0,1}$. We deduce that $f(U')$ contains an open annulus with radii in $|k^\times|$ whose closure contains~$\eta_{0,1}$. The result follows.

To prove the final part, it is enough to choose the morphism~$\tilde{f}$ in such a way that it does not vanish on any of the $\tilde k$-rational points of~$\cC_{x}$ corresponding to the elements of~$\cB'$. This may be done thanks to the independence of the associated valuations (see \cite[VI, \S 7, $n^\circ 2$, Th\'eor\`eme~1]{BourbakiAC56}).
\end{proof}

\begin{lemma}\label{lem:separationpoints}
Let~$x$ and~$y$ be distinct points of~$C^\an$. Then, there exist a Zariski-open subset~$O$ of~$C$, a morphism $g \colon O \to \AA^1_{k}$ and an open disc or open annulus~$A$ in~$\AA^{1,\an}_{k}$ with radii in~$|k^\times|$ such that
\begin{enumerate}[$i)$]
\item $(g^\an)^{-1}(A)$ contains~$x$;
\item $y$ does not belong to the closure of $(g^\an)^{-1}(A)$ in~$C^\an$.
\end{enumerate}
\end{lemma}
\begin{proof}
Let~$O$ be an affine Zariski-open subset of~$C$ such that~$O^\an$ contains~$x$ and~$y$. This is easy to construct by removing at most finitely many closed points from~$C$, since every curve with no proper irreducible component is affine. 

Consider a closed embedding $i \colon O \hookrightarrow \AA^N_{k}$ of~$O$ into an affine space~$\AA^N_{k}$ and its analytification $i^\an\colon O^\an \hookrightarrow \AA^{N,\an}_{k}$. Since~$x$ and~$y$ have different images in~$\AA^{N,\an}_{k}$, there exists $P \in k[T_{1},\dotsc,T_{N}]$, where $T_{1},\dotsc,T_{N}$ denote coordinates on~$\AA^N_{k}$, such that $|P(x)| \ne |P(y)|$. If $P(x)=0$, then, by density of $|k^\times|$ in~$\RR_{>0}$, we can find $r \in |k^\times|$ such that $|P(x)| < r < |P(y)|$. If $P(x) \ne 0$, then, similarly, we can find $r<s \in |k^\times|$ such that $|P(x)| \in (r,s)$ and $|P(y)|\notin [r,s]$. The result follows by noting that the polynomial~$P$ induces a morphism $\AA^N_{k} \to \AA^1_{k}$, hence by restriction a morphism $O \to \AA^1_{k}$.
 \end{proof}
 
 It follows from the definition of an algebraic brick $(O,V,f)$ that~$V$ is a connected component of the preimage by~$f$ of a subset of the affine line of a rather simple kind. However, we will soon turn to definability questions, and the definability of the corresponding subset of the affine line will not be enough to provide a definable counterpart of the brick. Indeed, connected components of definable spaces are not definable as a rule, since such a purely topological notion need not be expressible in the language. The aim of the following technical result is to prove that, under some finiteness hypotheses at the boundary, such a result will hold nonetheless in our setting.
 
\begin{definition}
We say that two subsets~$A$ and~$B$ of a $k$-analytic curve~$X$ are \emph{equal up to a finite set of $k$-rational points} if there exist two finite subsets~$A_{0}$ and~$B_{0}$ of~$X(k)$ such that 
\[A \cup A_{0} = B \cup B_{0}.\]
\end{definition}

\begin{proposition}\label{prop:separationbricks}
Let $\varphi\colon C \to C'$ be a morphism of algebraic curves over~$k$, let~$U'$ be a connected analytic domain of~$(C')^\an$ whose boundary is a finite set of points of type~1 or~2 and let~$U$ be a connected component of $(\varphi^\an)^{-1}(U')$ such that~$\varphi^\an$ is finite at each point of~$U$. Then, there exist a Zariski-closed subset~$Z$ of~$C$, a finite set~$M$, for each $m\in M$, a finite set~$N_{m}$ and, for each~$n\in N_{m}$, a Zariski-open subset~$O_{n}$ of~$C$, a morphism $f_{n}\colon O_{n} \to \AA^1_{k}$ and an open subset $A_{n}$ of~$\AA^{1,\an}_{k}$ that is either an open disc or an open annulus with radii in~$|k^\times|$ such that
\[U = (\varphi^\an)^{-1}(U') \cap Z^\an \cap \bigcup_{m\in M} \bigcap_{n\in N_{m}}  (f_{n}^\an)^{-1}(A_{n})\]
up to a finite set of $k$-rational points.
\end{proposition}
\begin{proof}
We want to express the connected component~$U$ in a ``definable way'' with respect to~$U'$. Here by ``definable way'' we simply mean to be a finite boolean combination as in the statement. Note however that at the level of $k$-points we do obtain a true definability transfer: if $U$ satisfies the identity of the statement and the $k$-points of $U'$ are $\cL$-definable (for $\cL$ either $\cL_\BB$ or $\cL_\BB^\an$), then the $k$-points of $U$ will also be $\cL$-definable. 

The strategy is as follows: we first reduce to the case where~$C$ and~$C'$ are smooth and proper, then to the case where the boundary of~$U'$ contains only points of type~2. The next step of the proof is to locally isolate~$U$ in a definable way above a given point point of~$U'$ or of its closure. We conclude by a compactness argument.

\medbreak

Let~$Z$ be the union of the irreducible components of~$C$ on which~$\varphi$ is not constant. It is a Zariski-closed subset of~$C$ whose analytification contains~$U$ by assumption. Up to replacing~$C$ by~$Z$, we may assume that~$\varphi$ has relative dimension~0.

Since the curves~$C$ and~$C'$ are generically smooth and we may work up to finitely many $k$-rational points, up to shrinking~$C$ and~$C'$, we may assume that they are smooth. The morphism~$\varphi$ extends to a morphism~$\bar\varphi$ between smooth compactifications~$\bar C$ and~$\bar C'$ of~$C$ and~$C'$. Since we may work up to finitely many $k$-rational points, up to shrinking~$U'$, we may assume that $(\bar\varphi^\an)^{-1}(U')$ is contained in~$C^\an$. We may now replace~$C$ and~$C'$ by $\bar C$ and $\bar C'$ respectively and~$\varphi$ by~$\bar\varphi$ without changing~$U$ and~$U'$, so we may assume that~$C$ and~$C'$ are smooth and projective. Note that~$\varphi$ is then finite and $C^\an$ and~$(C')^\an$ compact. 

\medbreak

Let~$x$ be a point of type~1 at the boundary of~$U$. There exists a neighborhood~$D$ of~$x$ in~$C^\an$ that is isomorphic to an open disc and such that $D \cap U = D\setminus \{x\}$. In particular, $U \cup \{x\}$ is still a connected analytic domain of~$C^\an$ and~$x$ does not belong to its boundary anymore.

Denote by~$U_{1}$ the set of type~1 points at the boundary of~$U$. By assumption, it is finite. Using the previous argument repeatedly, we show that $U \cup U_{1}$ is still a connected analytic domain of~$C^\an$ and that its boundary contains no points of type~1. 

Let us define the set~$U'_{1}$ similarly. Since each point in~$U'_{1}$ is smooth, it cannot belong to the closure of another connected component of~$(\varphi^\an)^{-1}(U)$. It follows that $U'\cup U'_{1}$ is a connected component of $\varphi^{-1}(U\cup U_{1})$. Up to replacing~$U$ and~$U'$ by $U \cup U_{1}$ and $U' \cup U'_{1}$, we may assume that~$U'$ has no points of type~1 in its boundary.

\medbreak

For each subset~$W$ of~$C^\an$, denote by~$\ov W$ the closure of~$W$ in~$C^\an$. Set $V := (\varphi^\an)^{-1}(U')$. It follows from the assumptions that $\ov{U}\setminus U$ and $\ov{V}\setminus V$ are finite sets of points of type~2. 

For each point $x\in C^\an$ of type~2 and each analytic domain~$W$ of~$C^\an$, denote by~$\cB_{x}(W)$ the set of branches emanating from the point~$x$ that belong to~$W$. Note that, for each~$x$ of type~2, we have $\cB_{x}(U) \subseteq \cB_{x}(V)$ and that, if $x\in \ov{U}\setminus U$ (resp. $x\in \ov{V}\setminus V$), then the set~$\cB_{x}(U)$ (resp. $\cB_{x}(V)$) is finite, because $U$ is connected (resp. $V$ has finitely many connected components). Set $\cB(U) := \bigcup_{x\in \ov{U}\setminus U} \cB_{x}(U)$ and $\cB(V) := \bigcup_{x\in \ov{V}\setminus V} \cB_{x}(V)$. Those are finite sets too.

Let $x\in \ov{U}\setminus U$ be a point of type~2 and let $b\in\cB_{x}(U)$. By Proposition~\ref{prop:type2branch}, there exist a Zariski-open subset~$O_{b}$ of~$C$, a morphism $f_{b} \colon O_{b} \to \AA^1_{k}$ and an open annulus~$A_{b}$ with radii in~$|k^\times|$ in~$\AA^{1,\an}_{k}$ such that $(f_{b}^\an)^{-1}(A_{b})$ contains the branch~$b$ but no other element of~$\cB_{x}(V)$. In particular, the closure~$F_{b}$ of $(f_{b}^\an)^{-1}(A_{b}) \cap (V\setminus U)$ in~$C^\an$ does not contain the point~$x$. By Lemma~\ref{lem:separationpoints}, for each $y \in F_{b}$, there exist a Zariski-open subset~$O_{y}$ of~$C$, a morphism $g_{y} \colon O_{y} \to \AA^1_{k}$ and an open disc or open annulus~$A_{y}$ with radii in~$|k^\times|$ in~$\AA^{1,\an}_{k}$ such that~$x$ belongs to $(g_{y}^\an)^{-1}(A_{y})$ and~$y$ does not belong to the closure of $(g_{y}^\an)^{-1}(A_{y})$ in~$C^\an$. Since~$C^\an$ is compact, $F_{b}$ is compact too, hence there exists a finite subset~$F^0_{b}$ of~$F_{b}$ such that
\[F_{b} \subseteq \bigcup_{y\in F^0_{b}} C^\an \setminus \overline{(g_{y}^\an)^{-1}(A_{y})}.\]
The set
\[U_{b} := V \cap (f_{b}^\an)^{-1}(A_{b}) \cap \bigcap_{y\in F^0_{b}} (g_{y}^\an)^{-1}(A_{y})\]
is an open subset of~$U$ containing~$b$ and the set 
\[U^- := U \setminus \bigcup_{b\in\cB(U)} U_{b}\] 
is compact.

Let~$x\in U^-$. The closure of~$V\setminus U$ in~$C^\an$ is a compact set that does not contain~$x$. Using Lemma~\ref{lem:separationpoints} as before, we deduce that there exist a finite set~$E_{x}$ and, for each $e \in E_{x}$, a Zariski-open subset~$O_{e}$ of~$C$, a morphism $g_{e} \colon O_{e} \to \PP^1_{k}$ and an open disc or open annulus~$A_{e}$ with radii in~$|k^\times|$ in~$\AA^{1,\an}_{k}$ such that
\[U_{x} := V \cap \bigcap_{e\in E_{x}} (g_{e}^\an)^{-1}(A_{e})\]
is an open subset of~$U$ containing~$x$.

Since~$U^-$ is compact, there exists a finite subset~$U^0$ of~$U^-$ such that 
$U^- \subseteq \bigcup_{x\in U^0} U_{x}$.
We conclude by writing 
\[U =  \bigcup_{b\in\cB(U)} U_{b} \cup \bigcup_{x\in U^0} U_{x}.\]
\end{proof}

\begin{corollary}\label{cor:definablealgbricks}
For each algebraic brick~$(O,V,f)$ of~$C^\an$, the set~$V(k)$ is a definable subset of~$C(k)$ and the map $V(k) \to k$ induced by~$f$ is definable.
\end{corollary}
\begin{proof}
Apply Proposition~\ref{prop:separationbricks} with $\varphi=f$, $U'=f(V)$ and $U=V$.
\end{proof}

\begin{corollary}\label{cor:definablereductionalg}
For each point $x\in C^\an$ of type~2 and each algebraic tube $(O,V,f)$ centered at~$x$, the map $V(k) \to \cC_{x}(\tilde k)$ induced by~$\rho_{V}$ is definable.
\end{corollary}
\begin{proof}
By Corollary~\ref{cor:definablealgbricks}, $V(k)$ is a definable subset of~$C(k)$. By Lemma~\ref{lem:residue_morph}, there exists a Zariski-open subset~$O'$ of~$C$ such that $(O')^\an$ contains~$x$ and the map $((O')^\an\cap V)(k) \to \cC_{x}(\tilde k)$ induced by~$\rho_{V}$ is definable.

The set~$\cW$ of connected components of~$V\setminus\{x\}$ that are not contained in~$(O')^\an$ is finite. By Proposition~\ref{prop:separationbricks}, for each $W\in \cW$, the set $W(k)$ is definable. Since the map of the statement is constant on such a~$W$, the result follows.
\end{proof}

\subsubsection{Global decomposition}\label{sec:globaldecomposition}

Let~$C$ be an algebraic curve.

\begin{lemma}\label{lem:intersectionalgebraic bricks}
Let~$V$ and~$V'$ be algebraic bricks of~$C^\an$. Then $V\cap V'$ is a finite disjoint union of algebraic bricks of~$C^\an$.
\end{lemma}
\begin{proof}
Assume that~$V$ or~$V'$ is an algebraic open disc or annulus. Without loss of generality, we may suppose that~$V'$ is. In this case, $V\cap V'$ is an open subset of~$V$ with finitely many boundary points in~$V$. Denote by~$E$ this set of boundary points. The set $V\cap V'$ is a disjoint union of connected components of $V \setminus E$. Note that all the connected components of $V\setminus E$ are finite disjoint union of algebraic bricks and that only finitely many of them are not open discs. But if an open disc is contained in~$V$ (resp.~$V'$), then either it is equal to the whole~$V$ (resp.~$V'$) or its closure is contained in~$V$ (resp.~$V'$). We deduce that no connected component of~$V\setminus E$ that is an open disc may be contained in $V\cap V'$, except when this connected component is equal to $V\cap V'$ itself. The result follows.

It remains to consider the case where~$V$ and~$V'$ are algebraic tubes. If~$V$ and~$V'$ have the same center, then their intersection is an algebraic tube (with the same center). If~$V$ and~$V'$ have different centers, say~$x$ and~$x'$ respectively, denote by~$D$ (resp.~$D'$) the unique connected component of~$V\setminus\{x\}$ (resp.~$V'\setminus\{x'\}$) containing the center of~$V'$ (resp.~$V$). Note that the algebraic tubes $T:=V\setminus D$ and $T':=V'\setminus D'$ are disjoint and that $T \subset D'$ and $T'\subset D$. Writing  $V \cap V' = T \sqcup T' \sqcup (D\cap D')$, we are reduced to the case of an intersection of two discs, which we already dealt with.
\end{proof}

\begin{lemma}\label{lem:complementalgebraic bricks}
Let~$V$ and~$V'$ be algebraic bricks of~$C^\an$. Then $V\setminus V'$  is a finite disjoint union of algebraic bricks of~$C^\an$.
\end{lemma}
\begin{proof}
Since $V\setminus V' = V\setminus (V\cap V')$ and $V\cap V'$ is a finite disjoint union of algebraic bricks by Lemma~\ref{lem:intersectionalgebraic bricks}, we may assume that $V' \subseteq V$. We may also assume that $V'\ne V$. We distinguish several cases.

\begin{itemize}
\item  $V$ is an algebraic open disc and $V'$ is an algebraic open disc

Then $V\setminus V'$ is an algebraic open annulus.

\item $V$ is an algebraic open disc and $V'$ is an algebraic open annulus

Then $V\setminus V'$ is the disjoint union of a closed disc (hence an algebraic tube) and a semi-open annulus (hence the disjoint union of an algebraic open annulus and an algebraic tube).

\item $V$ is an algebraic open annulus and $V'$ is an algebraic open disc

If $V'$ is a maximal open disc in~$V$, then $V\setminus V'$ is the disjoint union of an algebraic tube (centered at the boundary point of~$V'$ in~$V$) and two algebraic open annuli.

In general, $V'$ is contained in a maximal disc~$D$ of~$V$. Writing $V\setminus V' = (V\setminus D) \sqcup (D \setminus V')$, we are reduced to the previous cases.

\item $V$ is an algebraic open annulus and $V'$ is an algebraic open annulus

If $V'$ is contained in a disc~$D$ of~$V$, then, writing $V\setminus V' = (V\setminus D) \sqcup (D \setminus V')$, we are reduced to the previous cases.

Assume that~$V'$ is not contained in a disc of~$V$, \emph{i.e.} $V'\cap \Gamma_{V} \ne\emptyset$. Let $x \in V' \cap \Gamma_{V}$. Then, the connected components of~$V\setminus\{x\}$ are open discs, except for exactly two of them that are open annuli. If $x\notin \Gamma_{V'}$, then the connected components of~$V'\setminus\{x\}$ are open discs with boundary~$\{x\}$, except for exactly one of them. As a consequence, one of these open discs with boundary~$\{x\}$ is contained in an open annulus whose boundary contains~$\{x\}$, which is impossible. We deduce that $V' \cap \Gamma_{V} \subseteq \Gamma_{V'}$. 

The argument above also shows that, for a point $x\in \Gamma_{V'} \cap \Gamma_{V}$, the two connected components of $V'\setminus\{x\}$ that are not discs lie inside the two connected components of $V\setminus\{x\}$ that are not discs. In other words, the two branches emanating from~$x$ corresponding to~$\Gamma_{V'}$ coincide that corresponding to~$\Gamma_{V}$. It follows that $\Gamma_{V'} \subseteq \Gamma_{V}$, and we deduce that $V\setminus V'$ is a semi-open annulus or a disjoint union of two semi-open annuli.

\item $V$ is an algebraic open disc and $V'$ is an algebraic tube

Then $V\setminus V'$ is the disjoint union of an algebraic open annulus and finitely many algebraic open discs.

\item $V$ is an algebraic tube and $V'$ is an algebraic open disc

Since $V' \subseteq V$ and the center~$x$ of~$V$ is a boundary point of~$V$, it cannot belong to~$V'$, which has no boundary. It follows that~$V'$ is contained in a connected component~$D$ of~$V \setminus\{x\}$. Writing $V\setminus V' = D\setminus V'$, we are reduced to a previous case.

\item $V$ is an algebraic open annulus and $V'$ is an algebraic tube

If~$V'$ is contained in a disc~$D$ of~$V$, then, writing $V\setminus V' = (V\setminus D) \sqcup (D\setminus V')$, we are reduced to the previous cases.

Assume that~$V'$ is not contained in a disc of~$V$, \emph{i.e.} $V' \cap \Gamma_{V} \ne \emptyset$. Arguing as in the case where $V$ and $V'$ are open algebraic annuli, we prove no point other than the center~$x$ of~$V'$ may belong to~$\Gamma_{V}$. It follows that $V' \cap \Gamma_{V} =\{x\}$ and we deduce that $V\setminus V'$ is a disjoint union of two algebraic open annuli and finitely many algebraic open discs (with boundary~$\{x\}$).

\item $V$ is an algebraic tube and $V'$ is an algebraic open annulus

By the same argument as in the case where $V$~is an algebraic tube and $V'$~is an algebraic open disc, we prove that $V'$~is contained in a connected component~$D$ of~$V \setminus\{x\}$, where $x$ is the center of~$V$. Noting that $T := V\setminus D$ is an algebraic tube and writing $V\setminus V' = T \sqcup (D\setminus V')$, we are reduced to a previous case.

\item $V$ is an algebraic tube and $V'$ is an algebraic tube

Let~$x$ be the center of~$V$. If~$V'$ is contained in a connected component~$D$ of~$V \setminus\{x\}$, then, noting that $T := V\setminus D$ is an algebraic tube and writing $V\setminus V' = T \sqcup (D\setminus V')$, we are reduced to a previous case.

Otherwise, a boundary argument as above shows that $V$ and~$V'$ have the same center. It follows that $V \setminus V'$ is a disjoint union of finitely many connected components of~$V\setminus\{x\}$, hence a finite union of algebraic open discs.
\end{itemize}
\end{proof}

\begin{corollary}\label{cor:disjoint}
Let $\cV$ be a finite set of algebraic bricks of~$C^\an$. Then, there exists a finite set~$\cW$ of disjoint algebraic bricks of~$C^\an$ such that 
\[\bigcup_{V \in \cV} V =  \bigsqcup_{W \in \cW} W.\]
\qed
\end{corollary}

\begin{theorem}\label{thm:alg-brick}
Assume that~$C$ is proper and smooth. Then, there exists a finite partition~$\cV$ of~$C^\an$ into algebraic bricks. 
\end{theorem}
\begin{proof}
Since~$C$ is proper and smooth, $C^\an$ is proper and smooth too. In particular, it is compact.

By Proposition~\ref{prop:type3}, each point of type~3 in~$C^\an$ has a neighborhood that is an algebraic open annulus. By Propositions~\ref{prop:type1} and~\ref{prop:type4}, each point of type~1 or~4 in~$C^\an$ has a neighborhood that is an algebraic open disc. By Propositions~\ref{prop:type2generic} and~\ref{prop:type2branch}, each point of type~2 in~$C^\an$ has a neighborhood that is the union of an algebraic tube centered at that point and finitely many algebraic annuli (corresponding to the branches missing in the algebraic tube). By compactness of~$C^\an$, it follows that there exists a finite cover~$\cT$ of~$C^\an$ made of algebraic algebraic bricks. The result now follows from Corollary~\ref{cor:disjoint}.
\end{proof}

We now extend the result to arbitrary algebraic curves. By using the same kind of arguments as in the proofs of Lemmas~\ref{lem:intersectionalgebraic bricks} and~\ref{lem:complementalgebraic bricks}, it is not difficult to prove the following result.

\begin{lemma}\label{lem:algebraicbricksminuspoints}
Let~$V$ be an algebraic brick of~$C^\an$ and let~$F$ be a finite subset of~$C(k)$. Then $V\setminus F$ admits a finite partition into algebraic bricks.
\qed
\end{lemma}

\begin{corollary}\label{cor:algbricks}
There exists a finite subset~$E$ of~$C(k)$ such that $C^\an \setminus E$ admits a finite partition into algebraic bricks. 

\end{corollary}
\begin{proof}
We may identify~$C$ to an open subset of a projective curve~$\bar C$ over~$k$ such that $E_{1} := \bar{C} \setminus C$ is a finite subset of~$\bar C(k)$.

Denote by~$E_{2}$ the singular locus of~$\bar C$. Since~$k$ is algebraically closed and~$\bar C$ is reduced, $\bar C$ is generically smooth, hence~$E_{2}$ is also a finite subset of~$\bar C$. 
Let~$\tilde C$ be the normalisation of~$\bar C$ and denote by $n \colon \tilde{C} \to \bar C$ the corresponding morphism. Set $F := E_{1} \cup E_{2}$ and $\tilde F := n^{-1}(F)$. Then~$n$ induces an isomorphism $\tilde{C} \setminus \tilde F \to C\setminus T$, hence, to conclude, it is enough to find a finite subset~$\tilde E$ of~$\tilde C(k)$ containing~$\tilde F$ such that $\tilde C \setminus \tilde E$ admits a finite partition into algebraic bricks. Since~$\tilde C$ is smooth and projective, the result now follows from Theorem~\ref{thm:alg-brick} and Lemma~\ref{lem:algebraicbricksminuspoints}.
\end{proof}

\subsection{Analytic curves}

In this section, we give analogues of the results we obtained in the analytic setting. We will first handle the smooth case and then allow singularities.

\subsubsection{Quasi-smooth curves}

We first adapt the definition of bricks. We fix a quasi-smooth connected strictly $k$-affinoid curve $X = \cM(\cA)$.

\begin{definition}
An \emph{analytic open disc} is the data of a strict affinoid domain~$W$ of~$X$, an open subset~$V$ of~$W$ and a morphism $f \colon W \to \DD_{k}$ such that
\begin{enumerate}[$i)$]
\item $f(V)$ is an open disc with radius in~$|k^\times|$ in~$\DD_{k}$;
\item $f_{|V}$ induces an isomorphism onto its image.
\end{enumerate}

An \emph{analytic open annulus} is the data of a strict affinoid domain~$W$ of~$X$, an open subset~$V$ of~$W$ and a morphism $f \colon W \to \DD_{k}$ such that
\begin{enumerate}[$i)$]
\item $f(V)$ is an open annulus with radii in~$|k|$ in~$\DD_{k}$;
\item $f_{|V}$ induces an isomorphism onto its image.
\end{enumerate}

An \emph{analytic tube centered at a point~$x\in X$} is the data of a strict affinoid domain~$W$ of~$X$, an affinoid domain~$V$ of~$W$ with Shilov boundary~$\{x\}$ and a morphism $f \colon W \to \DD_{k}$ such that
\begin{enumerate}[$i)$]
\item $f^\an(x) = \eta_{0,1}$ and $f^\an(V)$ is equal to~$\DD_{k}$ possibly deprived of finitely many open unit discs;
\item for each connected component~$E$ of~$V\setminus\{x\}$, $f(E)$ is a connected component of $f(V)\setminus\{\eta_{0,1}\}$ and the morphism~$f$ induces an isomorphism between~$E$ and~$f(E)$.
\end{enumerate}

We call \emph{analytic brick} any triple $(W,V,f)$ of one of the three preceding sorts. 

When we speak about the topological properties of an analytic brick, we will mean the topological properties of~$V$.
\end{definition}

We have analogues of the results of Section~\ref{sec:localalgebraic} in the analytic setting. 

\begin{proposition}\label{prop:type1an}
Let $x\in X$ be a point of type~1. For each neighborhood~$U$ of~$x$ in~$X$, there exists an analytic open disc of~$X$ that contains~$x$ and is contained in~$U$. 
\end{proposition}
\begin{proof}
Since~$X$ is smooth at~$x$, there exists an open neighborhood~$U_{0}$ of~$x$ and an \'etale morphism $f_{0} \colon U_{0} \to \AA^{1,\an}_{k}$. The result now easily follows from the inverse function theorem.
\end{proof}

For the other types of points, as before, the key point is a density statement.

\begin{lemma}\label{lem:densityHxan}
Let $x \in X$ and $\eps>0$. For each $\alpha_{0} \in \sH(x)$, there exists an affinoid neighborhood~$V$ of~$x$ and a morphism 
\[f \colon V = \cM(\cA_{V}) \to \DD_{k} = \cM(k\{T\})\]
such that, if we denote by~$\alpha$ the image of the coordinate~$T$ on~$\DD_{k}$ by the map
\[k\{T\} \to \cA_{V} \to \sH(x),\] 
we have $|\alpha - \alpha_{0}| < \eps$.
\end{lemma}
\begin{proof}
By definition, the fraction field of~$\cA$ is dense in~$\sH(x)$. It follows that there exists $a,b \in \cA$ with $b(x)\ne 0$ such that $\big|\frac{a(x)}{b(x)} - \alpha_{0}\big| < \eps$. Let $V = \cM(\cA_{V})$ be an affinoid neighborhood of~$x$ in~$X$ such that~$b$ does not vanish on~$V$. Then the image of~$b$ in~$\cA_{V}$ is invertible, hence~$a/b$ defines an element~$\alpha$ of~$\cA_{V}$.

Let $c \in k^\times$ such that $|c|�\ge \|\alpha\|_{V}$. We have a bounded morphism $k\{|c|^{-1} T\} \to \cA_{V}$ sending~$T$ to~$\alpha$. Since multiplication by~$c$ induces an isomorphism between $k\{T\}$ and $k\{|c|^{-1} T\}$, the result follows.
\end{proof}

The next results are then proven exactly as in the algebraic case by using Lemma~\ref{lem:densityHxan} instead of Lemma~\ref{lem:densityHx}.

\begin{proposition}\label{prop:type4an}
Let $x\in X$ be a point of type~4. For each neighborhood~$U$ of~$x$ in~$X$, there exists an analytic open disc of~$X$ that contains~$x$ and is contained in~$U$. 
\qed
\end{proposition}

\begin{proposition}\label{prop:type3an}
Let $x\in X$ be a point of type~3. For each neighborhood~$U$ of~$x$ in~$X$, there exists an analytic open annulus of~$X$ that contains~$x$ and is contained in~$U$. 
\qed
\end{proposition}

\begin{proposition}\label{prop:type2generican}
Let $x\in X$ be a point of type~2. For each neighborhood~$U$ of~$x$ in~$X$, there exists an analytic tube $(W,V,f)$ of~$X$ centered at~$x$ that is contained in~$U$. 
\qed
\end{proposition}

As in the discussion following Proposition~\ref{prop:type2generic}, in the setting of Proposition~\ref{prop:type2generican}, $V$ is a tube and we have an associated reduction map $\rho_{V}\colon V \to \cC_{x}$ whose image~$\cU_{V}$ is an affine curve over~$\tilde k$. We fix a closed embedding $\iota_{V} \colon \cU_{V} \hookrightarrow \AA^r_{\tilde k}$.

\begin{lemma}\label{lem:residue_morphan}
There exist an affinoid domain~$Y$ of~$X$ containing~$x$ and morphisms $f_{1},\dotsc,f_{r} \colon Y \to \DD_{k}$ such that 
\begin{enumerate}[$i)$]
\item for each $i\in \{1,\dotsc,r\}$, $f_{i}(x) = \eta_{0,1}$;
\item the map 
\[z \in O^\an \cap V\setminus\{x\} \mapsto (\red(f_{1}^\an(z)),\dotsc,\red(f_{r}^\an(z))) \in \tilde{k}^r\] 
coincides with the restriction of $\iota_{V}\circ\rho_{V}$ to $O^\an \cap V\setminus\{x\}$.
\end{enumerate}
\qed
\end{lemma}

\begin{proposition}\label{prop:type2branchan}
Let $x\in X$ be a point of type~2. For each neighborhood~$U$ of~$x$ in~$X$ and each connected component~$B$ of~$U\setminus\{x\}$, there exists an open analytic annulus $(W,V,f)$ of~$X$ contained in~$B$ and whose closure contains~$x$. Moreover, given a finite set~$\cB'$ of connected components of~$U\setminus\{x\}$ not containing~$B$, we may ensure that $f^{-1}(f(V))$ does not meet any element of~$\cB'$. 
\qed
\end{proposition}

We can now adapt the arguments given in Section~\ref{sec:globaldecomposition} to obtain the following result.

\begin{theorem}\label{thm:an-brick}
There exists a finite subset~$E$ of~$X(k)$ such that $X\setminus E$ admits a finite partition into analytic bricks. \qed
\end{theorem}

The separation results are easier in the analytic setting since two distinct points of a curve may be put into disjoint affinoid domains. It follows that the analogue of Lemma~\ref{lem:separationpoints} holds. Using this remark, we may derive an analogue of Proposition~\ref{prop:separationbricks}.

\begin{proposition}\label{prop:separationbricksanalytic}
Let $\varphi\colon X \to X'$ be a morphism of smooth strictly $k$-affinoid curves, let~$U'$ be a connected analytic domain of~$X'$ whose boundary is a finite set of points of type~1 or~2 and let~$U$ be a connected component of $\varphi^{-1}(U')$ such that~$\varphi$ is finite at each point of~$U$. Then, there exist a Zariski-closed subset~$Z$ of~$X$, a finite set~$M$, for each $m\in M$, a finite set~$N_{m}$ and, for each~$n\in N_{m}$, an affinoid domain $W_{n}$ of~$X$, a morphism $f_{n}\colon W_{n} \to \DD^1_{k}$ and an open subset $A_{n}$ of~$\DD^{1,\an}_{k}$ that is either an open disc or an open annuli with radii in~$|k^\times|$ such that
\[U = \varphi^{-1}(U') \cap Z \cap \bigcup_{m\in M} \bigcap_{n\in N_{m}} f_{n}^{-1}(A_{n})\]
up to a finite set of $k$-rational points.
\qed
\end{proposition}

\begin{corollary}\label{cor:definableanbricks}
For each analytic brick~$(W,V,f)$ of~$X$, the set~$V(k)$ is a definable subset of~$X(k)$ and the map $V(k) \to k^\circ$ induced by~$f$ is definable.
\qed
\end{corollary}

\begin{corollary}\label{cor:definablereductionan}
For each point $x\in X$ of type~2 and each analytic tube~$(W,V,f)$ centered at~$x$, the map $V(k) \to \cC_{x}(\tilde k)$ induced by~$\rho_{V}$ is definable.
\qed
\end{corollary}

\subsubsection{Arbitrary curves}

We now adapt our definitions to be able to handle non-smooth curves as well. Let $X = \cM(\cA)$ be a reduced irreducible strictly $k$-affinoid curve.

\begin{definition}
Let~$X$ be a strictly $k$-affinoid space and~$Y$ be $k$-analytic space. A \emph{compactifiable rational map} $f \colon X \dasharrow Y$ is the data of 
\begin{enumerate}[$i)$]
\item a nowhere dense Zariski-closed subset~$S$ of~$X$;
\item a strictly $k$-affinoid space~$X'$;
\item a morphism $n \colon X' \to X$ such that the induced morphism $X' \setminus n^{-1}(S) \to X\setminus S$ is an isomorphism;
\item a morphism $f' \colon X' \to Y$.
\end{enumerate} 

We will call $X\setminus S$ the regularity locus of~$f$ and say that~$f$ is regular on an analytic domain~$U$ of~$X$ if $U \subseteq X\setminus S$. 

We will commonly use~$f$ as a shortcut for $f' \circ n^{-1}_{|X\setminus S}$ and use the following representation:
\[
\begin{tikzcd}
X' \ar[d,"n"'] \ar{r}{f'} & Y. \\
X \ar[ru, dashed,"f"']\\
\end{tikzcd}
\]
\end{definition}

\begin{remark}
In the previous setting, if~$Y$ is strictly $k$-affinoid, then the restriction of~$f$ to $X(k)\setminus S(k)$ is $\cL^\an$-definable.
\end{remark}

\begin{definition}
Let $X$ be a $k$-analytic curve. 

A \emph{rational analytic open disc} is the data of a strict affinoid domain~$W$ of~$X$, an open subset~$V$ of~$W$ and a compactifiable rational map $f \colon W \dasharrow \DD_{k}$ such that
\begin{enumerate}[$i)$]
\item $f$ is regular on~$V$; 
\item $f(V)$ is an open disc with radius in~$|k^\times|$ contained in~$\DD_{k}$;
\item $f_{|V}$ induces an isomorphism onto its image.
\end{enumerate}

A \emph{rational analytic open annulus} is the data of a strict affinoid domain~$W$ of~$X$, an open subset~$V$ of~$W$ and a compactifiable rational map $f \colon W \dasharrow \DD_{k}$ such that
\begin{enumerate}[$i)$]
\item $f$ is regular on~$V$;
\item $f(V)$ is an open annulus with radii in~$|k|$ contained in~$\DD_{k}$;
\item $f_{|V}$ induces an isomorphism onto its image.
\end{enumerate}

A \emph{rational analytic tube centered at a point~$x\in X$} is the data of a strict affinoid domain~$W$ of~$X$, an affinoid domain~$V$ of~$W$ with Shilov boundary~$\{x\}$ and a compactifiable rational map $f \colon W \dasharrow \DD_{k}$ such that
\begin{enumerate}[$i)$]
\item $f$ is regular on~$V$;
\item $f^\an(x) = \eta_{0,1}$ and $f^\an(V)$ is equal to~$\DD_{k}$ possibly deprived of finitely many open unit discs;
\item for each connected component~$E$ of~$V\setminus\{x\}$, $f(E)$ is a connected component of $f(V)\setminus\{\eta_{0,1}\}$ and the morphism~$f$ induces an isomorphism between~$E$ and~$f(E)$.
\end{enumerate}

We call \emph{rational analytic brick} any triple $(W,V,f)$ of one of the three preceding sorts.
\end{definition}

As in the discussions following Propositions~\ref{prop:type2generic} and~\ref{prop:type2generican}, in the setting of Proposition~\ref{prop:type2generican}, $V$ is a tube and we have an associated reduction map $\rho_{V}\colon V \to \cC_{x}$. 

The following results are easily derived from Corollaries~\ref{cor:definableanbricks} and~\ref{cor:definablereductionan}.

\begin{corollary}\label{cor:definableratanbricks}
For each rational analytic brick~$(W,V,f)$ of~$X$, the set~$V(k)$ is a definable subset of~$X(k)$ and the map $V(k) \to k^\circ$ induced by~$f$ is definable.
\qed
\end{corollary}

\begin{corollary}\label{cor:definablerareductionan}
For each point $x\in X$ of type~2 and each rational analytic tube~$(W,V,f)$ centered at~$x$, the map $V(k) \to \cC_{x}(\tilde k)$ induced by~$\rho_{V}$ is definable.
\qed
\end{corollary}

It is not difficult to check that if a strictly $k$-affinoid curve~$X$ admits a finite partition into analytic bricks, then, for each finite subset~$S$ of~$X(k)$, $X\setminus S$ admits a finite partition into rational analytic bricks. Using this kind of arguments, together with the fact that the normalisation of a strictly $k$-affinoid space is a strictly $k$-affinoid space isomorphic to the first one outside a finite number of $k$-rational points, we obtain the following analogue of Theorem~\ref{thm:an-brick}.

\begin{corollary}\label{thm:an-brick-singular}
There exists a finite subset~$E$ of~$X(k)$ such that $X\setminus E$ admits a finite partition into rational analytic bricks. 
\qed
\end{corollary}

\section{Definable analytic curves and morphisms}\label{sec:functorial}

Through this section we let $\cL$ denote either $\cL_\BB$ or $\cL_\BB^\an$.

\subsection{Facades}

Given $X$ a $k$-analytic curve, we use the notation $X(k)$ to denote the set of $k$-rational points of $X$. We will often identify $X^{(1)}$ and $X(k)$. 

Let $\mathbf{Def}_k(\cL)$ denote the category of $\cL$-definable sets with  $\cL$-definable maps as morphisms. 

\begin{definition}
An \emph{$\cL$-definable category of $k$-analytic curves} consists in the data of
\begin{enumerate}[$i)$]
\item a subcategory~$\mathcal{C}$ of the category of $k$-analytic curves;
\item a functor $j\colon \mathcal{C}\to \mathbf{Def}_k(\cL)$;
\item  for every object $X$ of $\mathcal{C}$, a bijection $j_X\colon X(k)\to j(X)$
\end{enumerate} 
such that, for any morphism $f\colon X\to Y$ of $\mathcal{C}$, the following diagram commutes:  
\[
\begin{tikzcd}
X(k) \ar{d}{f} \ar{r}{j_X} & j(X)\ar{d}{j(f)} \\
Y(k) \ar{r}{j_Y} & j(Y)\\
\end{tikzcd}
\texto{.}{-1.5pc}
\]
\end{definition}

We will often refer to an $\cL$-definable category of $k$-analytic curves abusively as~$\mathcal{C}$. Given an object~$X$ of~$\mathcal{C}$, we identify the set $X(k)$ with $j(X)$ and freely speak of $X(k)$ as an $\cL$-definable set. For example, a subset $A\subseteq X(k)$ is said to be $\cL$-definable if $j_X(A)$ is $\cL$-definable. 

Two main examples of definable categories will be considered in this article: 
\begin{enumerate}
\item the category of analytifications of algebraic curves with morphisms the analytifications of algebraic morphisms. Given $f^\an\colon X^\an\to Y^\an$, the corresponding $\cL_\BB$-definable map $j(f^\an) \colon j(X^\an) \to j(Y^\an)$ corresponds to the $\cL_{\mathrm{ring}}$-definable map induced by $f\colon X(k)\to Y(k)$ as explained in Convention~\ref{convention.alg}. Since in this situation $X^\an(k)=X(k)$, this provides the desired bijections.   
\item The category of strictly $k$-affinoid curves with morphisms the morphisms induced by bounded morphisms of the corresponding affinoid algebras. If $X=\mathcal{M}(\mathcal{A})$ is a strictly $k$-affinoid curve, to any Lipshitz presentation $f$ of $\mathcal{A}$ there is an associated $\cL_\BB^\an$-definable set $X_f(k)$ which is in bijection with $X(k)$ as explained in Convention \ref{convention.an}. The functoriality of such an association follows from Lemma \ref{lem:mapdistinguished}. 
\end{enumerate}



\begin{definition}\label{def:facade} Let $X$ be a $k$-analytic curve in an $\cL$-definable category of $k$-analytic curves~$\mathcal{C}$. An \emph{$\cL$-facade}~$\cS$ of~$X$ consists of the following data: 
\begin{enumerate}[$i)$]
\item a finite triangulation $S$ with an associated skeleton $\Gamma=(S,E)$ of $X$ and an associated retraction map~{$\tau\colon X\to\Gamma$}; 
\item for each edge~$I$ of~$E$, a pair $(V_I, f_I)$ such that 
\begin{enumerate}[(a)]
\item $V_I =\tau^{-1}(I)$,
\item $f_I\colon V_I\to \AA_k^{1,\an}$ is a morphism, $f_{I}(V_{I})$ is an open annulus and $f_I$ induces an isomorphism between~$V_{I}$ and its image,
\item $V_I(k)$ is an $\cL$-definable subset of $X(k)$,
\item the restriction $f_I\colon V_I(k)\to k$ is $\cL$-definable; 
\end{enumerate}
\item for each vertex~$x\in S^{(2)}$, an integer $m(x)\geqslant 0$, tuples $(W_x, f_x)$ and $(V_{x,i}, f_{x,i})$ with $1\leqslant i\leqslant m(x)$ such that 
\begin{enumerate}[(a)]
\item $\tau^{-1}(x)=W_x\sqcup \bigsqcup_{i=1}^{m(x)} V_{x,i}$,
\item $W_{x}$ is a tube centered at~$x$ (see Section~\ref{sec:residuecurve}),
\item $f_x\colon W_x\to \DD_k$ is a morphism such that $f_x(x) = \eta_{0,1}$, $f_x(W_x)$ is equal to~$\DD_{k}$ deprived of finitely many open unit discs, for each connected component~$C$ of~$W_x\setminus\{x\}$, $f_x(C)$ is an open unit disc and $f_x$ induces an isomorphism between~$C$ and its image, 
\item $f_{x,i}\colon V_{x,i}\to \AA^{1,\an}_k$ is a morphism, $f_{x,i}(V_{x,i})$ is an open unit disc and $f_{x,i}$ induces an isomorphism between $V_{x,i}$ and its image, 
\item $W_x(k)$ and $V_{x,i}(k)$ are $\cL$-definable subsets of $X(k)$,
\item the restrictions $f_x\colon W_x(k)\to k^\circ$, $f_{x,i}\colon V_{x,i}(k)\to k$ and $\rho_{W_x}\colon W_x(k)\to \cC_x(\tilde{k})$ are $\cL$-definable (see Section \ref{sec:residuecurve}).
\end{enumerate}
\end{enumerate}
\end{definition} 

We show the existence of $\cL$-facades in the following cases: 

\begin{theorem}\label{thm:facade_existence} Let $X$ be a $k$-analytic curve. 
\begin{enumerate}[$i)$]
\item If $X=Y^\an$ for an algebraic curve $Y$, then there is an $\cL_\BB$-facade of $X$.
\item If $X$ is a strictly $k$-affinoid curve, then there is an $\cL_\BB^\an$-facade of $X$.
\end{enumerate}
\end{theorem} 

\begin{proof} For part $i)$, by Corollary \ref{cor:algbricks} there is a finite triangulation $S$ of $X$ which induces a partition of $X$ into algebraic bricks. The data of this partition provides the ingredients of the required $\cL_\BB$-facade. That the restriction to $k$ of all this data is $\cL_\BB$-definable follows both from the definition of algebraic brick and Corollaries \ref{cor:definablealgbricks} and \ref{cor:definablereductionalg}. The proof of part $ii)$ is analogous: we obtain a partition of $X$ into rational analytic bricks by Corollary \ref{thm:an-brick-singular}, and the definability assumption follows from the definition of rational analytic brick together with Corollaries \ref{cor:definableratanbricks} and \ref{cor:definablerareductionan}.  
\end{proof}

\begin{remark}
Let~$X$ be a $k$-analytic curve that admits an $\cL$-facade associated to some triangulation~$S$ with associated skeleton~$\Gamma$ and retraction~$\tau$. For each subgraph~$\Gamma'$ of~$\Gamma$, the analytic curve $\tau^{-1}(\Gamma')$ admits an $\cL$-facade induced by that of~$X$.
\end{remark}

\begin{question} Are there other $\cL$-definable categories of $k$-analytic curves admitting $\cL$-facades? 
\end{question}

\begin{definition}\label{def:facade-ref} Let $\cS_1, \cS_2$ be facades of $X$, for $X$ a $k$-analytic curve. For $i=1,2$, we let $\Gamma_i=(S_i,E_i)$ denote the associated skeleton of $\cS_i$ and $\tau_i$ the corresponding retraction map. We say that $\cS_2$ is a \emph{refinement} of $\cS_1$ if 
\begin{enumerate}[$i)$]
\item the triangulation $S_2$ refines $S_1$;
\item if $I\in E_{2}$, one of the following holds 
\[
\begin{cases}
(f_{I})_{|V_I}=(f_{I'})_{|V_I} & \text{ if $I\subseteq \tau_1^{-1}(I')$ for some $I'\in E_{1}$}, \\
(f_{I})_{|V_I}=(f_x)_{|V_I} & \text{ if $I\subseteq W_x$ for some $x\in S_{1}^{(2)}$,} \\
(f_{I})_{|V_I}=(f_{x,i})_{|V_I} & \text{ if $I\subseteq V_{x,i}$ for some $x\in S_{1}^{(2)}$;} 
\end{cases} 
\]
\item if $x\in S_2^{(2)}\setminus S_1$ then $m(x)=0$; 
\item if $x\in S_2^{(2)}$, there exists an automorphism~$\sigma$ of~$\AA^1_{k}$ such that
\[
\begin{cases}
(f_{x})_{|\tau_2^{-1}(x)}=(\sigma^\an\circ f_I)_{|\tau_2^{-1}(x)} & \text{ if $x\in \tau_1^{-1}(I)$ for some $I\in E_{1}$,} \\
(f_{x})_{|\tau_2^{-1}(x)}=(\sigma^\an\circ f_y)_{|\tau_2^{-1}(x)} & \text{ if $x\in W_y$ for some $y\in S_{1}$,} \\
(f_{x})_{|\tau_2^{-1}(x)}=(\sigma^\an\circ f_{y,i})_{|\tau_2^{-1}(x)} & \text{ if $x\in V_{y,i}$ for some $y\in S_{1}$.} 
\end{cases} 
\]
\end{enumerate}
\end{definition} 

\begin{remark} \label{rem:facade-ref} Let $X$ be a $k$-analytic curve and $\cS_1$ be an $\cL$-facade of $X$. Let $S_2$ be a finite triangulation refining $S_1$. Then there is an $\cL$-facade $\cS_2$ with underlying triangulation $S_2$ which refines $\cS_1$. Indeed, one just defines the functions $f_I$ for $I\in E_2$ and $f_x$ for $x\in S_2^{(2)}$ as imposed by Definition \ref{def:facade-ref} with a suitable choice of an algebraic automorphism of $\AA_k^1$. 
\end{remark}

\begin{lemma}\label{lem:no_open_disc} Let $X$ be a $k$-analytic curve and $\cS_1$ be an $\cL$-facade of $X$. There is an $\cL$-facade $\cS_2$ refining $\cS_1$ such that for all $x\in S_2^{(2)}$, $m(x)=0$.  
\end{lemma}

\begin{proof} Let $x\in S_1^{(2)}$ with $m_{\cS_1}(x)>0$. By induction, it suffices to build an $\cL$-facade $\cS_2$ refining $\cS_1$ in which $m_{\cS_2}(x)<m_{\cS_1}(x)$ and $m_{\cS_2}(x')\leqslant m_{\cS_1}(x')$ for all $x'\in S_1^{(2)}\setminus\{x\}$ (note that by the definition of refinement, $m_{\cS_2}(y)=0$ for all $y\in S_2^{(2)}\setminus S_1^{(2)}$). Take any $y\in V_{x,1}^{(1)}$ and let $S_2$ be the triangulation $S_2:=S_1\cup \{y\}$. The associated skeleton has one new edge corresponding to the path from $x$ to $y$. Setting $f_I:= f_{x,1}$, gives an $\cL$-facade $\cS_2$ refining $\cS_1$ for which $m_{\cS_2}(x) = m_{\cS_1}(x)-1$. 
\end{proof}

\subsection{Definable set associated to a facade}

We will now associate an $\cL$-definable set to a given $\cL$-facade $\cS$ of a $k$-analytic curve $X$. Let $\Gamma=(S,E)$ be the skeleton associated to $\cS$. Note that the curve $X$ can be written as the following disjoint union
\[
X =  S^{(1)} \sqcup \bigsqcup_{I\in E} V_I\sqcup \bigsqcup_{x\in S^{(2)}} \left[W_x \sqcup\bigsqcup_{i=1}^{m(x)} V_{x,i} \right]. 
\]
We will associate a definable set to each part of the previous disjoint union. To that end, we need to introduce some notation concerning residue curves. Let $x\in S^{(2)}$, $W_x$ be its associated tube and $f_{x}$ be the corresponding morphism. As explained in Section~\ref{sec:residuecurve}, we have a map $\rho_{W_x}\colon W_x\to \cC_x$ with image a Zariski open subset $\cU_{W_x}$ of $\cC_x$. For simplicity, we will denote from now on $\cU_{W_x}$ by $\cU_x$. Let
\begin{equation}\label{eq:epsilon}\stepcounter{eqn}\tag{E\arabic{eqn}} 
\epsilon_x\colon W_x\setminus\{x\}\to \cU_x(\tilde{k})\times \DD_k \hspace{0.5cm} z\mapsto(\rho_{W_x}(z),f_x(z)).
\end{equation} 
The map $\epsilon_x$ is injective and its image is equal to the set
\begin{equation*}
Z_x:=\{(\alpha,y)\in \cU_x(\tilde{k})\times \DD_k :  \tilde{f}_x(\alpha)=\red(y)\}. 
\end{equation*}
We define the set $Z_x^\cS$ as the $\cL$-definable set given by:  
\begin{equation}\label{eq:Z}\stepcounter{eqn}\tag{E\arabic{eqn}} 
Z_x^\cS:=\{(\alpha,\eta_{a,r})\in \cU_x(\tilde{k})\times \DD_k^\Def :  \tilde{f}_x(\alpha)=\red(\eta_{a,r})\}.
\end{equation} 

\begin{definition}\label{def:defset} 
Let $X$ be a $k$-analytic curve and $\cS$ be an $\cL$-facade of $X$. Let $\Gamma=(S,E)$ be the skeleton associated to $\cS$. We define the $\cL$-definable set $X^{\cS}$ as 
\begin{equation}\label{eq:defset}\stepcounter{eqn}\tag{E\arabic{eqn}} 
X^{\cS}:=\bigsqcup_{x\in S^{(1)}} \eta_{1,0} \sqcup \bigsqcup_{I\in E} f_I(V_I)^{\Def}\sqcup \bigsqcup_{x\in S^{(2)}} \left[\eta_{0,1} \sqcup Z_x^\cS \sqcup \bigsqcup_{i=1}^{m(x)} f_{x,i}(V_{x,i})^{\Def} \right] 
\end{equation}
\end{definition}

\begin{definition} For $X$ as in the previous definition, we let $\varphi_\cS\colon X^{\D}\to X^\cS$ be the bijection given by: 
\begin{itemize}
\item for $x\in S^{(1)}$, $\varphi_\cS(x)$ is the corresponding copy of $\eta_{1,0}$, 
\item $(\varphi_\cS)_{| V_I^{\D}} = (f_I)_{|V_I^{\D}}$, 
\item for $x\in S^{(2)}$, $\varphi_\cS(x)$ is the corresponding copy of $\eta_{0,1}$, 
\item for $x\in S^{(2)}$, $(\varphi_\cS)_{|V_{x,i}^{\D}}= (f_{x,i})_{|V_{x,i}^{\D}}$, 
\item for $x\in S^{(2)}$, $(\varphi_\cS)_{|(W_x\setminus\{x\})^{\D}}= (\epsilon_x)_{|(W_x\setminus\{x\})^{\D}}$. 
\end{itemize}
\end{definition}

\begin{definition}\label{def:map} For $i=1,2$, let $X_i$ be a $k$-analytic curve and $\cS_i$ be an $\cL$-facade of $X_i$. Let $h\colon X_1\to X_2$ be a $k$-analytic morphism.  
We define the map $h_{\cS_1\cS_2}\colon X_1^{\cS_1}\to X_2^{\cS_2}$ as the unique map that makes the following diagram commute

\[
\begin{tikzcd}
X_1^{\D} \ar{d}{\varphi_{\cS_1}} \ar{r}{h} & X_2^\D \ar{d}{\varphi_{\cS_2}} \\
X_{1}^{\cS_1} \ar{r}{h_{\cS_1,\cS_2}} & X_{2}^{\cS_2}\\
\end{tikzcd}
\texto{.}{-1.5pc}
\]

We say that the pair $(\cS_1,\cS_2)$ is $h$-compatible if the underlying pair of triangulations $(S_1,S_2)$ is $h$-compatible. 
\end{definition}

The main theorem of this section is the following:

\begin{theorem}\label{thm:main} Let~$X_{1}$ and~$X_{2}$ be $k$-analytic curves and let $h\colon X_1\to X_2$ be a $k$-analytic morphism. Suppose that 
\begin{enumerate}[$i)$]
\item for every pair of $\cL$-facades $(\cS_1,\cS_2)$ of $X_1$ and $X_2$ there exists an $h$-compatible pair $(\cS_1',\cS_2')$ where $\cS'_1$ and $\cS'_{2}$ respectively refine $\cS_1$ and~$\cS_{2}$.
\item the restriction $h\colon X_1(k)\to X_2(k)$ is $\cL$-definable. 
\end{enumerate}
Then, for any $\cL$-facades $\cS_i$ of $X_i$, the map $h_{\cS_1\cS_2}\colon X_1^{\cS_1}\to X_2^{\cS_2}$ is $\cL$-definable. 
\end{theorem}

Before going into the proof of this theorem, let us show some instances where its hypotheses are satisfied. 

\begin{lemma}\label{lem:morphism_definable} Let $h\colon X_1\to X_2$ be a compactifiable morphism of nice curves of relative dimension~0. If the restriction $h\colon X_1(k)\to X_2(k)$ is $\cL$-definable then all hypotheses of Theorem \ref{thm:main} are satisfied. In particular: 
\begin{enumerate}[$i)$]
\item if $X_{1}$ and $X_{2}$ are analytifications of algebraic curves and $h$ is the analytification of an algebraic morphism of relative dimension~0 between the latter, or
\item if $X_1, X_2$ are strictly $k$-affinoid curves and $h$ is any morphism of relative dimension~0, 
\end{enumerate}
then all hypotheses of Theorem \ref{thm:main} are satisfied.
\end{lemma}
\begin{proof} This follows from Corollary \ref{cor:compact_compatible} and Remark \ref{rem:facade-ref}. Point~i) follows from Remark~\ref{rem:compactifiable} and point~ii) is obvious.
\end{proof}

\begin{question} Are there other $\cL$-definable categories of $k$-analytic curves whose objects and morphisms satisfy the hypotheses of Theorem \ref{thm:main}?
\end{question}

Let us now prove Theorem \ref{thm:main}. We need to show first the following lemma: 

\begin{lemma}\label{lem:facade-id} Let~$X$ be $k$-analytic curve. Let~$\cS_1,\cS_2$ be $\cL$-facades of~$X$ such that $\cS_2$ is a refinement of $\cS_1$. Then
the function~$\id_{\cS_1\cS_2}$ is $\cL_\BB$-definable. 
\end{lemma}

\begin{proof} We proceed by induction on the number of points in $S_2$ which are not in $S_1$, $n:=|S_2|-|S_1|$. If $n=0$, the map $\id_{\cS_1\cS_2}$ corresponds to the identity map and there is nothing to show. Suppose the result holds for $n$ and assume that $|S_2|-|S_1| = n+1$. By Lemma \ref{lem:ref-arity}, there is $x\in S_2\setminus S_1$ of arity $\leqslant 2$ and $S_1':=S_2\setminus\{x\}$ is a triangulation refining $S_1$. The facade $\cS_2$ induces by restriction a facade $\cS_1'$ refining $\cS_1$ with associated triangulation $S_1'$. By induction $\id_{\cS_1,\cS_1'}$ is definable and $\id_{\cS_1,\cS_2}=\id_{\cS_1',\cS_2}\circ \id_{\cS_1,\cS_1'}$. Hence, it suffices to show the result for $\cS_1=\cS_1'$ and $n=1$. Let $\{x\}=S_2\setminus S_1$. Note that it suffices to show definability up to a finite set, since functions with a finite domain are always definable. We split in cases:

\

\textbf{Case 1:} $x\in I$ for $I\in E_1$. Then there are $I_1,I_2\in E_2$ such that $I_1\cup\{x\}\cup I_2=I$. The map $\id_{\cS_1,\cS_2}$ is the identity map on $X^{\cS_1}\setminus f_{I}(V_{I})$. Let $\eta_{a,r}:=f_{I}(x)$. The set $f_I(V_I)$ partitions as
\[
f_{I_1}(V_{I_{1_1}})\sqcup f_{I_2}(V_{I_2}) \sqcup Y,
\] 
where $Y$ is the remaining part $f_I(V_I)\setminus (f_{I_1}(V_{I_{1_1}})\sqcup f_{I_2}(V_{I_2}))$. It suffices to show that $\id_{\cS_1,\cS_2}$ is definable in each piece. For $i=1,2$, the map $\id_{\cS_1,\cS_2}$ on $f_{I_i}(V_{I_{1_i}})^\D$ is the identity. For $Y$, let $\sigma$ be an automorphism of~$\AA^1_{k}$ such that $(f_{x})_{|W_x}=(\sigma^\an\circ f_I)_{|W_x}$. We have $(\id_{\cS_1,\cS_2})_{|Y\setminus\{\eta_{a,r}\}}={\sigma^\an}_{|Y\setminus\{\eta_{a,r}\}}$, which is an $\cL_\BB$-definable function (see Remark \ref{rem:auto}). 

\

\textbf{Case 2:} $x\in W_y$ for $y\in S_1^{(2)}$. Let $I\in E_2$ be the edge between $y$ and $x$. Among our data we have $(V_I,f_I)$, $(W_x,f_x)$ and $(W_y',f_y')$ such that 
\begin{enumerate}
\item $W_y$ is partitioned as $W_y'\sqcup V_I\sqcup W_x$;
\item $(f_y')_{|W_y'}=(f_y)_{|W_y'}$, $(f_I)_{|V_I}=(f_y)_{|V_I}$ and $(f_x)_{|W_x}=(\sigma^\an \circ f_y)_{|W_x}$, where~$\sigma$ is an automorphism of~$\AA^1_{k}$.  
\end{enumerate}
The map $\id_{\cS_1,\cS_2}$ is the identity map on $X^{\cS_1}\setminus Z_y^{\cS_1}$. Let $\eta_{a,r}:=f_y(x)$ and $\alpha_0\in \cU_{y}(\tilde{k})$ be such that $\tilde{f}_y(\alpha_0)=\red(\eta_{a,r})$. Since $m(x)=0$, the set $\eta_{0,1}\sqcup Z_y^{\cS_1}$ partitions as $(\eta_{0,1}\sqcup Z_y^{\cS_2})\sqcup Y_1\sqcup Y_2$, where 
\[
Y_1:=\{(\alpha_0,\eta_{b,s})\in Z_y^{\cS_1} :  D(b,s)\not\subseteq D(a,r)\} \text{ and } Y_2:=\{(\alpha_0,\eta_{b,s})\in Z_y^{\cS_1} :  D(b,s)\subseteq D(a,r)\}.
\]
Since $Y_1$ and $Y_2$ are definable, it suffices to show that the restriction of $\id_{\cS_1,\cS_2}$ to each piece is definable. The reader can check that 
\begin{itemize}
\item $(\id_{\cS_1,\cS_2})_{|(\eta_{0,1}\sqcup Z_y^{\cS_2})}$ is the identity;
\item $(\id_{\cS_1,\cS_2})_{|Y_1}$ is the projection to the second coordinate;
\item $(\id_{\cS_1,\cS_2})_{|Y_2}$ is the definable function
\[
(\alpha_0,\eta_{b,s})\mapsto (\widetilde{f_x}^{-1}(\res(\sigma^\an(\eta_{b,s}))), \sigma^\an(\eta_{b,s})), 
\]
\end{itemize} 
where $\widetilde{f_x}\colon \cC_x\to \PP_{\tilde{k}}^1$ is the morphism induced by $f_x=\sigma^\an\circ f_y$ at $x$, which in this case is an isomorphism. 

\

\textbf{Case 3:} $x\in V_{y,i}$ for $y\in S_1$ and $1\leqslant i\leqslant m(y)$. This case is analogous to Case 2. 
\end{proof}

\begin{proof}[Proof of Theorem \ref{thm:main}:] Let~$\cS_{1}$ and~$\cS_{2}$ be $\cL$-facades of~$X_{1}$ and~$X_{2}$ respectively.

\begin{claim}\label{claim:radial} We may suppose that
\begin{enumerate}
\item for every $x\in S_i^{(2)}$, $m_{\cS_i}(x)=0$;
\item the pair $(\cS_1,\cS_2)$ is $h$-compatible.  
\end{enumerate}
\end{claim} 

Suppose that the result holds for every pair of $\cL$-facades satisfying the claim. By Lemma \ref{lem:no_open_disc}, for $i=1,2$, there is a refinement $\cS_i'$ of $\cS_i$ such that for every $x\in (S_i')^{(2)}$, $m_{\cS_i'}(x)=0$. By assumption~$(i)$ of the theorem, there are $\cL$-facades $\cS_1^*$ and $\cS_{2}^*$ of $X_1$ and $X_{2}$ respectively, refining $\cS_1'$ and $\cS_{2}'$ respectively, such that the pair $(\cS_1^*,\cS_2^*)$ is $h$-compatible. Note that by definition of refinement, for $i=1,2$ and every $x\in (S_i^*)^{(2)}$, we have $m_{\cS_i^*}(x)=0$, thus $(\cS_1^*,\cS_2^*)$ satisfies the conditions of the claim. We have then the following commutative diagram: 
\[
\begin{tikzcd}
X_1^{\D} \ar[d,"\varphi_{\cS_1}"'] \arrow[dd, bend right=80, "\varphi_{\cS_1^*}"'] \ar{r}{h} & X_2^{\D}\ar{d}{\varphi_{\cS_2}} \arrow[dd, bend left=80, "\varphi_{\cS_2^*}"]\\
X_{1}^{\cS_1} \ar[d,"\id_{\cS_1\cS_1^*}"'] \ar{r}{h_{\cS_1,\cS_2}} &X_{2}^{\cS_2}\ar{d}{\id_{\cS_2\cS_2^*}}\\
X_{1}^{\cS_1^*} \ar{r}{h_{\cS_1^*,\cS_2^*}} &X_{2}^{\cS_2^*} 
\end{tikzcd}
\texto{.}{-1.5pc}
\]

By Lemma \ref{lem:facade-id}, $\id_{\cS_i\cS_i^*}$ is a definable bijection for $i=1,2$. Therefore, since $h_{\cS_1^*,\cS_2^*}$ is definable by assumption, $h_{\cS_1,\cS_2}$ is defined by 
\[
h_{\cS_1,\cS_2}(x)=  (\id_{\cS_2\cS_2^*}^{-1} h_{\cS_1^*,\cS_2^*} \id_{\cS_1\cS_1^*})(x), 
\]
which shows the claim. 

\

We suppose from now on that the pair $(\cS_1,\cS_2)$ satisfies conditions (1) and (2) of Claim \ref{claim:radial}. It suffices to show that the restriction of $h_{\cS_1,\cS_2}$ to 
\begin{enumerate}
\item $\varphi_{\cS_1}(\tau_1^{-1}(I)\cap X_1^{\D})$ for each $I\in E_1$ and 
\item $\varphi_{\cS_1}(\tau_1^{-1}(x) \cap X_1^{\D})$ for each $x\in S_1^{(2)}$  
\end{enumerate}
are $\cL$-definable. 
We split the argument in these two cases. 

\

\textbf{Case 1:} Let $I\in E_1$. By the definition of $X^{\cS_1}$, $\varphi_{\cS_1}(V_I^\D)$ is equal to $f_I(V_I^\D) \subseteq  \BB$. Let $A_1$ denote the set $f_I(V_I^\D)$. By $h$-compatibility, there is an interval $J\in E_2$ such that $h(V_I)=V_J$. Let $A_2$ denote the set $f_J(V_J^\D)\subseteq \BB$. We have the following commutative diagram 
\[
\begin{tikzcd}
V_I^\D \ar{d}{h} \ar{r}{f_I} & A_1 \ar{d}{h_{\cS_1\cS_2}} \\
V_J^\D \ar{r}{f_J} & A_2
\end{tikzcd}
\texto{.}{-1.5pc}
\]

Let $I':=f_I(I^{(2)})$, $J':=f_J(J^{(2)})$. Abusing notation, we identify $A_1(k)$ with the set $\{a\in k :  \eta_{a,0}\in A_1\}$. Note that the restriction of $h_{\cS_1\cS_2}$ to $A_1(k)$ is $\cL$-definable. Indeed, $(f_I)_{|V_I(k)}$ is $\cL$-definable and $\cL$-definably invertible, $h_{|X_1(k)}$ is $\cL$-definable, $(f_J)_{|V_J(k)}$ is $\cL$-definable and we have
\[
\forall a\in A_{1}(k),\ h_{\cS_1\cS_2}(a)=f_J(h(f_I^{-1}(a))). 
\]
Let us first show that $h_{\cS_1\cS_2}$ restricted to $I'$ is $\cL$-definable. For $i=1,2$, let $a_i\in k$ and $u_i, v_i\in |k|$ be such that 
\[
I'=\{\eta_{a_1,r} :  u_1<r<v_1\} \text{ and } J'=\{\eta_{a_2,r} :  u_2<s<v_2\}.
\] 
Note that both $I'$ and $J'$ are definable sets. By Lemma~\ref{lem:kpointsannuli}, for $r\in (u_{1},v_{1}) \cap |k^\times|$ and $s\in (u_{2},v_{2}) \cap |k^\times|$, we have
\begin{equation}\label{eq:equivcase1}\stepcounter{eqn}\tag{E\arabic{eqn}} 
h_{\cS_1\cS_2}(\eta_{a_1,r})=\eta_{a_2,s} \Leftrightarrow D(a_2,s)\setminus D^-(a_2,s)=\{h_{\cS_1\cS_2}(x)\in k  :  x\in D(a_1,r)\setminus D^-(a_1,r)\}.
\end{equation}
and, for $\eta_{b_{1},r}\in A_1\setminus I'$ and $\eta_{b_{2},s}\in A_2\setminus J'$, we have
\begin{equation}\label{eq:equivcase1.5}\stepcounter{eqn}\tag{E\arabic{eqn}} 
h_{\cS_1\cS_2}(\eta_{b_{1},r})=\eta_{b_{2},s} \Leftrightarrow D(b_{2},s)=\{h_{\cS_1\cS_2}(x)\in k  :  x\in D(b_{1},r)\}.
\end{equation}

\

\textbf{Case 2:} Let~$x\in S_1^{(2)}$ and~$y=h(x)$. Set~$B_1:=W_x^\D \setminus\{x\}$ and~$B_2:=W_y^\D \setminus\{y\}$. In this case, by part (1) of Claim \ref{claim:radial}, $(\varphi_{\cS_1})_{|B_1}$ (resp.~$(\varphi_{\cS_2})_{| B_2}$) is equal to~$\epsilon_x$ (resp~ $\epsilon_y$) as defined in \eqref{eq:epsilon} and we have the following commutative diagram
\[
\begin{tikzcd}
B_1\ar{d}{h} \ar{r}{\epsilon_x} & Z_x^{\cS_1} \ar{d}{h_{\cS_1\cS_2}} \\
B_2\ar{r}{\epsilon_y} & Z_y^{\cS_2}
\end{tikzcd}
\texto{.}{-1.5pc}
\]

Abusing notation, consider the set 
\begin{equation}\label{eq:Z(k)}\stepcounter{eqn}\tag{E\arabic{eqn}} 
Z_x^{\cS_1}(k):=\{(\alpha,a)\in \cU_x(\tilde{k})\times k  :  (\alpha,\eta_{a,0})\in Z_x^{\cS_1}\}.
\end{equation}
Let us first show that the restriction of~$h_{\cS_1\cS_2}$ to~$Z_x^{\cS_1}(k)$ is $\cL$-definable. By definition of facade, the restriction of $\epsilon_x$ to $W_x(k)$ is $\cL$-definable since both $f_x$ and $\rho_{W_x}$ restricted to $W_x(k)$ are $\cL$-definable. Recall that~$\eps_{x}$ induces a bijection from~$W_{x}(k)$ to~$Z_{x}^{\cS_{1}}(k)$ and that its inverse is automatically $\cL$-definable. For all $(\alpha,a)\in Z_x^{\cS_1}(k)$, we have
\[
h_{\cS_1\cS_2}(\alpha,a)= \epsilon_y(h(\epsilon_x^{-1}(\alpha,a))=(\rho_{W_y}(h(\epsilon_x^{-1}(\alpha,a))), f_y(h(\epsilon_x^{-1}(\alpha,a)))). 
\]
Then, as in Case 1, for any $(\alpha,\eta_{a,r})\in Z_x^{\cS_1}$, we have
\begin{equation}\label{eq:equivalencecase2}\stepcounter{eqn}\tag{E\arabic{eqn}} 
h_{\cS_1\cS_2}(\alpha,\eta_{a,r}) = (\beta,\eta_{b,s}) \Leftrightarrow 
\begin{cases}
\beta= \rho_{W_y}(h(\epsilon_x^{-1}(\alpha,y))) \text{ for all $y\in D(a,r)$ and }&\\

D(b,s) = \{h_{\cS_1\cS_2}(\alpha,y)\in \cU_x(\tilde k)\times k  :  y\in D(a,r)\}. &   
\end{cases}
\end{equation}
The equivalence \eqref{eq:equivalencecase2} holds both by Lemma \ref{lem:kpoints}, since for every $\alpha\in \cU_x(\tilde{k})$ the morphism 
\[
f_y\circ h\circ \epsilon_x^{-1}\colon (Z_x)^{\cS_1}_\alpha \to (Z_y)^{\cS_2}_{\beta},
\]
where $(Z_x)^{\cS_1}_\alpha$ and $(Z_y)^{\cS_2}_\beta$ denote the fiber of $Z_x^{\cS_1}$ at $\alpha$ and the fiber of $Z_y^{\cS_2}$ at $\beta$ respectively, is a morphism of open discs in $\AA_{k}^{1,\an}$, and by the fact that $\epsilon_x^{-1}(\alpha,a)=\epsilon_x^{-1}(\alpha,y)$ for all $y\in D(a,r)$. This completes Case 2 and the proof. 
\end{proof}

We finish this section with the following observation about the deformation retraction introduced in Section \ref{sec:prelim}. 

\begin{remark}\label{rem:deformation} Let $[0,1]_k$ denote $[0,1]\cap |k|$. Let $X$ be a $k$-analytic curve and $\cS$ be an $\cL$-facade of~$X$. The map $\nu_\cS$ making the following diagram commute 
\[
\begin{tikzcd}
{[0,1]_k}\times X^{\D} \ar{d}{\id\times \varphi_{\cS}} \ar{r}{\tau_S} & X{^\D} \ar{d}{\varphi_{\cS}} \\
{[0,1]_k}\times X^{\cS} \ar{r}{\nu_\cS} & X^{\cS} \\
\end{tikzcd}
\]
is $\cL$-definable. This follows essentially by the form of of $\tau_S$ given in Lemmas \ref{lem:taudisc} and \ref{lem:tauannulus}. We leave the details to the reader. 
\end{remark}

\section{Definable subsets of~\texorpdfstring{$\BB$}{B}}\label{sec:definable}

\subsection{Basic $\BB$-radial sets} Let $|k|^\infty$ denote the set $|k|\cup\{+\infty\}$ where $+\infty$ is a new formal element bigger than every element in $|k|$. In particular, this allows us to treat $k$ (resp. $\AA^{1,\an}_{k}$) as an open disc (resp. a Berkovich open disc) of radius $+\infty$ centered at some (any) point in $k$. 

Since $\BB$ is interpretable in the three-sorted language $\cL_3$, every $\cL_\BB^\an$-definable subset of $X\subseteq \BB$ is of the form $\bb(Y)$ for some $\cL_3^\an$-definable set $Y\subseteq k\times |k|$ (recall that~$\bb$ denotes the quotient map). Instead of writing $\bb(Y)$, we will directly use the $\eta$ notation and write $X=\{\eta_{x,r}\in \BB  :  \varphi(x,r)\}$  where $\varphi(x,r)$ is an $\cL_3^\an$-formula defining $Y$. We set $X\cap k := \{a \in k :  \eta_{a,0} \in X\}$.
 
The main objective of this section is to show that $\cL_\BB^\an$-definable subsets of $\BB$ are finite disjoint unions of the following $\cL_\BB$-definable basic blocks, which we call \emph{basic $\BB$-radial sets}. 

\begin{definition}\label{def:radialset} A subset $X\subseteq \BB$ is a \emph{basic $\BB$-radial set} if it is empty or equal to one of the following definable sets:

\noindent $\bullet$ \emph{Points:} for $a\in k$ and $s\in |k|$,  
\begin{equation}\label{eq:point}\tag{$R_0$}
X=\{\eta_{a,s}\}.
\end{equation}
$\bullet$ \emph{Branch segments:} for $a\in k$ and $s_1,s_2\in |k|^\infty$ such that $s_1<s_2$,  
\begin{equation}\label{eq:branch}\tag{$R_1$}
X=\{\eta_{x,r} :  x=a \wedge s_1<r<s_2\}.
\end{equation}
$\bullet$ \emph{Annulus cylinders:} for $a\in k$, $s_2\in |k|^\infty$, $s_1,\rho_1,\rho_2\in |k|$ and $g_1,g_2\in \QQ$,
\begin{equation}\label{eq:acylind1}\tag{$R_2$}
X=\left\{\eta_{x,r} \in \BB:\begin{array}{l} D(x,r)\subseteq D^-(a,s_2)\setminus D(a,s_1) \ \wedge \\
\rho_1|x-a|^{g_1} = \ r < |x-a|\end{array}\right\}. 
\end{equation}
or 
\begin{equation}\label{eq:acylind2}\tag{$R_3$}
X=\left\{\eta_{x,r}\in \BB :\begin{array}{l} D(x,r)\subseteq D^-(a,s_2)\setminus  D(a,s_1) \ \wedge \\
\rho_1|x-a|^{g_1} <  r  <  \rho_2|x-a|^{g_2} \leqslant |x-a| \end{array}\right\}.
\end{equation} 
$\bullet$ \emph{Closed disc cylinder:} for $a\in k$, $s,s_1,s_2\in |k|$ such that $s_1<s_2\leqslant s$, $n\in \NN$ and $b_1,\ldots, b_n\in k$ such that $|a-b_i|=s$, 
\begin{equation}\label{eq:cdcylin1}\tag{$R_4$}
X=\{\eta_{x,r} \in \BB :  D(x,r)\subseteq D(a,s)\setminus  (\bigcup_{i=1}^n D^-(b_i,s)) \ \wedge r = s_1  \}
\end{equation}
or 
\begin{equation}\label{eq:cdcylin2}\tag{$R_5$}
X=\{\eta_{x,r} \in \BB :  D(x,r)\subseteq D(a,s)\setminus  (\bigcup_{i=1}^n D^-(b_i,s)) \ \wedge \ s_1 < r < s_2 \}.
\end{equation}
$\bullet$ \emph{Open disc cylinders:} for $a\in k$, $s,s_{1},s_{2}\in |k|^\infty$ such that $s_1<s_2\leqslant s$ 
\begin{equation}\label{eq:odcylin1}\tag{$R_6$}
X=\{\eta_{x,r} \in \BB :  D(x,r)\subseteq D^-(a,s) \ \wedge s_1 = \ r \}
\end{equation}
or 
\begin{equation}\label{eq:odcylin2}\tag{$R_7$}
X=\{\eta_{x,r} \in \BB :  D(x,r)\subseteq D^-(a,s) \ \wedge \ s_1 < r < s_2 \}.
\end{equation}

A subset $X\subseteq \BB$ is a \emph{$\BB$-radial set} if it is a finite disjoint union of basic $\BB$-radial sets.
\end{definition}

\begin{remark} \hspace{1cm}
\begin{enumerate}[$i)$]
\item The only non-empty basic $\BB$-radial sets which are definably connected (\emph{i.e.}, not the union of two disjoint open non-empty definable subsets) are points, branch segments and open disc cylinders of type~$(R_{7})$ with $s=s_{2}$. Every non-empty basic $\BB$-radial set is infinite except for points.
\item Let $X$ be a non-empty annulus cylinder as in $(R_2)$ or $(R_3)$. By possibly changing $s_1$, one may always suppose that for every $s_2<r<s_1$, the set 
\[
X\cap \{\eta_{x,r}\in \BB  :  D(x,r)\subseteq D(a,r)\setminus D^-(a,r)\}
\]
is non-empty. 
\end{enumerate}
\end{remark}

\begin{lemma}\label{lem:booleancombination} Finite intersections and complements of basic $\BB$-radial sets are finite disjoint unions of basic $\BB$-radial sets. More generally, $\BB$-radial sets are stable under finite boolean combinations.
\end{lemma} 

\begin{proof}
The proof is a (tedious) case verification similar to the proof of Lemmas \ref{lem:intersectionalgebraic bricks} and \ref{lem:complementalgebraic bricks}. We do one case to let the reader have an idea of the kind of decomposition one obtains. So suppose $X$~is an annulus cylinder as defined in (R2) 
\begin{equation*}
X=\left\{\eta_{x,r} \in \BB:\begin{array}{l} D(x,r)\subseteq D^-(a,s_2)\setminus D(a,s_1) \ \wedge \\
\rho_1|x-a|^{g_1} = \ r < |x-a|\end{array}\right\}, 
\end{equation*}
with $s_1<r<s_2$. In Figure \ref{fig1} we depict $X$ as the bold black rays and the corresponding collection of $\BB$-radial sets whose union equals the complement $\BB\setminus X$.  
\begin{figure}[!htb]
        \center{\includegraphics[scale=0.8]
        {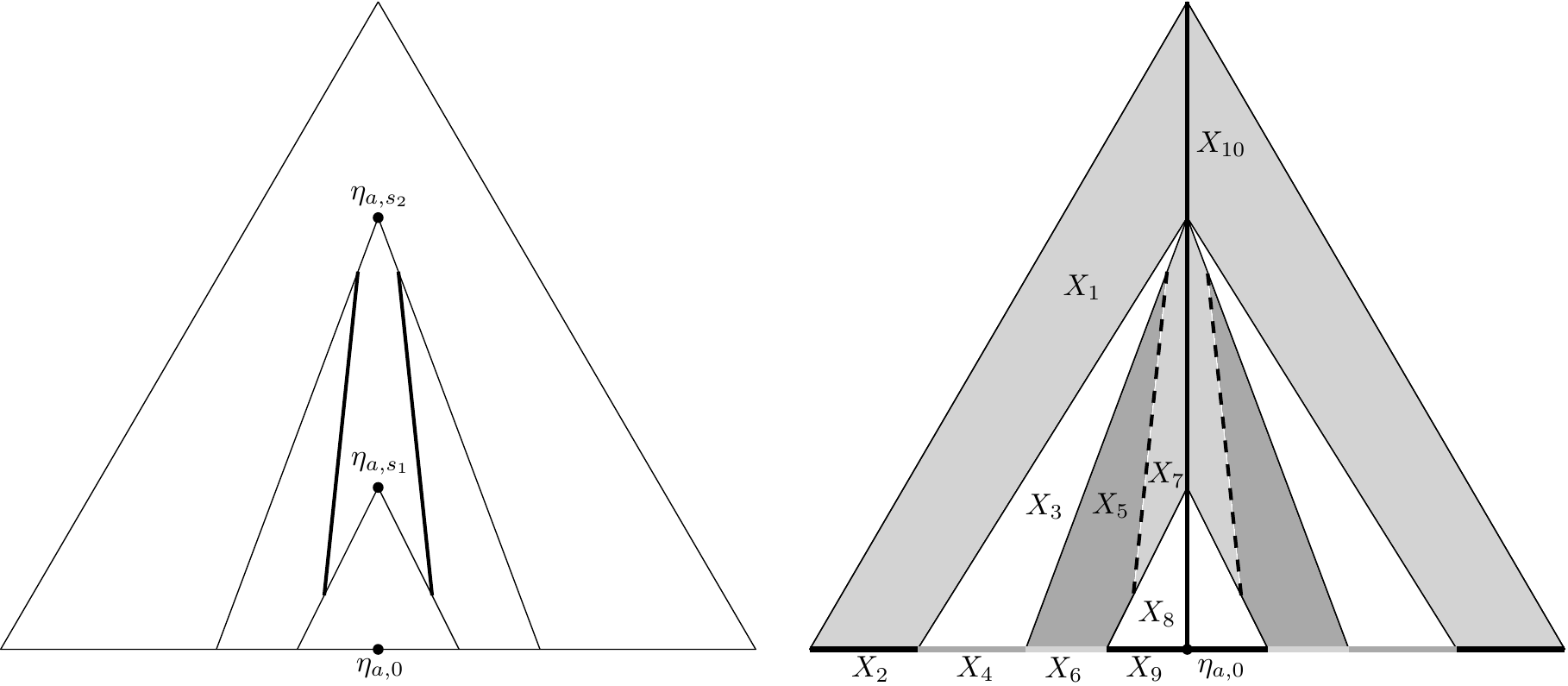}}
        \caption{\label{fig1} $\BB$-radial decomposition of $\BB\setminus X$}
      \end{figure}
Formally, the corresponding collection is given by the following $\BB$-radial sets which we enumerate which we enumerate in decreasing distance to $\eta_{a,0}$:
\begin{itemize}
\item two annulus cylinders
\[
X_1=\{\eta_{x,r} \in \BB :  D(x,r)\subseteq D^-(a,\infty)\setminus D(a,s_2) \ \wedge \ 0<r < |x-a| \},
\]
\[
X_2=\{\eta_{x,r} \in \BB :  D(x,r)\subseteq D^-(a,\infty)\setminus D(a,s_2) \ \wedge r=0 \},
\]
\item two closed disc cylinders
\[
X_3=\{\eta_{x,r} \in \BB :  D(x,r)\subseteq D(a,s_2)\setminus  D^-(a,s_2) \ \wedge 0<r<s_2 \},
\]
\[
X_4=\{\eta_{x,r} \in \BB :  D(x,r)\subseteq D(a,s_2)\setminus  D^-(a,s_2) \ \wedge r = 0  \} \text{ and }
\]
\item three annulus cylinders
\[
X_5=\left\{\eta_{x,r}\in \BB :\begin{array}{l} D(x,r)\subseteq D^-(a,s_2)\setminus  D(a,s_1) \ \wedge \\
0 <  r  <  \rho_1|x-a|^{g_1} \leqslant |x-a| \end{array}\right\} \text{ and }
\]
\[
X_6=\left\{\eta_{x,r}\in \BB :\begin{array}{l} D(x,r)\subseteq D^-(a,s_2)\setminus  D(a,s_1) \ \wedge \\
0 =  r  <  |x-a| \end{array}\right\},
\]
\[
X_7=\left\{\eta_{x,r}\in \BB :\begin{array}{l} D(x,r)\subseteq D^-(a,s_2)\setminus  D(a,s_1) \ \wedge \\
\rho_1|x-a|^{g_1} <  r  <  |x-a| \end{array}\right\}, 
\]
\item two closed disc cylinders 
\begin{equation*}
X_8=\{\eta_{x,r} \in \BB :  D(x,r)\subseteq D(a,s_1) \ \wedge r = 0  \} \text{ and }
\end{equation*}
\begin{equation*}
X_{9}=\{\eta_{x,r} \in \BB :  D(x,r)\subseteq D(a,s_1)\ \wedge \ 0 < r < s_1 \}.
\end{equation*}
\item one branch segment 
\[
X_{10}=\{\eta_{x,r} :  x=a \wedge 0<r<\infty\}.
\]
\item one final point $\eta_{a,0}$. 
\end{itemize}
\end{proof}

It will be useful to consider other kinds of subsets definable subsets of~$\BB$.

\begin{definition}\label{def:Bbrick} A subset $X\subseteq \BB$ is a \emph{$\BB$-brick} if it is empty or equal to one of the following definable sets:

\noindent $\bullet$ \emph{$k$-rational points:} for $a\in k$,  
\begin{equation}\label{eq:krationalpoint}\tag{$B_0$}
X=\{a\}.
\end{equation}
$\bullet$ \emph{Open discs:} for $a\in k$, $s\in |k|^\infty \setminus\{0\}$,
\begin{equation}\label{eq:disc}\tag{$B_1$}
X = \DD^-(a,s)^\D.
\end{equation}
$\bullet$ \emph{Open annuli:} for $a\in k$, $s_1\in |k|$ and $s_2\in |k|^\infty$,
\begin{equation}\label{eq:annulus}\tag{$B_2$}
X = \DD^-(a,s_{2})^\D \setminus \DD(a,s_{1})^\D.
\end{equation}
$\bullet$ \emph{Tubes:} for $a\in k$, $s\in |k^\times|$, $n\in \NN$ and $b_{1},\dotsc,b_{n} \in k$ such that $|a-b_{i}|=s$,
\begin{equation}\label{eq:tube}\tag{$B_3$}
X = \DD(a,s)^\D \setminus \bigcup_{i=1}^n \DD^-(b_{i},s)^\D.
\end{equation}
If a $\BB$-brick $X$ is a $k$-rational point or an open disc, the skeleton of~$X$ is $\Gamma_{X} := \emptyset$. If $X$ is an open annulus as above, the skeleton of~$X$ is 
$\Gamma_{X} := \{\eta_{a,r} \in \BB :  s_{1}<r< s_{2}\}$. If $X$ is a tube as above, the skeleton of~$X$ is 
$\Gamma_{X} := \{\eta_{a,s}\}$.
\end{definition}

\begin{remark}\label{rem:skeleton}
Let~$X$ be a brick. Then $X$ satisfies the following properties:
\begin{enumerate}[$i)$]
\item The sets~$X$ and $X\setminus \Gamma_{X}$ are $\BB$-radial.
\item If $\eta_{a,r} \in X \setminus \Gamma_{X}$, then $D(a,r) \subseteq X \cap k$.
\end{enumerate}
\end{remark}

\begin{lemma}\label{lem:distance}
Let~$X \subseteq \BB$ be a $\BB$-brick. Let $b\in (\BB \setminus X) \cap k$.
\begin{enumerate}[$i)$]
\item If~$X$ is a $k$-rational point, an open disc or a tube, then
\[\forall x,x' \in X\cap k,\ |x-b| = |x'-b|.\]
\item If~$X$ is an open annulus, there exist disjoint subsets $U_{X}$ and $V_{X}$ of~$\BB$ with $U_{X} \cup V_{X} = \BB \setminus X$ and $b_{X} \in U_{X}\cap k$ such that, if $b\in U_{X}\cap k$, then
\[\forall x\in X\cap k,\ |x-b| = |x-b_{X}|\]
and, if $b\in V_{X}\cap k$, then
\[\forall x\in X\cap k,\ |x-b| = |b-b_{X}|.\]
\end{enumerate}
\end{lemma}
\begin{proof} Part $(i)$ is clear. For part $(ii)$, suppose $X = \DD^-(a,s_{2})^\D \setminus \DD(a,s_{1})^\D$ and set~$U_{X}:= D(a,s_{1})$, $V_{X}:= \BB \setminus D^-(a,s_{2})$ and $b_{X}:= a$. If~$b\in U_X$, then~$|b-b_X|\leqslant s_1$. Therefore, $|x-b|=|x-b_X|$ for all~$x\in X\cap k$ since $|x-b_X|>s_1$. If $b\in V_X$, then $|x-b_X|<s_2\leqslant |x-b|$ for all~$x\in X\cap k$, which shows that $|x-b|=|b-b_X|$. 
\end{proof}

\begin{definition}\label{def:BSwisscheese} Let $X\subseteq k$ be a Swiss cheese:
\[X = D^\eps(a,r) \setminus \bigcup_{i=1}^n D^{\eps_{i}}(a_{i},r_{i}),\]
with $a\in k$, $r\in |k|^\infty$, $\eps \in\{-,\emptyset\}$ with $(r,\eps) \ne (+\infty,\emptyset)$, $n\in \NN$, $a_{1},\dotsc,a_{n} \in k$, $r_{1},\dotsc,r_{n}\in |k|$, $\eps_{1},\dotsc,\eps_{n} \in \{-,\emptyset\}$.

We set 
\[X^\BB := \DD^\eps(a,r)^\D \setminus \bigcup_{i=1}^n \DD^{\eps_{i}}(a_{i},r_{i})^\D.\]
A subset of~$\BB$ of the latter form will be called a \emph{$\BB$-Swiss cheese}.
\end{definition}

Note that $X^\BB \cap k = X$ and that the association $X\mapsto X^\BB$ is well-defined since the discs appearing in the definition of~$X$ are unique.

\begin{lemma}\label{lem:Swisscheese1}
Let~$\cA$ be a finite partition of~$k$ into Swiss cheeses. Then $\cA^\BB := \{A^\BB  :  A\in \cA\}$ is a finite partition of~$\BB$ into $\BB$-Swiss cheeses.
\end{lemma}
\begin{proof}
It is easy to check that, if $A_{1}$ and~$A_{2}$ are disjoint Swiss cheeses, then $A_{1}^\BB$ and~$A_{2}^\BB$ are disjoint, which shows that the elements of~$\cA^\BB$ are disjoint.

To show that~$\cA^\BB$ covers~$\BB$, let $\eta_{a,r} \in \BB$. Let $\cKK$ be an elementary extension of~$\mathbf{k}$ and $x \in L$ whose image by the projection $\AA^{1,\an}_{L} \to \AA^{1,\an}_{k}$ is~$\eta_{a,r}$. The collection $\cA(\cKK) := \{A(\cKK)  :  A\in \cA\}$ covers $L$, hence there exists $A\in \cA$ such that $x\in A(\cKK)$. Since, for every $P\in k[T]$, we have $|P(x)| = |P(\eta_{a,r})|$, it follows that $\eta_{a,r} \in A^\BB$.
\end{proof}

\begin{lemma}\label{lem:Swisscheese2}
Any $\BB$-Swiss cheese admits a finite partition into $\BB$-bricks. \qed
\end{lemma}

\begin{corollary}\label{cor:Bbrick}
Let~$\cA$ be a finite partition of~$k$ into $\cL$-definable subsets. Then, there exists a finite partition~$\cB$ of~$\BB$ into $\BB$-bricks such that, for each $B\in\cB$, there exists $A\in \cA$ with $B\cap k\subseteq A$. \qed
\end{corollary}

\subsection{Characterization of definable subsets of $\BB$}

We will need the following result about valuations of $\cL^\an$-terms: 

\begin{proposition}\label{prop:terms} Let $x$ be a $\mathbf{VF}$-variable, $t_1(x),\ldots, t_n(x)$ be $\cL^\an$-terms and $g_1,\ldots,g_n$ be rational numbers. Then there is a finite partition $\mathcal{B}$ of $\BB$ into $\BB$-bricks such that, for every $B\in \mathcal{B}$, there are $\rho_B\in |k|$, $h_B\in \QQ$ and $a_B\in k$ satisfying the following property: for all $x\in B\cap k$,
\[
\prod_{i=1}^n |t_i(x)|^{g_i}= \rho_{B}|x-a_B|^{h_B}. 
\] 
\end{proposition}
\begin{proof} By \cite[Theorem 5.5.3]{CL}, there is a definable partition $\mathcal{A}$ of $k$ such that for each $A\in\mathcal{A}$, 
\[
(\forall x\in A) (t_i(x)=P_{i,A}(x)\,u_{iA}(x)),
\]
where $P_{iA}\in k(X)$ is a rational function without poles in $A$ and $u_{iA}(x)$ is an $\cL^\an$-term such that $|u_{iA}(x)|=r_A$ for all $x\in A$ with $r_A\in |k^\times|$. Without loss of generality we may suppose $r_A=1$. Let~$Z$ be the finite set containing all zeros and poles of the non-zero $P_{iA}$'s. 

Let~$\cA'$ be a partition of~$k$ into definable sets that refines both~$\cA$ and the partition $\{\{b\} :  b\in Z\}\cup \{k\setminus Z\}$. By Corollary~\ref{cor:Bbrick}, there exists a finite partition~$\cB$ of~$\BB$ into $\BB$-bricks such that, for each $B\in\cB$, there exists $A\in \cA'$ with $B\cap k\subseteq A$. We claim that $\mathcal{B}$ has all required properties. 

Let $B\in \mathcal{B}$. By construction, there exists a unique $A\in \cA$ such that $B\cap k \subseteq A$. If~$B$ is a $k$-rational point, the result of the statement is obvious, so we assume otherwise. In this case, $B \cap Z = \emptyset$. Let $i \in \{1,\dotsc,n\}$. For each $x\in B \cap k$, we have
\[
|t_i(x)|=|P_{iA}(x)||u_{iA}(x)|=|P_{iA}(x)|=\rho_{iA}\prod_{b\in Z}|x-b|^{n_b}, 
\] 
where $\rho_{iA}\in |k|$ and $n_b\in \ZZ$. If~$B$ is an open disc or a tube, then, by part $(i)$ of Lemma~\ref{lem:distance}, $|t_{i}|$ is constant on~$B(k)$, and the result holds. If~$B$ is an open annulus, then, by part $(ii)$ of Lemma~\ref{lem:distance}, there are subsets $U_B, V_B\subseteq \BB$ and $b_B\in U_B\cap k$ such that for all $x\in B\cap k$
\begin{align*}
|t_i(x)| 	& = \rho_{iA}\prod_{b\in Z \cap V_{B}}|x-b|^{n_b} \prod_{b\in Z\cap U_{B}} |x-b|^{n_b}\\
							 	& = \rho_{iA}\prod_{b\in Z \cap V_{B}}|b-b_{B}|^{n_b} \prod_{b\in Z\cap U_{B}} |x-b_B|^{n_{b}}.
\end{align*} 
\end{proof}

\begin{theorem}\label{thm:defsets} Every $\cL_\BB^\an$-definable set $X\subseteq \BB$ is a finite disjoint union of basic $\BB$-radial sets. 
\end{theorem}
\begin{proof} By quantifier elimination, we may assume that $X$ is of the form
\[
X:=\{\eta_{x,r}\in \BB  :  \bigvee_{i\in I}\bigwedge_{j\in J} \varphi_{ij}(x,r)\}
\]
with $I, J$ finite sets and each $\varphi_{ij}(x,r)$ an atomic or negated atomic $\cL_3^\an$-formula. By commuting over unions (i.e. $\bb$ commutes over unions) we have that 
\[
X=\bigcup_{i\in I} \{\eta_{x,r}\in \BB  :  \bigwedge_{j\in J} \varphi_{ij}(x,r)\}
\]
and therefore (by Lemma~\ref{lem:booleancombination}) it suffices to show the result for $\varphi(x,r)$ a formula of the form $\bigwedge_{j\in J} \varphi_{j}(x,r)$ with $\varphi_{j}(x,r)$ an atomic or negated atomic $\cL_3^\an$-formula (note that $\bb$ does not necessarily commute over intersections and complements). 

By standard manipulations, we may suppose that every atomic formula $\varphi_j(x,r)$ is of the form 
\[
\prod_{i\in I_j} |t_{ji}(x)|^{g_{ji}}r^{h_j} \bowtie_j s_j,
\]
where $I_j$ is a finite set, $t_{ji}$ is an $\cL^\an$-term, $g_{ji}, h_j\in \QQ$, $s_j\in |k|$ and $\bowtie_j$ is an element of $\{=,<,>\}$. 

By a repeated use of Proposition \ref{prop:terms} and taking common refinements of finite partitions, there is a finite partition $\mathcal{B}$ of $\BB$ into $\BB$-bricks such that, for each $B\in \mathcal{B}$, we have, using Remark~\ref{rem:skeleton},
\[
\eta_{x,r}\in B \setminus \Gamma_{B}\Rightarrow \bigwedge_{j\in J}\varphi_j(x,r) \leftrightarrow \bigwedge_{j\in J} \rho_{jB}|x-a_B|^{g_{jB}}\bowtie_{jB} r^{h_{jB}}
\]
where $a_B\in k$, $\rho_{jB}\in |k|$, $g_{jB}\in \QQ$ and $h_{jB}\in\{0,1\}$. It suffices to show, by Lemma~\ref{lem:booleancombination}, that both $(B\setminus \Gamma_B)\cap X$ and $\Gamma_B\cap X$ are $\BB$-radial for each $B\in \mathcal{B}$. We proceed by cases.  

If there is $j'\in J$ such that $h_{j'B}=0$, then the condition $\rho_{j'B}|x-a_B|^{g_{j'B}}\bowtie_{j'B} 1$ only imposes a restriction on the $x$ variable. Hence, by applying Corollary~\ref{cor:Bbrick} and possibly refining~$\mathcal{B}$, we may suppose that, for each $B\in \mathcal{B}$, we have 
 
\[
\eta_{x,r}\in B\setminus \Gamma_{B} \Rightarrow \bigwedge_{j\in J}\varphi_j(x,r) \leftrightarrow \bigwedge_{j\in J\setminus \{j'\}} \rho_{jB}|x-a_B|^{g_{jB}}\bowtie_{jB} r^{h_{jB}}. 
\]
Thus, by induction on the size of $J$, we may assume that $h_{jB}=1$ for all $j\in J$. 

Similarly, if there are $j,j'\in J$ such that both $\bowtie_{jB}$ and $\bowtie_{j'B}$ are ``$=$'', this imposes a restriction on the variable $x$ by the condition 
\[
\rho_{jB}|x-a_B|^{g_{jB}}=\rho_{j'B}|x-a_B|^{g_{j'B}}.
\] 
Thus, by further refining $\mathcal{B}$ we may assume $\bowtie_j$ is ``$=$'' for at most one element $j\in J$. 

Suppose first that there is $j\in J$ for which $\bowtie_{jB}$ is ``$=$''. By possibly further refining $\mathcal{B}$, either $(B\setminus \Gamma_{B})\cap X$ is empty or 
\[
(B\setminus \Gamma_{B})\cap X= (B\setminus \Gamma_{B})\cap \{\eta_{x,r} :  \rho_{jB}|x-a_B|^{g_{jB}}=r\},
\] 
which is an intersection of $\BB$-radial sets (see part $(i)$ of Remark \ref{rem:skeleton}). 

We may thus suppose that $\bowtie_{jB}\in\{<,>\}$ for all $j\in J$. For $\square\in\{<,>\}$, let 
\[
J_\square=\{j\in J :  \ \bowtie_{jB} \text{ is ``$\square$'' }\},
\] 
so that $J=J_<\sqcup J_>$. By possibly refining $\mathcal{B}$, we may suppose that for all distinct $j,j'\in J_\square$
\[
\eta_{x,r}\in B\setminus \Gamma_{B}\Rightarrow \rho_{jB}|x-a_B|^{g_{jB}} \ \square' \ \rho_{j'B}|x-a_B|^{g_{j'B}}, 
\]
for $\square'\in\{<,>\}$. Hence, there are $j_1,j_2\in J$ such that 
\begin{align*}
\eta_{x,r}\in B\setminus \Gamma_{B}\Rightarrow \rho_{j_1B}|x-a_B|^{g_{j_1B}} =\max_{j\in J_<}(\rho_{jB}|x-a_B|^{g_{jB}}) & \text{ and } \\
\eta_{x,r}\in B\setminus \Gamma_{B}\Rightarrow \rho_{j_2B}|x-a_B|^{g_{j_2B}} =\min_{j\in J_>}(\rho_{jB}|x-a_B|^{g_{jB}}),
\end{align*}
which shows that 
\[
(B\setminus \Gamma_{B})\cap X= (B\setminus \Gamma_{B})\cap \{\eta_{x,r}\in \BB  :  \rho_{j_1B}|x-a_B|^{g_{j_1B}}<r<\rho_{j_2B}|x-a_B|^{g_{j_2B}}\},
\] 
which is an intersection of $\BB$-radial sets. 

The previous argument shows the result for all $B\setminus \Gamma_{B}$ with $B\in\mathcal{B}$, so it remains to show the result for $\Gamma_{B}\cap X$ with $B\in\mathcal{B}$. If $B$ is a $k$-rational point, an open disc or a tube, there is nothing to show. Suppose that $B$ is an open annulus. Then~$\Gamma_{B}$ is definably isomorphic to a definable interval of~$|k|$. Since the induced structure on $|k|$ is o-minimal, any definable subset of $\Gamma_{B}$ is a finite disjoint union of points and branch segments as in \eqref{eq:branch}. This shows that $\Gamma_{B}\cap X$ is a $\BB$-radial. 
\end{proof}

\section{Definable subsets of curves}\label{sec:def_of_curves}

Let $X$ be a $k$-analytic curve over $k$ and $\cS$ be a facade of $X$. Our goal in this section is to describe the definable subsets of $X^\cS$. We will provide such a description by reducing to a description of the definable subsets of $\BB$. We will first need some preliminary lemmas.

\subsection{Consequences of $\BB$-radiality}

We gather in this section some consequences of Theorem \ref{thm:defsets} and some auxiliary lemmas that will be later used. In model-theoretic terms, we will show orthogonality between residue curves and the sort $\BB$. This will essentially follow from the orthogonality between the value group and the residue field which we somehow recall for the reader's convenience.  

\begin{lemma}\label{lem:ort} Let~$C$ be an algebraic curve over~$\tilde k$. Let $X\subseteq |k|\times C(\tilde{k})$ be an $\cL_\mathbb{B}$-definable set such that, for each $s\in |k|$, the fiber $X_{s} := \{x \in C(\tilde k)  :  (s,x) \in X\}$ is finite. Then $\pi_2(X)$ is finite, where $\pi_2\colon X\to C(\tilde{k})$ denotes the projection to the second coordinate.  
\end{lemma}
\begin{proof} 
Set $A := \pi_{2}(X)$.
There exists an elementary extension~$\mathbf{k}'$ of~$\mathbf{k}$ such that $\card(|k'|)<\card(\tilde{k'})$. By Lemma \ref{lem:min}, for each $s\in |k'|$, the fiber~$X_{s}$ is finite, hence 
\[\card(A(\mathbf{k}')) \le \card(|k'|) < \card(\tilde{k'}) = \card(C(\tilde k')).\]
It follows that $A(\mathbf{k}')$ is a definable subset of~$C(\tilde k')$ that is not cofinite, hence it is finite.
\end{proof}

\begin{lemma}\label{lem:EQinfty}(Elimination of $\exists^{\infty}$ for $\BB$) 
Let $S$ be a definable set and let $X\subseteq S\times \mathbb{B}$ be a definable set such that, for each $s\in S$, the fiber $X_s:=\{x\in \mathbb{B} :  (s,x)\in X\}$ is finite. Then there is $N$ such that for all $s\in S$, the fiber $X_s$ has less than $N$ elements.  
\end{lemma} 
\begin{proof}
Suppose not. Then, by compactness, there is an elementary extension $\mathbf{k}\prec \mathbf{k}'$ and $s\in S(\mathbf{k}')$ such that $X_s$ is infinite. By Proposition \ref{thm:defsets}, $X_s$ must contain a non-empty basic $\BB$-radial set $D$ which is not a point. Let $D$ be defined by a formula $\varphi(x,s,a)$ with $a$ some parameters in $k'$. We have that 
\[
\mathbf{k}'\models (\exists s)(\exists a)\big((\exists x_1 \exists x_2\in \mathbb{B})(\varphi(x_1,s,a)\wedge \varphi(x_2,s,a)\wedge x_1\neq x_2) \wedge (\forall y\in \BB)(\varphi(y,s,a)\to y\in X_s)\big). 
\]
This implies that there are $s'\in S(\mathbf{k})$ and $a'$ in $k$ such that $X_{s'}$ contains a non-empty basic $\BB$-radial set $\{x\in \BB  :  \varphi(x,s',a')\}$ which is not a point. Since every non-empty basic $\BB$-radial set which is not a point is infinite, this contradicts the assumption.  
\end{proof}

\begin{lemma}\label{lem:finiteimage} Let $X\subseteq (\tilde{k})^n\times \mathbb{B}$ be an $\cL_\mathbb{B}$-definable set such that, for each $a\in (\tilde{k})^n$, the fiber $X_a:=\{x\in \mathbb{B}  :  (a,x)\in X\}$ is finite. Then $\pi_2(X)$ is finite, where $\pi_2\colon X\to \mathbb{B}$ denotes the projection to the second coordinate. 
\end{lemma}
\begin{proof} Suppose for a contradiction that $Z:=\pi_2(X)$ is infinite. By Proposition \ref{thm:defsets}, $Z$ must contain a non-empty basic $\BB$-radial set~$D$ which is not a point. The cardinality of $D$ is greater than or equal to the cardinality of some non-empty open interval with endpoints in $|k|$. There exists an elementary extension~$\mathbf{k}'$ of~$\mathbf{k}$ such that $\tilde{k}=\tilde{k'}$ and every open interval in $|k'|$ is of cardinality strictly bigger than $\card(\tilde{k'})$. By Lemma \ref{lem:EQinfty}, there is $N$ such that for all $a\in (\tilde{k})^n$, the fiber $X_a$ has less than $N$ elements, and the same property holds over~$\tilde k'$. We therefore have 
\[\card(Z(\mathbf{k}'))\leqslant \card(\tilde{k'})=\card(\tilde{k}) < \card(D(\mathbf{k}'))\leqslant \card(Z(\mathbf{k}')),  
\]
a contradiction. 
\end{proof}

We finish this subsection with a slightly technical result. We will use the following notation. For $s\in |k|$ such that $s<1$ and $c\in \tilde{k}$, we set $D[s]:=\{\eta_{x,r}\in \DD_k^{(1,2)} :  r=s\}$ and $D_c[s] := D[s]\cap\red^{-1}(c)$. Let $\cL$ denote either $\cL_\BB$ or $\cL_\BB^\an$. 

\begin{lemma}\label{lem:opendisc} Let $C$ be an algebraic curve over~$\tilde{k}$ and $h\colon C\to \AA^1_{\tilde{k}}$ be a finite morphism. Let $s\in |k|$ be such that $s<1$. Let $H\subseteq C(\tilde k)\times D[s]$ be an $\cL$-definable set such that the projection of $H$ onto the first coordinate is a cofinite subset $U\subseteq C(\tilde k)$ and for all $\alpha\in U$, $H_\alpha\subseteq D_{h(\alpha)}[s]$. Then there is a finite set $F\subseteq U$ such that either $H_\alpha = \emptyset$ for all $\alpha\in U\setminus F$, or $H_\alpha=D_{h(\alpha)}[s]$ for all $\alpha\in U\setminus F$. 
\end{lemma}

\begin{proof} Suppose for a contradiction such a finite set $F$ does not exist. Then there is a cofinite subset $U'\subseteq U$ such that for all $\alpha\in U'$, the fiber $H_\alpha$ is a non-empty proper $\cL$-definable subset of $D_{h(\alpha)}[s]$. Consider for each $\alpha\in U'$ the definable subset of $k$ given by 
\[
W_\alpha:=\bigcup_{\eta_{x,s}\in H_\alpha} D(x,s),  
\]
which is a subset of $\res^{-1}(h(\alpha))$. Let $W_\alpha^c:=\res^{-1}(h(\alpha))\setminus W_\alpha$. By assumption, both $W_\alpha$ and $W_\alpha^c$ are non-empty. By $C$-minimality (Theorem \ref{thm:cmin}), $W_\alpha$ and $W_\alpha^c$ are disjoint unions of Swiss cheeses. Since they cover $\res^{-1}(h(\alpha))$, either $W_\alpha$ or $W_\alpha^c$ contains a set of the form $\res^{-1}(h(\alpha))\setminus D(a,r)$ for some $a\in \res^{-1}(h(\alpha))$ and $r<1$. By strong minimality, there is a cofinite set $U''\subseteq U'$ such that either for all $\alpha\in U''$, $W_\alpha$ contains such a set or for all $\alpha\in U''$, $W_\alpha^c$ contains such a set. Suppose without loss of generality the latter happens. Then, for each $\alpha\in U''$, $W_{\alpha}$ is contained in a closed subdisc of the open disc $\res^{-1}(h(\alpha))$. Denote by~$D_\alpha$ the smallest one. Let $g\colon U'\to \BB$ be the definable function sending $\alpha\in U''$ to $\eta_{x,r} \in \BB$ with $D(x,r)=D_\alpha$. By Lemma \ref{lem:finiteimage}, $g$ has finite image, which contradicts the fact that for each $\alpha$, $g(\alpha)$ lies in $\red^{-1}(h(\alpha))$ (note that $h(U'')$ is infinite as $h$ is finite). 
\end{proof}

\begin{corollary}\label{cor:opendisc} Let $C$ be an algebraic curve over~$\tilde{k}$ and $h\colon C\to \AA^1_{\tilde{k}}$ be a finite morphism. Let $H\subseteq C(\tilde k)\times D(0,1)$ be an $\cL$-definable set such that the projection of $H$ onto the first coordinate is a cofinite subset $U\subseteq C(\tilde k)$ and for all $\alpha\in U$, $H_\alpha\subseteq \res^{-1}(h(\alpha))$. Then there is a finite set $F\subseteq U$ such that either $H_\alpha = \emptyset$ for all $\alpha\in U\setminus F$ or $H_\alpha=\res^{-1}(h(\alpha))$ for all $\alpha\in U\setminus F$. 
\end{corollary}

\begin{proof} Apply Lemma \ref{lem:opendisc} with $s=0$ to the definable set $H':=\{(\alpha,\eta_{x,0}) :  (\alpha,x)\in H\}$, which is in definable bijection with $H$. 
\end{proof}

\subsection{Reduction to $\BB$}

Let $X$ be a $k$-analytic curve and $\cS$ be an $\cL$-facade of $X$. Let $\Gamma=(S,E)$ be the skeleton associated to $\cS$. Recall that the $\cL$-definable set $X^{\cS}$ corresponds to 
\[
X^{\cS}:=\bigsqcup_{x\in S^{(1)}} \eta_{1,0} \sqcup \bigsqcup_{I\in E} f_I(V_I)^\D\sqcup \bigsqcup_{x\in S^{(2)}} \left[\eta_{0,1} \sqcup Z_x^\cS \sqcup \bigsqcup_{i=1}^{m(x)} f_{x,i}(V_{x,i})^\D \right] 
\]
where the set $Z_x^\cS$ was defined by 
\[
Z_x^\cS:=\{(\alpha,\eta_{a,r})\in \cU_x(\tilde{k})\times \DD_k^{(1,2)} :  \tilde{f}_x(\alpha)=\red(\eta_{a,r}) \}.
\]
To provide a description of the definable subsets of $X^\cS$, it is enough to describe the definable subsets of each piece in the previous disjoint union. The pieces $f_I(V_I)^\D$ and $f_{x,i}(V_{x,i})^\D$ are subsets of~$\BB$, whose definable subsets have been classified as $\BB$-radial sets. It remains to describe the definable subsets of the pieces of the form $Z_x^\cS$. 

The following proposition explains why, by passing to a larger finite triangulation, we can reduce the description of definable subsets of $Z_x^\cS$ to those of $\BB$. 

\begin{proposition}\label{prop:defsetsZx} Let $H$ be a definable subset of~$Z_x^\cS$ and let $\pi_1$ and $\pi_2$ denote the projections from $H$ onto $\cU_x(\tilde{k})$ and $\DD_k^{(1,2)}$ respectively. Then there is a finite set $F\subseteq \cU_x(\tilde{k})$, $n\in \NN$ and closed disc cylinders $D_1,\ldots,D_n$ of the form 
\[
D_i:=\{\eta_{w,r}\in \BB :  D(w,r)\subseteq D(0,1) \ \wedge \ r=s_i\}
\]
or 
\[
D_i:=\{\eta_{w,r}\in \BB :  D(w,r)\subseteq D(0,1)  \ \wedge \ s_{i,0}<r<s_{i,1}\}
\]
such that 
\[
H\cap \pi_1^{-1}(\cU_{x}(\tilde k)\setminus F)= \pi_2^{-1}(\bigcup_{i=1}^n D_i) \cap \pi_1^{-1}(\cU_{x}(\tilde k)\setminus F). 
\] 
\end{proposition}

\begin{proof} 
Consider the definable set $F\subseteq \cU_x(\tilde{k})$ defined by 
\[
F:=\{\alpha\in \cU_x(\tilde{k}) :  (\exists s)(s<1 \wedge \emptyset \subsetneq H_{\alpha} \cap D_{\tilde{f}_{x}(\alpha)}[s] \subsetneq D_{\tilde{f}_{x}(\alpha)}[s])\}.
\]
Let us show that $F$ is finite. Suppose towards a contradiction that $F$ is infinite. Consider the definable set $Y\subseteq F\times |k|$ defined by 
\[
Y:=\{(\alpha,s) :  s<1 \wedge \emptyset \subsetneq H_{\alpha} \cap D_{\tilde{f}_{x}(\alpha)}[s] \subsetneq D_{\tilde{f}_{x}(\alpha)}[s]\}.
\]
By the existence of Skolem functions in the value group, there is a definable function $g\colon Y\to |k|$ such that $(\alpha,g(\alpha))\in Y$ for all $\alpha\in F$. By Lemma \ref{lem:ort}, $g$ has finite image. By strong minimality of $\cU_x(\tilde{k})$, there is a cofinite subset~$F'$ of~$F$ such that $g(F')=\{s\}$ for some $s\in |k|$. But then, $H':=\{(\alpha,\eta_{a,r})\in H :  \alpha\in F', r=s\}$ contradicts Lemma \ref{lem:opendisc}. Therefore $F$ is finite. 

Set $C:= \cU_x(\tilde{k})\setminus F$. Consider the definable set 
\[G:=\{(\alpha,s)\in C\times |k|  :  H_{\alpha} = D_{\tilde{f}_{x}(\alpha)}[s]\}.\] 
By o-minimality, for each $\alpha\in C$, $G_{\alpha}$ is a finite union of points and intervals, therefore, $H_\alpha$ is a finite union of open disc cylinders as in \eqref{eq:odcylin1} or \eqref{eq:odcylin2}. By strong minimality of $C$, there is an integer~$N$ and a finite set $E\subseteq C$ such that for all $\alpha\in C\setminus E$, $H_\alpha$ is the union of $N$ open disc cylinders. By enlarging~$F$ adding~$E$, we may assume that~$E$ is empty.

By possibly further enlarging $F$, we may suppose there are definable functions $f_1,\ldots,f_{M}\colon C\to |k|$ and finite sets of indices $I_1$ and $I_2$ such that for all $\alpha\in C$
\[
G_\alpha= \bigcup_{i\in I_1} (f_{i}(\alpha),f_{i+1}(\alpha)) \cup \bigcup_{i\in I_2} \{f_i(\alpha)\}. 
\]
By Lemma \ref{lem:ort} and strong minimality, there is a finite set $E'\subseteq C$ such that all functions $f_i$ are constant on $C\setminus E'$. Let $s_i:=f_i(\alpha)$ for some (any) $\alpha \in C\setminus E'$. As before, up to enlarging~$F$, we may assume that~$E'$ is empty. Finally, consider, for $i\in I_1$, the disc cylinder 
\[
D_i:=\{\eta_{w,r}\in \BB  :  D(w,r)\subseteq D(0,1) \ \wedge \ s_i<r<s_{i+1} \},
\]
and, for $i\in I_2$, the disc cylinder
\[
D_{i}:=\{\eta_{w,r}\in \BB  :  D(w,r)\subseteq D(0,1) \ \wedge \ r=s_{i} \} .
\]
By construction, setting $F' := \tilde{f}_{x}^{-1}(\tilde{f}_{x}(F))$, we have
\[
H\cap \pi_1^{-1}(\cU_{x}(\tilde k)\setminus F')= \pi_2^{-1}(\bigcup_{i\in I_{1}\cup I_{2}} D_i) \cap \pi_1^{-1}(\cU_{x}(\tilde k)\setminus F'). 
\] 
\end{proof}

The previous theorem suggests the following definition:

\begin{definition}\label{def:basicZ_x} A subset $H\subseteq Z_x^\cS$ is called \emph{basic $\BB$-radial} if $H=\pi_2^{-1}(D) \cap Z_{x}^\cS$ for a closed cylinder $D$ of the form 
\[
D:=\{\eta_{w,r}\in \DD_k^\D :  \ r=s\} \text{ or }  D:=\{\eta_{w,r}\in \DD_k^\D :  \ s_{0}<r<s_{1}\}.
\]
\end{definition} 

\begin{definition}\label{def:BBradial} Let $\Gamma=(S,E)$ be the skeleton associated to $S$. A subset $Y\subseteq X^\cS$ is called \emph{$\BB$-radial} if 
\begin{enumerate}[$i)$]
\item for each $I\in E$, $Y\cap f_I(V_I)^\D$ is $\BB$-radial, 
\item for each $x\in S^{(2)}$, $Y\cap Z_x^\cS$ is a disjoint union of basic $\BB$-radial sets, 
\item for each $I\in E$, and $i\in\{1,\ldots, m(x)\}$, $Y\cap f_{x,i}(V_{x,i})^\D$ is $\BB$-radial.  
\end{enumerate}
\end{definition} 

Note that all $\BB$-radial sets are $\cL_\BB$-definable. 

\begin{theorem}\label{thm:definablecurves} Let $X$ be a $k$-analytic curve and $\cS$ be an $\cL$-facade of $X$. Let $Y\subseteq X^\cS$. Then $Y$ is $\cL$-definable if, and only if, there is a refinement $\cS'$ of $\cS$ and a $\BB$-radial set $Y'$ such that $Y=\id_{\cS\cS'}^{-1}(Y')$.  
\end{theorem}

\begin{proof} Assume that~$Y$ is $\cL$-definable. Fix $x\in S^{(2)}$ and let $H_x:=Y\cap Z_x^\cS$. By Proposition \ref{prop:defsetsZx}, there is a finite set $F_x\subseteq \cU_x(\tilde{k})$, $n_{x}\in \NN$ and closed disc cylinders $D_{1,x},\ldots,D_{n_x,x}$ such that 
\[
H_{x}\cap \pi_1^{-1}(\cU_{x}(\tilde k)\setminus F_{x})= \pi_2^{-1}(\bigcup_{i=1}^{n_{x}} D_{i,x}) \cap \pi_1^{-1}(\cU_{x}(\tilde k)\setminus F_{x}). 
\] 
where $\pi_1$ and $\pi_2$ denote the projections to $\cU_x(\tilde{k})$ and $\DD_k^\D$ respectively. Let $S'$ be refinement of~$S$ in which, for each $x\in S^{(2)}$, we add a point of type 1 in each connected component in $\bigcup_{\alpha\in F_x}\varphi_\cS^{-1}(\pi_1^{-1}(\alpha))$. Let $\cS'$ be an $\cL$-facade refining $\cS$ with underlying triangulation~$S'$ (see Remark~\ref{rem:facade-ref}). Note that $S^{(2)}=S'^{(2)}$. In particular, for every $x\in S'^{(2)}$, we have 
\[
\id_{\cS\cS'}(H_x) = \pi_2^{-1}(\bigcup_{i=1}^{n_{x}} D_{i,x}) \cap Z_{x}^{\cS'},  
\] 
which is a disjoint union of basic $\BB$-radial sets. For each $I\in E_{S'}$, the set $\id_{\cS\cS'}(Y\cap f_I(V_I)^\D)$ is $\BB$-radial by Theorem \ref{thm:defsets}. This same argument applies for each $\id_{\cS\cS'}(Y\cap f_{x,i}(V_{x,i})^\D)$ for $x\in S'^{(2)}$ and $i\in\{1,\ldots, m(x)\}$. Setting $Y':=\id_{\cS\cS'}(Y)$ shows the result. 

The converse follows from the $\cL_\BB$-definability of $\BB$-radial sets and Lemma~\ref{lem:facade-id}.
\end{proof}

\section{Radiality and definability}\label{sec:radiality}

In this section, we introduce the notion of radial set of a $k$-analytic curve. It is inspired by Temkin's article~\cite{temkin_2017_metric}, where a notion of radial morphism of $k$-analytic curves is provided. We will relate this notion to that of definable set and use this link to recover a result of Temkin showing that the topological ramification locus of a morphism of $k$-analytic curves admits a tame behavior.

In order to state the definition of radial set in greater generality, in this section (and only here), we will allow triangulations to be infinite. Recall that, by~\cite[Th\'eor\`eme~5.1.14]{Duc-book}, every quasi-smooth strictly $k$-analytic curve admits a triangulation (made of points of type~2 only). By a normalisation argument as in Lemma~\ref{lem:Swtov}, it follows that every strictly $k$-analytic curve admits a triangulation. 

Let~$X$ be a strictly $k$-analytic curve.

\subsection{Radiality}

Let~$S$ be a triangulation of~$X$. Let $x\in X$. If $x\notin \Gamma_{S}$, then the connected component~$C_{x}$ of~$X\setminus \Gamma_{S}$ containing~$x$ is isomorphic to the open unit disc~$\DD_{k}^-$. We denote by~$\rho_{S}(x)$ the radius of the image of the point~$x$ in~$\DD_{k}^-$. It does not depend on the choice of the isomorphism. If $x\in \Gamma_{S}$, we set $\rho_{S}(x) := 1$.

\begin{definition}\label{def:radialset_berk} We say that a subset~$Y$ of~$X$ is \emph{basic radial with respect to~$S$ over $k$} if it is of one of the following forms:
\begin{enumerate}[$i)$]
\item for some $x\in S^{(2)}$, $r_{1},r_{2} \in [0,1]\cap |k|$ and $\bowtie_{1}, \bowtie_{2} \in \{<,\leqslant\}$,
\[Y = \{y\in \tau_{S}^{-1}(x)  :  r_{1} \bowtie_{1} \rho_{S}(y) \bowtie_2 r_{2}\};\]
\item for some $I\in E_{S}$, $|k^\times|$-monomial functions $f_{1},f_{2} \colon I \to [0,1]$ and $\bowtie_{1}, \bowtie_{2} \in \{<,\leqslant\}$,
\[Y = \{y\in \tau_{S}^{-1}(I)  :  f_{1}(\tau_{S}(y)) \bowtie_{1} \rho_{S}(y) \bowtie f_{2}(\tau_{S}(y))\}.\]
\end{enumerate}
We say that a subset~$Y$ of~$X$ is \emph{radial with respect to~$S$ over $k$} if, for each $x\in S^{(2)}$, $Y\cap \tau_{S}^{-1}(x)$ is a finite union of basic radial sets with respect to~$S$ over~$k$ and, for each $I\in E_{S}$, $Y\cap \tau_{S}^{-1}(I)$ is a finite union of basic radial sets with respect to~$S$ over $k$. We say that a subset~$Y$ of~$X$ is \emph{radial over $k$} if it is radial over $k$ with respect to some triangulation. We let $\Rad_k(X)$ denote the set of radial subsets of $X$ over $k$. 
\end{definition}

The following lemma is an easy verification (yet again similar to Lemmas \ref{lem:intersectionalgebraic bricks}, \ref{lem:complementalgebraic bricks} and \ref{lem:booleancombination} for the first point) which is left to the reader. 

\begin{lemma}\phantomsection\label{lem:radialboolean}
\begin{enumerate}[$i)$]
\item The set $\Rad_k(X)$ is closed under finite boolean combinations. 
\item Let $Z\in \Rad_k(X)$ be a non-empty radial set. Then $Z^\D$ is non-empty.
\item For each point~$x$ of~$X$, there exists a basis of radial open (resp. radial strictly $k$-affinoid) neighborhoods of~$x$.
\qed
\end{enumerate}
\end{lemma} 

We now prove that the property of being radial is local. We start with a preliminary result.

\begin{lemma}\label{lem:radialsubcurve}
Let~$V$ be a compact analytic domain of~$X$. 
\begin{enumerate}[$i)$]
\item For each radial subset~$Y$ of~$X$, $Y \cap V$ is a radial subset of~$V$.
\item Each radial subset of~$V$ is a radial subset of~$X$. 
\end{enumerate}
\end{lemma}
\begin{proof}
For each triangulation~$S_{V}$ of~$V$, there exists a triangulation~$S$ of~$X$ such that $S\cap V = S_{V}$. One may moreover require that $\Gamma_{S} \cap V = \Gamma_{S_{V}}$ and that $\tau_{S}^{-1}(\Gamma_{S} \cap V) = V$. The results follow from these observations.
\end{proof}

Let us recall the definition of G-topology from~\cite[Section~1.3]{bleu}. We say that a family~$\cU$ of analytic domains of~$X$ is a covering for the G-topology if, for each point $x\in X$, there exists a finite subset~$\cU_{x}$ of~$\cU$ such that $\bigcap_{U \in \cU_{x}} U$ contains~$x$ and $\bigcup_{U \in \cU_{x}} U$ is a neighborhood of~$x$ in~$X$.

\begin{lemma}\label{lem:radiallocal}
Let~$Y$ be a subset of~$X$. Assume that there exists a family~$\cU$ of analytic domains of~$X$ that is a covering for the G-topology and such that, for each $U\in \cU$, $Y\cap U$ is a radial subset of~$U$. Then, $Y$ is a radial subset of~$X$.
\end{lemma}
\begin{proof}
Let~$U\in\cU$. For each $x\in U$, consider a compact analytic domain~$V_{x}$ of~$U$ that is a neighborhood of~$x$ in~$U$. The set $\cV = \{V_{x}  :  U\in\cU, x\in U\}$ is still a covering of~$X$ for the G-topology and, by Lemma~\ref{lem:radialsubcurve}, for each $V\in  \cV$, $Y\cap V$ is a radial subset of~$X$. 

Let~$S$ be a triangulation of~$X$ such that each connected component of~$X\setminus S$ is relatively compact in~$X$ (as in Ducros' original definition, see \cite[5.1.13]{Duc-book}). Up to refining~$S$, we may also assume that each connected component of~$X\setminus S$ that is an annulus has two distinct endpoints.

Let $x\in S^{(2)}$. The set~$\tau^{-1}_{S}(x)$ is compact, hence it may be covered by finitely many elements of~$\cV$. It follows that $\tau^{-1}_{S}(x) \cap Y$ is a radial subset of~$\tau^{-1}_{S}(x)$ with respect to some triangulation~$S_{x}$ of $\tau^{-1}_{S}(x)$ containing~$x$. Note that~$S_{x}$ is finite.

Let $I\in E_{S}$. The set $A_{I} := \tau^{-1}_{S}(I)$ is a relatively compact annulus and its boundary~$\partial A_{I}$ in~$X$ is made of two points of~$S$, say $x_{1}$ and~$x_{2}$.  
For each $i = 1,2$, denote by~$\cD_{i}$ the set of connected components of $\tau_{S}^{-1}(x_{i}) \setminus \{x_{i}\}$ (all of which are open discs) and pick a non-trivial cofinite subset~$\cD'_{i}$ of~$\cD_{i}$. Then, the set 
\[W := A_{I} \cup \bigcup_{i=1,2} \{x_{i}\} \cup \bigcup_{D\in \cD'_{i}} D\]
is a compact strictly $k$-analytic domain~$W$ of~$X$ containing~$A_{I}$ and whose topological skeleton is the closure of~$I$ in~$X$. As before, $W\cap Y$ is a radial subset of~$W$ with respect to some triangulation~$S_{W}$ of~$W$. 
Then, $S_{I} := S_{W} \cap A_{I}$ is a triangulation of~$A_{I}$ and $Y \cap A_{I}$ is a radial subset of~$A_{I}$ with respect to it. Note that~$S_{I}$ is finite.

Finally, set 
\[S' := S^{(1)} \cup \bigcup_{x\in S^{(2)}} S_{x} \cup \bigcup_{I\in E_{S}} S_{I}.\]
It is easy to check that~$S'$ is a triangulation of~$X$ and that~$Y$ is radial with respect to it.
\end{proof}

Let~$\cS$ be an $\cL$-facade of~$X$. Let $\DDD_k(X^\cS)$ denote the collection of definable subsets of $X^\cS$ over~$k$.

\begin{theorem}\label{thm:radialdef} The map 
\[
\delta_{k} \colon Y\in \Rad_k(X) \mapsto \varphi_\cS(Y^\D) \subseteq X^\cS
\]
induces a bijection from $\Rad_{k}(X)$ to $\DDD_{k}(X^\cS)$.
\end{theorem}

\begin{proof} 
Let $Y \in \Rad_{k}(X)$. By definition, $Y$~is a finite union of basic radial sets with respect to some triangulation~$S_{0}$ of~$X$. Let~$\cS'$ be an $\cL$-facade of~$X$ refining~$\cS$ whose underlying triangulation refines~$S_{0}$. Then, for each subset~$Y_{0}$ of~$X$ that is basic radial with respect to~$S_{0}$, the set $\varphi_{\cS'}(Y_{0})$ is $\BB$-radial, hence $\cL_{\BB}$-definable. Using Lemma~\ref{lem:facade-id} and the compatibiliy of the map~$\delta_{k}$ with unions, we deduce that $\delta_{k}(Y)$ belongs to $\DDD_{k}(X^\cS)$.

Let us now show that $\delta_k$ is injective. First note that $\delta_k$ preserves boolean combinations as both the function $\varphi_\cS$ and the restriction $(\cdot)^\D$ preserve boolean combinations. Let $Y_1, Y_2\in \Rad_k(X)$ be such that $Y_1\neq Y_2$ and suppose for a contradiction that $\delta_k(Y_1)=\delta_k(Y_2)$. Without loss of generality we may assume that $Z:=Y_1\setminus Y_2$ is non-empty. We have $\delta_k(Y_1\setminus Y_2)=\delta_k(Y_1)\setminus \delta_k(Y_2)$ and hence $\delta_k(Z)=\emptyset$, which contradicts Lemma~\ref{lem:radialboolean}.  

Let us now show surjectivity. Let $H\in \DDD_k(X^\cS)$. By Theorem \ref{thm:definablecurves} there is a refinement~$\cS'$ of~$\cS$ and a $\BB$-radial set $H'$ such that $H=\id_{\cS\cS'}^{-1}(H')$. Thus $H'$ is a disjoint union of basic $\BB$-radial sets. Working in the refined triangulation underlying $\cS'$, we let the reader verify that every basic $\BB$-radial set (both in the sense of Definition \ref{def:radialset} and Definition \ref{def:basicZ_x}) has a radial pre-image almost by definition. 
\end{proof}

\subsection{Weakly stable fields}

In this section, we prove a technical result about the local ring~$\cO_{X,x}$ at a point~$x$ of type~2, 3 or~4 in a $k$-analytic curve~$X$, namely that it is weakly stable. Let us first recall the definition (see \cite[Definitions 3.5.2/1 and 2.3.2/2]{BGR}). 

\begin{definition}
A valued field~$K$ is said to be \emph{weakly stable} if, for each finite extension~$L$ of~$K$ endowed with the spectral norm, each sub-$K$-vector space of~$L$ is closed. 
\end{definition}

\begin{lemma}\label{lem:discTp}
Assume that $k$ is of characteristic $p>0$. Let $a\in k$ and $r \in \R_{>0}$. Then, each element $f\in \cO(\DD(a,r))$ may be uniquely written as
\[f = f_{0}+ f_{1} \cdot T + \dotsb + f_{p-1} \cdot T^{p-1},\]
where, for each $i\in \{0,\dotsc,p-1\}$, $f_{i}$ is an element of $\cO(\DD(a,r))$ of the form $f^\dag_{i}((T-a)^p)$. Moreover, we have
\[\max_{0\le i\le p-1} ( |f_{i}(\eta_{a,r})| \, r^i\, (r/s)^{p-1-i}  ) \le |f(\eta_{a,r})| \le \max_{0\le i\le p-1}( |f_{i}(\eta_{a,r})| \, s^i ),\]
where $s := \max(|a|,r)$. 
\end{lemma}
\begin{proof}
Let $f\in \cO(\DD(a,r))$. It may be written as a series in powers of $T-a$, hence in the form 
\[ f = g^\dag_{0}((T-a)^p) + g^\dag_{1}((T-a)^p) \cdot (T-a) + \dotsb + g^\dag_{p-1}((T-a)^p) \cdot (T-a)^{p-1} \]
in a unique way. For each $i\in \{1,\dotsc,p-1\}$, the series $g_{i} := g^\dag_{i}((T-a)^p)$ is obtained from the series~$f$ by removing the coefficients of order different from~$i$ modulo~$p$ and dividing by $(T-a)^i$. Since the radius of convergence cannot decrease under these operations, $g_{i}$~converges on $\DD(a,r)$. The first part of the statement follows easily with, for each $i\in \{1,\dotsc,p-1\}$, 
\[ f_{i} = \sum_{j=i}^{p-1} {j \choose i} (-a)^{j-i} g_{j} .\]

Thanks to the non-archimedean inequality, we have
\[  |f(\eta_{a,r})| \le \max_{0\le i\le p-1}( |f_{i}(\eta_{a,r})| \, |T(\eta_{a,r})|^i ) = \max_{0\le i\le p-1}( |f_{i}(\eta_{a,r})| \, s^i ).\]
To prove the remaining inequality, note that we have $f_{p-1} = g_{p-1}$, hence 
\[|f(\eta_{a,r})| = \max_{0\le i\le p-1} (|g_{i}(\eta_{a,r})| \, r^i) \ge |f_{p-1}(\eta_{a,r})|\, r^{p-1}.\] 
Let $i \in \{1,\dotsc,p-1\}$. We have 
\begin{align*}
T^{p-1-i} f &= \sum_{j=i+1}^{p-1} T^p f_{j} \cdot T^{j-i-1} + \sum_{j=0}^{i-1} f_{j}  \cdot T^{j+ p -i-1} + f_{i} \cdot T^{p-1}\\
& = \sum_{j=i+1}^{p-1} ((T-a)^p + a^p) f_{j}  \cdot T^{j-i-1} + \sum_{j=0}^{i-1} f_{j}  \cdot T^{j + p-i-1} + f_{i} \cdot T^{p-1},
\end{align*}
which is the decomposition of the statement. Applying the inequality above, we get
\[ s^{p-1-i}\, |f(\eta_{a,r})| = |(T^{p-1-i} f)(\eta_{a,r})| \ge |f_{i}(\eta_{a,r})|\, r^{p-1}.\] 
The result follows.
\end{proof}

\begin{theorem}\label{thm:weaklystable}
Let~$X$ be a $k$-analytic curve and let $x\in X$ be a point of type~2, 3 or~4. Then, the field~$\cO_{X,x}$ is weakly stable.

In particular, for each finite valued extension~$L$ of~$\cO_{X,x}$, we have 
\[ [L \colon \cO_{X,x}] = [\hat L \colon \hat \cO_{X,x}].\]
\end{theorem}
\begin{proof}
Since a finite extension of a weakly stable field is still weakly stable, we may assume that $x$~belongs to~$\AA_k^{1,\an}$, by Noether normalization. By \cite[Proposition~3.5.1/4]{BGR}, fields of characteristic~0 are weakly stable, so we may assume that $k$~is of characteristic~$p>0$.

Assume that $x$ is of type~4. Then, there exists a nested family of closed discs $(\DD(a_{i},r_{i}))_{i\in I}$ such that the absolute value associated to~$x$ on~$k[T]$, where $T$ is a coordinate on~$\AA_k^{1,\an}$, is the infimum of the absolute values associated to the points~$\eta_{a_{i},r_{i}}$. 

The $r_{i}$'s are bounded above and below by a positive constant independently of~$i$. We can always remove some discs in the family so as to assume that none of them contains the point~0. In this case, we have $|a_{i}| \ge r_{i}$ and $|a_{i}|$ is actually independent of~$i$. For later use, let us choose $C,D \in \R_{>0}$ such that for all $i\in I$ and $j\in \{0,\dotsc,p-1\}$, we have 
\[C \le r_{i}^j (r_{i}/|a_{i}|)^{p-1-j} \textrm{ and } |a_{i}|^j \le D.\]

The family of discs $(\DD(a_{i},r_{i}))_{i\in I}$ forms a basis of neighborhoods of~$x$, so that the local ring $\cO_{x} := \cO_{\AA^{1,\an}_{k},x}$ at~$x$ may be written as
\[\cO_{x} = \varinjlim_{i\in I} \cO(\DD(a_{i},r_{i})).\]
Denote by $\cO_{x}^{1/p}$ the ring of $p^\textrm{th}$-roots of~$\cO_{x}$. By \cite[Theorem~3.5.3/1]{BGR}, to prove that~$\cO_{x}$ is weakly stable, it is enough to prove that $\cO_{x}^{1/p}$ endowed with the spectral norm is weakly $\cO_{x}$-cartesian, that is to say the spectral norm on~$\cO_{x}^{1/p}$ is equivalent to the norm induced by some $\cO_{x}$-basis (see \cite[Theorem~2.3.2/1]{BGR}).

Let $i\in I$. Consider the ring $\cO(\DD(a_{i},r_{i}))^{1/p}$ of $p^\textrm{th}$-roots of elements of $\cO(\DD(a_{i},r_{i}))$. Since $k$~is perfect, it consists of the series of the form 
\[\sum_{n\ge 0} c_{n} (T-a_{i})^{n/p} = \sum_{n\ge 0} c_{n} \big(T^{1/p}-a_{i}^{1/p}\big)^{n}\]
with coefficients in~$k$ such that $\lim_{n\to \infty} |c_{n}| r_{i}^{n/p} = 0$. We may identify it with $\cO(\DD(a_{i}^{1/p},r_{i}^{1/p}))$. Under this identification, its subring $\cO(\DD(a_{i},r_{i}))$ is sent to the ring of elements whose expansion only involves power of the variable that are multiples of~$p$. It now follows from Lemma~\ref{lem:discTp} that we have 
\[\cO(\DD(a_{i},r_{i}))^{1/p} = \bigoplus_{j=0}^{p-1} \cO(\DD(a_{i},r_{i})) \cdot T^{j/p}.\]
Moreover, denoting by~$\pi_{j}$ the projection onto the $j^\textrm{th}$ factor, for each $f \in \cO(\DD(a_{i},r_{i}))^{1/p}$, we have
\[ C \max_{0\le j\le p-1} \big(|\pi_{j}(f)(\eta_{a_{i},r_{i}})|)  \le |f(\eta_{a_{i}^{1/p},r_{i}^{1/p}})| \le D \max_{0\le j\le p-1} \big(|\pi_{j}(f)(\eta_{a_{i},r_{i}})|).\] 
Note that the direct sum decomposition is compatible with the restriction maps between the different discs. 

We have 
\[\cO_{x}^{1/p} = \varinjlim_{i\in I} \cO(\DD(a_{i},r_{i}))^{1/p} = \bigoplus_{j=0}^{p-1} \cO_{x} \cdot T^{j/p}.\]
The infimum of the norms of the discs $\DD(a_{i}^{1/p},r_{i}^{1/p})$ induces an absolute value on~$\cO_{x}^{1/p}$ that extends the absolute value on~$\cO_{x}$, hence it is the spectral nom of~$\cO_{x}^{1/p}$. The inequality above ensures that it is equivalent to the norm induced by the basis $(1,T^{1/p},\dotsc,T^{(p-1)/p})$, and the result follows.

\medbreak

We will not give a complete proof for~$x$ of type~2 or~3. Note that the result is already known in this case (see \cite[Th\'eor\`eme~4.3.13]{Duc-book}). For the sake of completeness, let us point out that it could be proven in a way that is completely similar to what we have done here. Let us only explain how to describe explicitly the local ring at such a point. 

Assume that $x$~is of type~3, hence of the form $\eta_{a,r}$. The closed annuli $\DD(a,s_{2}) \setminus \DD^-(a,s_{1})$, with $s_{1} < r < s_{2}$, form a basis of neighborhoods of~$x$ and, on each of those annuli, a global analytic function may be written as a power series
\[\sum_{n\in \Z} c_{n} (T-a)^n\]
with suitable convergence conditions.

Assume that $x$~is of type~2. Up to changing coordinates, we may assume that it is the point~$\eta_{1}$. Fix a family of representatives~$R$ of~$\tilde k$ in~$k^\circ$. The closed Swiss cheeses $\DD(0,s_{2}) \setminus \bigcup_{a \in R_{0}} \DD^-(a,s_{1})$, where $R_{0}$~is a finite subset of~$R$ and $s_{1} < 1 < s_{2}$, form a basis of neighborhoods of~$x$ and, on each of those Swiss cheeses, a global analytic function may be written as a power series
\[\sum_{n\ge 0} c_{n} \, T^n + \sum_{a\in R_{0}} \sum_{n\ge 1} c_{a,n} (T-a)^{-n}\]
with suitable convergence conditions (see \cite[Proposition~2.2.6]{FvdP}).

\medbreak

The last part of the statement follows from \cite[Proposition~2.3.3/6]{BGR}.
\end{proof}


\subsection{Base change} For any complete rank~1 valued extension~$L$ of~$k$, we denote by~$X_{L}$ the scalar extension of~$X$ to~$L$ and by $\pi_{L}\colon X_L\to X$ the natural projection morphism. 

Since~$k$ is algebraically closed, by \cite[Corollaire~3.14]{P}, for each~$x \in X$, the tensor norm on $\sH(x)\ho_{k} L$ is multiplicative, hence gives rise to a point of~$X_{L}$. We denote the latter by~$\sigma_{L}(x)$.

\begin{example}\label{ex:sigmaetaar}
If $X = \AA_k^{1,\an}$, for each $a\in k$ and $r\in \R_{\ge 0}$, we have $\sigma_{L}(\eta_{a,r}) = \eta_{a,r}$.
\end{example}

\begin{theorem}\label{thm:base_change} Let~$L$ be a complete rank~1 valued extension of~$k$. The map $\sigma_L\colon x\in X\mapsto \sigma_{L}(x) \in X_L$ is a section of~$\pi_{L}$ with the following properties:
\begin{enumerate}[$i)$]
\item if $x \in X^{(1)}$, then $\sigma_{L}(x) \in X_{L}^{(1)}$; 
\item if $x \in X^{(2)}$, then $\sigma_{L}(x) \in X_{L}^{(2)}$ and we have an isomorphism $\cC_{x} \otimes_{\tilde k} \tilde L \xrightarrow[]{\sim} \cC_{\sigma_L(x)}$, hence an isomorphism $\cC_{x}(\tilde L) \xrightarrow[]{\sim} \cC_{\sigma_L(x)}(\tilde L)$; 
\item if $S$ is a triangulation of $X$ with skeleton~$\Gamma_{S}$, then $\sigma_L(S)$ is a triangulation of $X_L$ with skeleton $\sigma_{L}(\Gamma_{S})$; 
\item if~$L$ is algebraically closed, given a facade $\cS$ of $X$, there is an associated facade $\cS_L$ of~$X_L$.
\end{enumerate}
\end{theorem}
\begin{proof}
Point i) is clear. Point~ii) follows from \cite[Lemma~1.15]{PV}. Point~iii) follows from~\cite[Corollary~2.16]{PP}.  For point~iv), we let the reader check that we obtain a facade by base-changing all the data to~$L$.
\end{proof}

\begin{corollary}\label{cor:radial_commute} Let $\cS$ be an $\cL$-facade of $X$. Let $\mathbf{L}$ be a rank~1 elementary extension of~$\mathbf{k}$. 

For $x\in S^{(2)}$, the isomorphism $\cC_{x}(\tilde L) \xrightarrow[]{\sim} \cC_{\sigma_L(x)}(\tilde L)$ given by point~ii) in Theorem \ref{thm:base_change} provides a canonical identification of  $Z_x^\cS(\mathbf{L})$ and $Z_{\sigma_L(x)}^{\cS_L}$. Therefore, we have a canonical identification of $X^\cS(\mathbf{L})$ and $X_L^{\cS_L}$.

Moreover, the following diagram commutes:
\[
\begin{tikzcd}
X_L^{\D} \ar{r}{\varphi_{\cS_L}} & X^{\cS_L}\ar[equal]{r} &  X^{\cS}(\mathbf{L}) \\
X^{\D} \ar{u}{\sigma_L}  \ar{r}{\varphi_{\cS}} & X^{\cS} \ar{ru}{\subseteq}
\end{tikzcd}
\texto{.}{-1.5pc}
\]

If~$X'$ is a strictly $k$-analytic curve, $\cS'$ an $\cL$-facade of~$X'$ and $h \colon X \to X'$ a $k$-analytic morphism such that the map $h_{\cS\cS'} \colon X^\cS \to X'{}^{\cS'}$ is definable, then the map $h_{L,\cS_{L}\cS'_{L}} \colon X_{L}^{\cS_{L}} \to X_{L}'{}^{\cS'_{L}}$ is definable too and we have a canonical identification $h_{L,\cS_{L},\cS'_{L}} = h_{\cS\cS'}(\mathbf{L})$.
\end{corollary}
\begin{proof} The canonical identification $Z_x^\cS(\mathbf{L}) = Z_{\sigma_L(x)}^{\cS_L}$ follows readily form the identification of the points of the residue curves. To prove the identification $X^\cS(\mathbf{L}) = X_L^{\cS_L}$, note that, by Theorem~\ref{thm:base_change}, the identity is a bijection between $f_{I_L}^{\cS_L}(V_{I_L})$ and $f_I^\cS(V_I)(\mathbf{L})$ (and similarly for $f_{x,i}(V_{x,i}^\D)$). 

The commutativity of the diagram follows by the same kind of arguments, using the explicit description of~$\sigma_{L}$ on each subset of the partition of~$X$ associated to the facade~$\cS$ (see Example~\ref{ex:sigmaetaar}).

Finally, the statement about the morphism is a consequence of the previous statements and the definitions.
\end{proof} 

\begin{definition} Let $X$ be as above and let $\mathbf{L}$ be a rank 1 elementary extension of $\mathbf{k}$. Let $Y\subseteq X$ be a radial set. We define $Y_L$ to be the radial subset of $X_L$ defined by
\[
Y_L:=\delta_L^{-1}(\delta_k(Y)(\mathbf{L})). 
\]
\end{definition} 

\begin{remark}\phantomsection\label{rem:radial} 
\begin{enumerate}[$i)$]
\item The set $Y_L$ coincides with the radial subset of $X_L$ obtained by applying the same conditions used to define $Y$. This follows by inspection on basic radial sets and the fact that the maps $\delta_k$ and $\delta_L$ preserve disjoint unions. 
\item Corollary~\ref{cor:radial_commute} ensures that we have $\sigma_{L}^{-1}(Y_{L}) = Y$. 
\end{enumerate}
\end{remark}

\begin{lemma}\label{lem:degreesigma}
Let~$h\colon X\to Y$ be a morphism of $k$-analytic curves. Let $x\in X$ and let~$L$ be a complete rank 1 valued field extension of $k$. Then, we have $h_{L}(\sigma_{L}(x)) = \sigma_{L}(h(x))$. 

Set $y:= h(x)$. If $\dim_{x}(h) = 0$, then we have
\[ h_{L}^{-1}(\sigma_{L}(y)) \cap \pi_{L}^{-1}(x) = \{\sigma_{L}(x)\}. \] 
In particular, if~$h$ is of relative dimension~0, then we have
\[ h_{L}^{-1}(\sigma_{L}(y)) = \{\sigma_{L}(x')  :  x' \in h^{-1}(y)  \}. \] 
If~$h$ is finite and flat at~$x$, then we have $\deg_{\sigma_{L}(x)}(h_{L}) = \deg_{x}(h)$.
\end{lemma}
\begin{proof}
By~\cite[Lemma~2.9]{PP}, we have $h_{L}(\sigma_{L}(x)) = \sigma_{L}(h(x))$. Let us now assume that~$h$ is of relative dimension~0.

Assume that~$x$ is of type~1. Then it is a $k$-rational point and the statement concerning the fiber follows from the fact that $\sH(x)=k$ is algebraically closed. The one about the degree follows from the fact that, in this case, it is equal to the ramification index, which is invariant by extension of scalars.

Assume that~$x$ is not of type~1. In order to prove the statements, we may localize around~$x$, hence assume that $h^{-1}(y) = \{x\}$. 

Assume, moreover, that~$h$ is finite at~$x$. Note that it is always flat at~$x$. Set $d:=\deg(h)$. We have $\deg_{x}(h) = d = \deg(h_{L})$, hence $\deg_{\sigma_{L}(x)}(h_{L}) \le d$.

By Theorem~\ref{thm:weaklystable}, we have
\[\deg_{x}(h) = [\cO_{X,x} \colon \cO_{Y,y}] = [\sH(x) \colon \sH(y)].\]
Let $(e_{1},\dotsc,e_{d})$ be a basis of~$\sH(x)$ over~$\sH(y)$. It induces a basis of $\sH(x)\ho_{k} L$ as a module over $\sH(y)\ho_{k} L$, hence a linearly independent family in~$\sH(\sigma_{L}(x))$ over $\sH(\sigma_{L}(y))$ by localizing and completing. It follows that $\deg_{\sigma_{L}(x)}(h_{L}) \ge d$, hence $\deg_{\sigma_{L}(x)}(h_{L}) = d$, as required. The statement about the fiber follows too.

Let us finally handle the case where $\dim_{x}(h) = 0$ without the assumption that~$h$ be finite at~$x$. By the Berkovich analytic version of Zariski's Main Theorem (see \cite[Th\'eor\`eme~3.2]{Ducros_Dimension}), there exist an affinoid neighborhood~$V$ of~$x$ in~$X$, an affinoid domain~$W$ of a $k$-analytic curve~$T$ such that the restriction of~$h$ to~$V$ factors as 
\[ V \to W \hookrightarrow T \to Y, \]
where $V \to W$ is finite and $T\to Y$ finite \'etale. We deduce that $h_{L}(\sigma_{L}(x)) = \sigma_{L}(h(x))$ since all the morphisms in the composition satisfy this property.
\end{proof}

\subsection{Radiality results}

\begin{lemma}\label{lem:radialimage} Let~$X$ and $Y$ be strictly $k$-analytic curves and $\cS, \cS'$ be~$\cL$-facades of~$X$ and $Y$ respectively. Let~$h\colon X\to Y$ be a morphism satisfying the hypothesis of Theorem \ref{thm:main}. Then 
\begin{enumerate}[$i)$]
\item if $A\subseteq X$ is a radial set so is $h(A)$ and  
\item if $A\subseteq Y$ is a radial set so is $h^{-1}(A)$. 
\end{enumerate}
\end{lemma} 

\begin{proof} Let us show $(i)$. By Theorem \ref{thm:radialdef}, $\varphi_\cS(A^\D)$ is $\cL$-definable and hence also $h_{\cS\cS'}(\varphi_{\cS}(A^\D))$, by Theorem~\ref{thm:main}. By Theorem \ref{thm:radialdef}, there is $B\in \Rad_k(X)$ such that $\varphi_{\cS'}(B^\D)=h_{\cS\cS'}(\varphi_\cS(A^\D))$. By Theorem~\ref{thm:radialdef} again, we have $B^\D=h(A^\D)$. We claim that $B=h(A)$. Let $\mathbf{L}$ be an elementary extension of $\mathbf{k}$ that is maximally complete, with value group $\RR_{>0}$. By Corollary~\ref{cor:radial_commute}, we have
\[
\delta_L(B_L) =  (h_{\cS\cS'}(\varphi_\cS(A^\D)))(\mathbf{L})=h_{L, \cS_L\cS_L'}(\varphi_\cS(A^\D)(\mathbf{L})) = \varphi_{\cS_L'}(h_{L}(A_L)),
\]
and taking $\delta_L^{-1}$ on both sides we get $B_L=h_L(A_L)$. By Remark~\ref{rem:radial} ii), we have $\sigma_L^{-1}(B_L)=B$ and Lemma~\ref{lem:degreesigma} now ensures that we have $B=\sigma_L^{-1}(h_L(A_L))=h(A)$. 

The proof of~ii) is analogous noting that, by Lemma~\ref{lem:degreesigma}, $\sigma_L^{-1}(h_L^{-1}(A_L))=h^{-1}(A)$. 
\end{proof} 

\begin{proposition}\label{prop:radialdefinable} Let~$h\colon X\to Y$ be a morphism of strictly~$k$-affinoid curves of relative dimension~0. For each $d\in \NN$, the set 
\[
M_{h,d}:=\{x\in X :  h^{-1}(h(x)) \text{ has cardinality $d$}\}
\]
is radial. 
\end{proposition}

\begin{proof} By Theorem \ref{thm:facade_existence}, there exist $\cL_\BB^\an$-facades $\cS$ and $\cS'$ be  of $X$ and $Y$ respectively. By Lemma~\ref{lem:morphism_definable} and Theorem~\ref{thm:main}, the map $h_{\cS\cS'}$ is $\cL_\BB^\an$-definable, hence the set 
\[
M_{h,d}^{\cS}:=\{x\in X^{\cS}  :  h_{\cS\cS'}^{-1}(h_{\cS\cS'}(x)) \text{ has cardinality $d$}\}
\]
is $\cL_\BB^\an$-definable over $k$. Furthermore, we have $\varphi_\cS(M_{h,d}^\D)=M_{h,d}^{\cS}$. By Theorem \ref{thm:radialdef}, there exists $B\in \Rad_k(X)$ such that $\varphi_{\cS}(B^\D)=M_{h,d}^{\cS}$. We claim that $B=M_{h,d}$. Let $\mathbf{L}$ be a rank~1 elementary extension of $\mathbf{k}$ that is maximally complete with value group $\RR_{>0}$. By Corollary \ref{cor:radial_commute}, we have
\[
\varphi_{\cS_L}(B_L)=\{x\in X_L^{\cS_L}  :  h_{L,\cS_L\cS_L'}^{-1}(h_{L,\cS_L\cS_L'}(x)) \text{ has cardinality $d$}\}.
\]
Since $X_L=X_L^\D$, $\varphi_{\cS_L}$ defines a bijection between $X_L$ and $X_{L}^{\cS_L}$, hence  
\[
B_L=\{x\in X_L  :  h_L^{-1}(h_{L}(x)) \text{ has cardinality $d$}\}.
\]
By Remark~\ref{rem:radial} ii), we have $B=\sigma_L^{-1}(B_L)$ and it now follows from Lemma~\ref{lem:degreesigma} that $B=M_{h,d}$. 
\end{proof}

\begin{corollary}\label{cor:localsize}
Let~$h\colon X\to Y$ be a morphism of strictly $k$-analytic curves of relative dimension~0 that is topologically proper. Then, for each $d\in \NN$, the set 
\[
M_{h,d}:=\{x\in X :  h^{-1}(h(x)) \text{ has cardinality $d$}\}
\]
is radial. 

If~$h$ is a compactifiable morphism of nice curves (\emph{e.g.} the analytification of an algebraic morphism of curves), then, for each $d\in \NN$, the set~$M_{h,d}$ is radial with respect to a finite triangulation of~$X$.
\end{corollary}
\begin{proof}
Let $x\in X$. Let~$V$ be a strictly $k$-affinoid neighborhood of~$h(x)$ in~$Y$. Since the fiber~$h^{-1}(h(x))$ is finite, it admits a strictly $k$-affinoid neighborhood~$U$ in $h^{-1}(V)$. By topological properness, there exists an affinoid neighborhood~$W$ of~$h(x)$ in~$V$ such that $h^{-1}(W) \subseteq U$. In this situation, we have
\[
\{x\in h^{-1}(W) :  h^{-1}(h(x)) \text{ has cardinality $d$}\} = \{x\in h^{-1}(W) :  h_{|U}^{-1}(h(x)) \text{ has cardinality $d$}\}, 
\]
which is a radial subset of~$h^{-1}(W)$ by Proposition~\ref{prop:radialdefinable} and Lemma~\ref{lem:radialsubcurve}. The result now follows from Lemma~\ref{lem:radiallocal}.

The last statement follows from the first.
\end{proof}

\begin{remark}
In the situation of Corollary~\ref{cor:localsize}, one can obtain similar radiality results for the set $\{y\in Y :  h^{-1}(y) \text{ has cardinality $d$}\}$.
\end{remark}

We now prove similar results for the local degree function. Let us first extend slightly the usual definition. 

\begin{definition}
Let $h \colon X \to Y$ be a flat morphism of $k$-analytic curves of relative dimension~0 and let $x\in X$. 

If $x\in \mathrm{Int}(X/Y)$ (for instance if $x$ is of type~1 or~4), then $h$~is finite at~$x$ and we denote by $\deg_{h}(x)$ the usual degree at~$x$, \emph{i.e.} the rank of the free $\cO_{Y,h(x)}$-module $\cO_{X,x}$.

If $x\in \partial(X/Y)$ (which implies that~$x$ is of type~2 or~3), then we set 
\[\deg_{h}(x) := [\sH(x) \colon \sH(h(x))].\]
\end{definition}

\begin{remark}\label{rem:degreerestriction}
Let $h \colon X \to Y$ be a flat morphism of $k$-analytic curves of relative dimension~0. It follows from Theorem~\ref{thm:weaklystable} that, for each analytic domain~$V$ of~$X$ and each point $x\in V$, we have
\[\deg_{h_{|V}}(x) = \deg_{h}(x).\]
Note also that the local degree behaves multiplicatively under composition of maps.
\end{remark}

By reducing to the finite Galois case, we will now deduce from Proposition~\ref{prop:radialdefinable} that, given a flat morphism of strictly $k$-analytic curves $h \colon X \to Y$ of relative dimension~0, the set of points of~$X$ with a given prescribed multiplicity is radial. This result was first proven by Temkin in the case where the curves are quasi-smooth and the morphism is finite (which is the crucial case). Under those assumptions, Temkin actually proves in~\cite[Theorem~3.4.11]{temkin_2017_metric} that the morphism admits a so-called ``radializing skeleton'', but this is equivalent to our statement by \cite[Theorem~3.3.10]{temkin_2017_metric}, whose proof is not difficult.

\begin{theorem}\label{thm:radiality_curves}
Let $h\colon X \to Y$ be a flat morphism of strictly $k$-analytic curves of relative dimension~0. Then, for each $d\in \NN$, the set
\[N_{h,d} := \{x \in X  :  \deg_{x}(h) = d\}\]
is radial.

If~$h$ is a compactifiable morphism of nice curves (\emph{e.g.} the analytification of an algebraic morphisms of curves), then the set~$N_{h,d}$ is radial with respect to a finite triangulation of~$X$.
\end{theorem}
\begin{proof}
We will prove the result in several steps.

\smallbreak

\textbf{Step 1:} \textit{$X$ and~$Y$ are strictly $k$-affinoid spaces and $h$ is a ramified Galois covering}

We may assume that~$X$ and~$Y$ are connected. In this case, for each $x\in X$, we have $\deg_{x}(h) = \deg(h)/\# h^{-1}(h(x))$ and the result follows from Proposition~\ref{prop:radialdefinable}. 

\smallbreak

\textbf{Step 2:} \textit{$X$ and~$Y$ are strictly $k$-affinoid spaces and $h$ is finite and flat}

By~\cite[Lemma~3.1.4]{temkin_2017_metric}, $h$ factors as the composition of a generically \'etale morphism and a purely inseparable morphism on which the degree is constant. As a consequence, we may assume that~$h$ is generically \'etale. By~\cite[Lemma~3.1.2]{temkin_2017_metric}, there exist a strictly $k$-affinoid space~$X'$ and a ramified Galois covering~$p$ such that $h' := h \circ p \colon X' \to Y$ is a ramified Galois covering. For $d\in \N$, set 
\[Q_{d} := \{x' \in X'  :  \deg_{x'}(h') = d \deg_{x'}(p)\}.\]
By Step~1 applied to~$h'$ and~$p$, this is a radial set. Since $p(Q_{d}) = N_{h,d}$, the result now follows from Lemma~\ref{lem:radialimage}.

\smallbreak

\textbf{Step 3:} \textit{The general case}

By Lemma~\ref{lem:radiallocal}, it is enough to prove that the result holds in the neighborhood of each point of~$x$. By Zariski's Main Theorem (see \cite[Th\'eor\`eme~3.2]{Ducros_Dimension}), there exist an affinoid neighborhood~$V$ of~$x$ in~$X$, an affinoid domain~$W$ of a $k$-analytic curve~$T$ such that the restriction of~$h$ to~$V$ factors as 
\[ V \to W \hookrightarrow T \to Y, \]
where $f \colon V \to W$ is finite and $g \colon T\to Y$ finite \'etale. Up to shrinking the spaces, we may assume that they are all strictly $k$-affinoid.

Let us denote by $n \colon Y' \to Y$ the normalization morphism of~$Y$ and consider the Cartesian diagram 
\[\begin{tikzcd}
V' := V \times_{k} Y' \arrow{r}{n'}\arrow{d}{h'} & V\arrow{d}{h}\\
Y' \arrow{r}{n}& Y
\end{tikzcd}
\texto{.}{-1.5pc}
\]
Let $d\in \NN$. Since the degree of a morphism at a point is preserved under base-change,  we have 
\[\{ x\in V \,\colon\, \deg_{x}(h) = d\} = n'(\{ x'\in V' \,\colon\, \deg_{x'}(h') = d\}).\]
By Lemma~\ref{lem:radialimage}, it is enough to prove that the set $\{ x'\in V' \,\colon\, \deg_{x'}(h') = d\}$ is radial, hence, up to replacing~$Y$, $V$, $W$ and~$T$ by their base-changes to~$Y'$ over~$Y$, we may assume that~$Y$ is normal. Note that this does not change the properties of the morphisms in the factorization of~$h$ above. Moreover, since~$Y$ is normal, $T$ and $W$ are normal too. In particular, the morphism $f \colon V \to W$ being finite, it is flat.

For $a\in \N$, set $V_{a} := \{ x\in V \,\colon\, \deg_{x}(f) = a\}$ and $T_{a} := \{ x\in T \,\colon\, \deg_{x}(g) = a\}$. Those sets are radial by Step~2. It now follows from Lemmas~\ref{lem:radialboolean}, \ref{lem:radialsubcurve} and~\ref{lem:radialimage} and Remark~\ref{rem:degreerestriction} that the set
\[\{x\in V \,\colon \, \deg_{x}(h) = d\} = \bigcup_{a+b = d} V_{a} \cap (f^{-1}(T_{b} \cap W))\]
is radial.

\smallbreak

\textbf{Step 4:} \textit{Refinement in the compactifiable case}

The case of a compactifiable morphism reduces to the general case with~$X$ and~$Y$ compact thanks to Remark~\ref{rem:flat}.
\end{proof}

\section{Comparison with Hrushovski-Loeser's spaces}\label{sec:HL}

In this section we establish the link with Hrushovski-Loeser theory of stably dominated types. For the reader's convenience we briefly recall some definitions although some familiarity with model theory will be assumed.  

\subsection{Brief review on Hrushovski-Loeser space $\widehat{X}$}

We work over a very large elementary extension $\cKK$ of $\mathbf{k}$ which should be thought as a universal domain in the sense of Weil. This is what model theorists called a \emph{monster model}. Any set of smaller cardinality than $\cKK$ will be called \emph{small} (so $k$ is small). In particular, we assume that every small valued field in consideration is embedded in $\cKK$. For a tuple of variables $x$, a \emph{global type} $p(x)$ is a type over $\cKK$. Let $\cL$ denote some expansion of $\cL_3$.  

\begin{definition} Let $C\subset\cKK$ be a small set and $p(x)$ be a global type. The type $p$ is \emph{$C$-definable} if for every $\cL$-formula $\varphi(x,y)$ (without parameters), there is an $\cL(C)$-formula $\psi_\varphi(y)$ (with parameters in $C$) such that for every $a\in \cKK^{|y|}$
\[
\varphi(x,a)\in p \Leftrightarrow \cKK\models \psi_\varphi(a)
\]
\end{definition} 

The map $\varphi(x,y)\mapsto \psi_\varphi(y)$ is called a \emph{scheme of definition for $p$}. If $p(x)$ is $C$-definable and $\cKK\subseteq A$, then we can extend $p(x)$ to a type over $A$, which we denote $p|A$, by following the rule given by the scheme of definition over a larger elementary extension $\cKK'$ of $\cKK$ containing $A$: 
\[
p|A:=\{\varphi(x,a) :  \cKK'\models\psi_\varphi(a), a\in A\}. 
\] 
For $C$-definable global types $p(x)$ and $q(y)$, the tensor product $p(x)\otimes q(y)$ is defined as follows: 
\[
p(x)\otimes q(y) = tp(a,b/\cKK), \textrm{ where } b\models q(y) \text{ and } a\models p|\cKK\cup\{b\}. 
\]
\begin{definition} A global $C$-definable type is \emph{generically stable} if 
\[
p(x)\otimes p(y) = p(y)\otimes p(x). 
\]
Given a definable set $X$ in $k$, we let $\widehat{X}(C)$ denote the set of $C$-definable generically stable types on $X$, that is, those types containing a formula which defines $X$.  
\end{definition} 

\begin{remark} There is a more general notion of generically stable type which coincides with the previous definition in NIP theories. Since ACVF is NIP, we will abuse terminology and use this restricted version hereafter. To know more about the general notion we refer the reader to \cite{simon}. 
\end{remark}

Note that the $\widehat{\cdot}$ operator is functorial in the following sense: if $X, Y$ are $C$-definable sets and $f\colon X\to Y$ is a $C$-definable function, then there is a function $\widehat{f}\colon\widehat{X}(C)\to \widehat{Y}(C)$ which extends $f$ defined by 
\[
\widehat{f}(p) = tp(f(a)/\cKK) \text{ for any $a\models p$}.
\]

\begin{theorem}\label{lem:domination} Let $X,Y$ be definable sets and $f\colon X\to Y$ be a definable function. Suppose everything is definable over $C$ and let $C\subseteq A$ be a small set. If $f$ is injective (resp. surjective) then $\widehat{f}\colon \widehat{X}(A)\to\widehat{Y}(A)$ is also injective (resp. surjective). 
\end{theorem}
\begin{proof}
Injectivity is easy to check. For surjectivity see \cite[Lemma 4.2.6]{HL} or \cite{johnson}. 
\end{proof}

In what follows we will work with $\widehat{X}(k)$ which for simplicity will be denoted simply by $\widehat{X}$ if no confusion arises. Note that $X(k)$ embeds in $\widehat{X}$ by sending $a\in X(k)$ to its type over $\cKK$. These types are called \emph{realized types over $k$}. 

Let us now provide the background for pro-definability. Let $(I,\leqslant)$ be a small directed partially ordered set. A \emph{$C$-definable directed system} is a collection $(X_i, f_{ij})$ such that 
\begin{enumerate}[$i)$]
\item for every $i\in I$, $X_i$ is a $C$-definable set; 
\item for every $i, j\in I$ such that $i\geqslant j$, $f_{ij}\colon X_i\to X_j$ is $C$-definable 
\item $f_{ii}$ is the identity on $X_i$ and $f_{ik} = f_{jk}\circ f_{ij}$ for all $i \geqslant j \geqslant k$.
\end{enumerate}  

A \emph{pro-$C$-definable} $X$ is the projective limit of a $C$-definable directed system $(X_i, f_{ij})$
\[
X:=\varprojlim_{i\in I} X_i.
\]
We say that $X$ is pro-definable if it is pro-$C$-definable for some small set of parameters $C$. 

For $X=\varprojlim_{i\in I} X_i$ and $Y=\varprojlim_{j\in J} Y_j$ two pro-$C$-definable sets with associated $C$-definable directed systems $(X_i,f_{ii'})$ and $(Y_j, g_{jj'})$, a \emph{pro-$C$-definable morphism} $\varphi\colon X\to Y$ is given by a function $n\colon J \to I$ and a family of $C$-definable functions $\{\varphi_{ij}\colon X_i\to Y_j  :  i\ge n(j)\}$ such that, for all $j\geqslant j'$ in $J$ and all $i \geqslant i'$ in $I$ with $i\ge n(j)$ and $i'\ge n(j')$, the following diagram commutes 
\[
\begin{tikzcd}
X_{i} \ar{d}{f_{ii'}} \ar{r}{\varphi_{ij}} & Y_{j} \ar{d}{g_{jj'}} \\
X_{i'} \ar{r}{\varphi_{i'j'}} & Y_{j'}
\end{tikzcd}
\texto{,}{-1.5pc}
\]
We denote by $\pi_i$ the canonical projection $\pi_i\colon X\to X_i$. 
\begin{enumerate}[$i)$]
\item $X$ is \emph{strict pro-definable} if it can be represented by a directed system $(X_i,f_{ij})$ where the transition maps $f_{ij}$ are surjective.
 \item $X$ is \emph{iso-definable} if it can be represented by a directed system strict $(X_i,f_{ij})$ where the transition maps $f_{ij}$ are bijections. Equivalently, $\widehat{X}$ is in pro-definable bijection with a definable set (see \cite[Corollary 2.2.4]{HL}). 
 \item If $X=\varprojlim_{i\in I} X_i$, a subset $Z\subseteq X$ is called \emph{relatively definable} if there is $i\in I$ and a definable subset $Z_i\subseteq X_i$ such that $Z=\pi_i^{-1}(Z_i)$. 
\end{enumerate}  

\begin{theorem}\label{thm:pro-def} Let $T$ be a complete NIP $\cL$-theory with elimination of imaginaries and $X$ be a definable set. Then for any small set of parameters, $\widehat{X}(C)$ is pro-definable in $\cL$. If $T$ is $Th(k,\cL_\mathcal{G})$ or $Th(k,(\cL_3^\an)^{\mathrm{eq}})$, then $\widehat{X}$ is in addition strict pro-definable. 
\end{theorem}

\begin{proof} This corresponds to \cite[Theorem 3.1.1]{HL}. For $Th(k,(\cL_3^\an)^{\mathrm{eq}})$ details can be found in \cite[Chapter 7]{johnson}. 
\end{proof}

Among other things, the pro-definability of $\widehat{X}$ allows us to have a natural notion of definable subsets of $\widehat{X}$, namely relatively definable subsets. Also, note that given a definable map $f\colon X\to Y$, the function $\widehat{f}$ is pro-definable. In the case where $X$ is a curve, Hrushovski and Loeser proved the following stronger result.

\begin{theorem}[{\cite[Theorem 7.1.1]{HL}}]\label{thm:iso-def} Let $X$ be an algebraic curve over $k$. Then $\widehat{X}$ is iso-definable in $\cL_\mathcal{G}$. 
\end{theorem}

The previous theorem states that, when $\widehat{X}=\varprojlim_{i\in I} X_i$, we can identify $\widehat{X}$ with some (any) definable set $X_i$ for $i$ large enough. Moreover, relative definable subsets of $\widehat{X}$ are in definable bijection with definable subsets of $X_i$ for $i$ large enough. The proof of Theorem \ref{thm:iso-def} is non-constructive and relies both on the elimination of imaginaries assumption from Theorem \ref{thm:pro-def} and Riemann-Roch's theorem. Although we know that $\widehat{X}$ is in pro-definable bijection with an $\cL_\mathcal{G}$-definable set, we cannot point at any such definable set. In the following subsection we show how to use the results in Sections \ref{sec:bricks} and \ref{sec:functorial} to provide a constructive proof of Theorem \ref{thm:iso-def}. In addition, our proof makes no additional use of Riemann-Roch (besides the potential uses contained in the existence of triangulations, Theorem \ref{thm:semistable}). 

\subsubsection{Pro-definable bijection between $\widehat{X}(k)$ and $X^\cS$}

Recall that, in what follows, $\widehat{X}$ denotes $\widehat{X}(k)$ for any definable set~$X$ over $k$. 

As observed in \cite[Section 3.2]{HL}, $\widehat{k}$ is in pro-definable bijection with $\BB$ (or, to be pedantically precise, $\widehat{\mathbf{VF}}(k)$ is in pro-definable bijection with $\BB(\mathbf{k})$). In what follows we will thus respectively identify $\widehat{k}$ with $\BB$ and $\widehat{D(0,1)}$ with $\DD_k^\D$. The goal of this section is to show that, given an $\cL_\BB$-facade $\cS$ of $X$, there is a pro-definable bijection between $\widehat{X}$ and $X^\cS$. 

\begin{theorem}\label{thm:prodef_facade} Let $X$ be a $k$-analytic curve and~$\cS$ be an $\cL$-facade of~$X$ (where $\cL$ denotes either $\cL_\BB$ or $\cL_\BB^\an$). Then there is a pro-definable bijection in $\cL^{\mathrm{eq}}$ between $\widehat{X}(k)$ and $X^\cS$. If $\cL$ is $\cL_\BB$, the pro-definable bijection can be taken over $\cL_\mathcal{G}$.
\end{theorem}

\begin{proof} Recall that the data of the facade $\cS$ induces the following definable partition of $X(k)$:  
\[
X(k):=\bigsqcup_{\eta_{a,0}\in S^{(1)}} \{a\} \sqcup \bigsqcup_{I\in E} V_I(k)\sqcup \bigsqcup_{x\in S^{(2)}} \left(W_x(k) \sqcup \bigsqcup_{i=1}^{m(x)} V_{x,i}(k) \right).  
\]
Let us show that
\begin{enumerate}
\item[$\bullet$] $\widehat{\{a\}}$ is in pro-definable bijection with $\eta_{1,0}$, 
\item[$\bullet$] $\widehat{V_I(k)}$ is in pro-definable bijection with $f_I(V_I)^{\Def}$ (which is a relative definable subset of $\BB$), 
\item[$\bullet$] $\widehat{W_x(k)}$ is in pro-definable bijection with $Z_x^\cS\cup\eta_{0,1}$, 
\item[$\bullet$] $\widehat{V_{x,i}(k)}$ is in pro-definable bijection with $f_{x,i}(V_I)^{\Def}$,   
\end{enumerate}
which will establish that $\widehat{X}$ is in pro-definable bijection with $X^\cS$. 

\textbf{Case 1:} We have that $\widehat{\{a\}}=\{a\}$ which is trivially in pro-definable bijection with $\eta_{0,1}$. 
 
\

\textbf{Case 2:} We have a definable morphism $f_I\colon X(k)\to k$ which restricted to $V_I(k)$ is a bijection onto its image. Then, by Lemma \ref{lem:domination}, $\widehat{f_I}\colon \widehat{V_I(k)}\to \widehat{f_I(V_I(k))}$ is a bijection.  

\

\textbf{Case 3:} We have a bijective $\cL$-definable map $\epsilon_x\colon W_x(k)\to Z_x^\cS(k)$ given by $\rho_{W_x}\times f_x$ (where $Z_x^\cS(k)$ is as in \eqref{eq:Z(k)}). By Lemma \ref{lem:domination}, the map $\widehat{\epsilon_x} \colon \widehat{W_x(k)} \to \widehat{Z_x^\cS(k)}$ is a bijection. Since $Z_x^\cS(k)\subseteq \cU_x(k)\times D(0,1)$, $\widehat{Z_x^\cS(k)}\subseteq (\cU_x(k)\times D(0,1))^{\widehat{}}$. In general, the stable completion of a product is not the product of the stable completions so, $(\cU_x(k)\times D(0,1))^{\widehat{}}$ is not equal to  $\widehat{\cU_x(k)}\times \widehat{D(0,1)}$. Nevertheless, in our situation this will almost be the case as we now show. Let 
\[
\pi\colon (\cU_x(k)\times D(0,1))^{\widehat{}}\to \widehat{\cU_x(k)}\times \widehat{D(0,1)}
\] 
be the pro-definable map given by 
\[
tp(\alpha,a/L) \mapsto (tp(\alpha/L), tp(a/L)). 
\]

\begin{claim} The map $\pi\circ\widehat{\epsilon_x}$ is injective. 
\end{claim}

Let $r$ and $r'$ be two elements in $\widehat{W_x(k)}$ such that $\pi(\widehat{\epsilon_x}(r)) = \pi(\widehat{\epsilon_x}(r'))$. Let $p(x,y):=tp(\alpha_0,a_0/\cKK)$ and $p'(x,y):=tp(\alpha_1,a_1/\cKK)$ be two elements in $(\cU_x(k)\times D(0,1))^{\widehat{}}$ such that $\widehat{\epsilon_x}(r)=p$, $\widehat{\epsilon_x}(r')=p'$. 
We have $tp(\alpha_0/\cKK)=tp(\alpha_1/\cKK)$ and $tp(a_0/\cKK)=tp(a_1/\cKK)$. Since $\widehat{\epsilon_x}$ is a bijection, it suffices to show that $p=p'$. We split in cases depending on whether $tp(\alpha_0/\cKK)$ is realized in $\cKK$ or corresponds to the generic type of $\cC_x$, which we denote by $\eta_{\cC_x}$. 

\

\emph{Case 3.1:} Suppose that $\alpha_0\in \cU_x(\tilde{L})$. Then $\alpha_0=\alpha_1$ and $tp(\alpha_0,a_0/\cKK)=tp(\alpha_0,a_1/\cKK)$ since $tp(a_0/\cKK)=tp(a_1/\cKK)$. 

\

\emph{Case 3.2:} Suppose that $\alpha_0\notin \cU_x(\tilde{L})$. Hence, $tp(\alpha_0/\cKK)$ (resp. $tp(\alpha_1/\cKK)$) is the generic type $\eta_{\cC_x}$ of $\cC_x$. Let $\cKK'$ be an elementary extension of $\cKK$ containing $\alpha_0,\alpha_1,a_0 $ and $a_1$. Suppose for a contradiction that $tp(\alpha_0,a_0/\cKK)\neq tp(\alpha_1,a_1/\cKK)$. Hence, there is a formula $\varphi(x,y)$ over $\cKK$ such that $\cKK'\models\varphi(\alpha_0,a_0)\wedge \neg\varphi(\alpha_1,a_1)$. By assumption, both $(\alpha_0,a_0)$ and $(\alpha_1,a_1)$ are contained in $Z_x^\cS(L')$. Therefore, we may suppose that $\varphi(x,y)$ defines in $\cKK'$ a set $H\subseteq Z_x^\cS(L')$. By Corollary \ref{cor:opendisc} (applied to $h=\tilde{f}_x$), there is a finite set $F\subseteq \cU_x(\tilde{L})$ such that either $H_\alpha$ is empty for all $\alpha\in \cU_x(\tilde{L})\setminus F$ or $H_\alpha=\res^{-1}(\tilde{f}_x(\alpha))$ for all $\alpha\in \cU_x(\tilde{L})\setminus F$. Clearly, the same statement holds for all $\alpha\in \cU_x(\tilde{L}')\setminus F$. By assumption, neither $\alpha_0$ not $\alpha_1$ is contained in $F$. Since $\varphi(\alpha_0,a_0)$ holds, $H_\alpha$ must be non-empty for all $\alpha\in \cU(\tilde{L}')\setminus F$. In particular, $a_0$ satisfies the formula over $\cKK$
\[
(\forall x)(x\notin F \rightarrow \varphi(x,a_0)). 
\]
But since $tp(a_0/\cKK)=tp(a_1/\cKK)$, $a_1$ must also satisfy such formula, which implies that $\varphi(\alpha_1, a_1)$ holds, a contradiction. This shows the claim.  

\

By the claim we have 
\[
\pi\circ\widehat{\epsilon_x}(\widehat{W_x}) =\{(tp(\alpha/\cKK),tp(a/\cKK))\in \widehat{\cU_x(\tilde{L})}\times \widehat{D(0,1)}  :  \tilde{f_x}(\alpha)=\res(a)\}.
\]

To conclude, it remains to note that $\widehat{\cU_x(\tilde{k})}=(\eta_{\cC_x}\sqcup \cU_x(\tilde{k}))$ and therefore 
\[
\pi\circ\widehat{\epsilon_x}(\widehat{W_x}) =\{(\eta_{\cC_x},\eta_{0,1})\}\cup \{(tp(\alpha/\cKK),tp(a/\cKK))\in \cU_x(\tilde{L})\times \widehat{D(0,1)}  :  \tilde{f_x}(\alpha)=\res(a)\}.
\]
The set $\{(\eta_{\cC_x},\eta_{0,1})\}$ is trivially in pro-definable bijection with $\{\eta_{0,1}\}$ and since $\widehat{D(0,1)}$ is in pro-definable bijection with $\DD_k^\D$, we obtain that $\pi\circ\widehat{\epsilon_x}(\widehat{W_x}) \setminus \{\eta_{\cC_x},\eta_{0,1}\}$ is in pro-definable bijection with 
\[
\{(\alpha,\eta_{a,r})\in \cU_x(L)\times \DD_k^\D  :  \tilde{f_x}(\alpha)=\red(a)\}=Z_x^\cS(\cKK).
\]

\

\textbf{Case 4:} Analogous to Case 2.  

\end{proof}

Given an $k$-algebraic curve $X$ and an $\cL$-facade $\cS$ of $X$, we let $\psi_\cS\colon \widehat{X(k)}\to X^\cS$ be the pro-definable bijection given by Theorem \ref{thm:prodef_facade}. The following gives us a different proof of Theorem \ref{thm:main} for the analytification of a morphism between algebraic curves:

\begin{corollary}\label{cor:main} Let $h\colon X\to Y$ be a morphism of algebraic curves over $k$. Let $\cS$ and $\cS'$ be $\cL_\BB$-facades of $X^\an$ and $Y^\an$ respectively (which exist by Theorem \ref{thm:facade_existence}). Then the map $h_{\cS\cS}'$ making the following diagram commute
\[
\begin{tikzcd}
\widehat{X(k)} \ar{d}{\psi_{\cS}} \ar{r}{\widehat{h}} & \widehat{Y(k)} \ar{d}{\psi_{\cS'}} \\
X^{\cS} \ar{r}{h_{\cS\cS'}'} & Y^{\cS} 
\end{tikzcd}
\texto{,}{-1.5pc}
\]
is $\cL_\BB$-definable. In addition, $h_{\cS\cS'}'=h_{\cS\cS'}$, where $h_{\cS\cS'}$ is the map given by Definition \ref{def:map}.   
\end{corollary} 

\begin{proof} By Theorem \ref{thm:prodef_facade}, $\psi_\cS$ and $\psi_{\cS'}$ are pro-definable bijections, hence $h_{\cS\cS}'\colon X^{\cS} \to Y^{\cS'}$ is pro-definable (in $\cL_{\mathcal{G}}$). Since a pro-definable map between definable sets is definable, $h_{\cS\cS'}'$ is $\cL_{\mathcal{G}}$-definable. Its domain and codomain are $\cL_\BB$-definable, which shows that $h_{\cS\cS'}'$ is $\cL_\BB$-definable. 

To show that $h_{\cS\cS'}'=h_{\cS\cS'}$, let $\mathbf{L}$ be a maximally complete elementary extension of $\mathbf{k}$ with value group $\RR$. Consider the following diagram
\[
\begin{tikzcd}
X^\cS(\mathbf{L})  \ar[equal]{r} &  X_L^{\cS_L} \ar{d}{h_{L\cS_L\cS_L'}} & \ar{l}{\varphi_{\cS_L}}	X_L^\an \ar{d}{h_L} \ar{r}{\beta_X} &	\widehat{X(k)}(L) \ar{d}{\widehat{h}} \ar{r}{\psi_{\cS_L}} & X^\cS(\mathbf{L}) \ar{d}{h_{\cS\cS'}'}\\
Y^\cS(\mathbf{L}) \ar[equal]{r} &  Y_L^{\cS_L'} & \ar{l}{\varphi_{\cS_L'}}	Y_L^\an \ar{r}{\beta_Y} &	\widehat{Y(k)}(L) \ar{r}{\psi_{\cS_L'}} & Y^{\cS'}(\mathbf{L})
\end{tikzcd}
\texto{,}{-1.5pc}
\]
where $\beta_X$ and $\beta_Y$ are bijections given by \cite[Lemma 14.1.1]{HL} (noted by $\pi_X$ and $\pi_Y$ in \cite{HL}) and the equalities are given by Corollary \ref{cor:radial_commute}. By \cite[Proposition 14.1.3]{HL} and Theorem \ref{thm:main} every square in the diagram is commutative and therefore $h_{L\cS_L\cS_L'}= h_{\cS\cS_L'}'$. Now by Lemma \ref{lem:degreesigma} and Corollary \ref{cor:radial_commute} the restriction of $h_{L\cS_L\cS_L'}'$ to $X^\cS$ (\emph{i.e.} to $X^\cS(\mathbf{k})$) is $h_{\cS\cS}$, therefore $h_{\cS\cS'}'=h_{\cS\cS'}$. 
\end{proof}

We finish with two remarks.

\begin{remark} 
Let $X$ be a $k$-algebraic curve and $\cS$ be an $\cL$-facade of $X$. Using Theorem \ref{thm:prodef_facade}, Hrushovski-Loeser's main theorem (or the curve version \cite[Theorem 7.5.1]{HL}) translates into the existence of a definable deformation retraction for $X^\cS$ (since pro-definable maps between definable sets are definable) onto a $\Gamma$-internal subset of $X^\cS$ (for the definition of $\Gamma$-internal, see \cite[Section~2.8]{HL}). However, as the $\cL$-facade~$\cS$ already carries all needed information to define such a map, it is not difficult to obtain this result early in our theory (see Remark \ref{rem:deformation}). 
\end{remark}

\begin{remark} Let $h\colon X\to Y$ be a morphism of $k$-analytic curves, $\cS$ and $\cS'$ be $\cL_\BB^\an$-facades of $X$ and $Y$ respectively and suppose that the restriction $h\colon X(k)\to Y(k)$ is $\cL_\BB^\an$-definable. In order to extend Corollary \ref{cor:main} to this setting, one needs to show the existence of bijections $\beta_X$ and $\beta_Y$ which make the diagram commute. 

Assuming such bijections exist and the result also holds in this more general setting, it would seem that this proof removes the $h$-compatibility assumption (condition~$(i)$) in Theorem \ref{thm:main} . Nevertheless, Lemma \ref{lem:morphism_definable} indicates that such a hypothesis is always verified in the cases we are considering.  
\end{remark}

\end{document}

%% file: commands.tex
\declaretheorem[style=definition,qed=$\dashv$,numberwithin=section]{definition}
\declaretheorem[style=definition,qed=,sibling=definition]{remark}
\declaretheorem[style=definition,qed=$\dashv$, sibling = definition]{lemma-definition}

\declaretheorem[style=theorem, sibling = definition]{theorem}
\declaretheorem[style=theorem, sibling = definition]{lemma}
\declaretheorem[style=theorem, sibling = definition]{proposition}
\declaretheorem[style=theorem, sibling = definition]{claim}

\declaretheorem[style=theorem, sibling = definition]{corollary}
\declaretheorem[style=theorem, sibling = definition]{question}
\declaretheorem[style=definition, sibling = definition]{example}
\declaretheorem[style=definition, sibling = definition]{convention}

\newcommand{\Q}{\mathbb{Q}}
\newcommand{\QQ}{\Q}

\renewcommand{\AA}{\mathbb{A}}
\newcommand{\Z}{\mathbb{Z}}
\newcommand{\ZZ}{\Z}
\newcommand{\N}{\mathbb{N}}
\newcommand{\NN}{\mathbb{N}}

\newcommand{\R}{\mathbb{R}}
\newcommand{\RR}{\R}

\newcommand{\PP}{\mathbb{P}}

\newcommand{\cM}{\mathcal{M}}
\newcommand{\cO}{\mathcal{O}}

\newcommand{\Def}{{(1,2)}}

\newcommand{\texto}[2]{\xymatrix{{}\save[]+<0pc,{#2}> *\txt{#1} \restore}}

\def\cL{{\mathcal L}}
\def\cA{{\mathcal A}}

\def\cC{{\mathcal C}}

\def\BB{{\mathbb B}}
\def\bb{{\mathbf b}}

\def\cB{{\mathcal B}}

\newcommand\cD{\mathcal{D}}
\newcommand\cE{\mathcal{E}}
\newcommand\cF{\mathcal{F}}

\newcommand\cS{\mathcal{S}}
\newcommand\cT{\mathcal{T}}
\newcommand\cU{\mathcal{U}}
\newcommand\cV{\mathcal{V}}
\newcommand\cW{\mathcal{W}}
\newcommand\cKK{\mathbf{L}}

\newcommand{\sH}{\mathscr{H}}
\newcommand{\sM}{\mathscr{M}}

\newcommand{\DD}{\mathbb{D}}

\newcommand\wc{{\mkern 2mu\cdot\mkern 2mu}}
\newcommand\va{|\wc|}
\newcommand\nm{\|\wc\|}
\newcommand{\an}{\mathrm{an}}

\newcommand{\red}{\mathrm{red}}
\renewcommand{\sp}{\mathrm{sp}}
\newcommand{\Rad}{\mathrm{Rad}}
\newcommand{\DDD}{\mathrm{Def}}
\newcommand{\simto}{\xrightarrow{\sim}}

\newcommand{\res}{\mathrm{res}}
\newcommand{\id}{\mathrm{id}}
\newcommand{\D}{{(1,2)}}
\newcommand{\cLr}{\mathcal{L}_\mathrm{ring}}
\newcommand\eps{\varepsilon}

\newcommand{\wt}{\widetilde}
\newcommand{\ov}{\overline}

\DeclareMathOperator{\ho}{\hat{\otimes}}
\DeclareMathOperator{\card}{card}

\renewcommand{\le}{\leqslant}
\renewcommand{\ge}{\geqslant}
\renewcommand{\epsilon}{\varepsilon}